\documentclass{article}

\usepackage[table]{xcolor}
\usepackage[margin=1in]{geometry}
\usepackage{amsthm}
\usepackage{amsmath}
\usepackage{graphicx}
\usepackage{float}
\usepackage{subfig}
\usepackage{tikz}
\usepackage{amssymb}
\usepackage{enumerate}
\usepackage{enumitem}
\usepackage{algorithm}
\usepackage{algpseudocode}
\usepackage{amsmath}
\usepackage{multirow}
\usepackage[colorlinks,
          linkcolor = teal,		
          anchorcolor = blue,	
          urlcolor = blue,		
          citecolor = teal,
           ]{hyperref}

\usepackage{appendix}
\usepackage{eqparbox}
\newtheorem{definition}{Definition}[section]
\newtheorem{lemma}{Lemma}[section]
\newtheorem{remark}{Remark}[section]

\newtheorem{assumption}{Assumption}
\newtheorem{theorem}{Theorem}[section]

\newcommand\keywords[1]{\textbf{Keywords}: #1}
\newcommand\msc[1]{\textbf{Mathematical Subject Classifications}: #1}

\usepackage[
    backend=biber,
    style=alphabetic,
    doi=false,
    url=false,
    isbn=false
]{biblatex}

\usepackage{pifont}

\definecolor{DarkBlue}{RGB}{0,60,137}

\definecolor{redd}{RGB}{179,30,0}

\definecolor{gre}{RGB}{0,153,0}

\usepackage{tcolorbox}
\newtcolorbox{problock}{
  colframe=black,
  colback=gray!2,
  arc=5pt,
  left=10pt,
  right=0pt,
  top=5pt,
  bottom=5pt,
  boxrule=1pt
}

\usepackage[framemethod=TikZ]{mdframed}

\mdfdefinestyle{myframe}{
    linecolor=black,
    linewidth=2pt,
    roundcorner=5pt,
    backgroundcolor=gray!10,
    innertopmargin=1pt,
    innerbottommargin=2pt,
    innerleftmargin=2pt,
    innerrightmargin=2pt,
}

\def\Halmos{\mbox{\quad$\square$}}

\DeclareUnicodeCharacter{0300}{\`}
\begin{document}
\title{On the Convergence of Constrained Gradient Method}

\author{Danqing Zhou\footnotemark[1], Hongmei Chen\footnotemark[2], Shiqian Ma\footnotemark[3], Junfeng Yang\footnotemark[4]}

\footnotetext[1]{School of Mathematics, Nanjing Audit University, Nanjing, P. R. China. Email: dqkyo@foxmail.com.}
\footnotetext[2]{School of Mathematical Sciences, Sichuan Normal University, Chengdu, P. R. China.
Email: hmchen@sicnu.edu.cn.}
\footnotetext[3]{Data Science and AI Institute, and Department of Applied Mathematics and Statistics,  Johns Hopkins University, Baltimore, MD 21218. Email: sqma@jhu.edu.}
\footnotetext[4]{School of Mathematics, Nanjing University, Nanjing,  P. R. China. Email: jfyang@nju.edu.cn}

\maketitle

\begin{abstract}
The constrained gradient method (CGM) has recently been proposed to solve convex optimization and monotone variational inequality (VI) problems with general functional constraints. While existing literature has established convergence results for CGM, the assumptions employed therein are quite restrictive; in some cases, certain assumptions are mutually inconsistent, leading to gaps in the underlying analysis. This paper derives rigorous and improved convergence guarantees for CGM under weaker and more reasonable assumptions,
covering strongly convex optimization problems, as well as VI problems under both strongly monotone and general monotone settings.
Preliminary numerical experiments are provided to verify the validity of CGM and demonstrate its efficacy in addressing such problems.
\end{abstract}
 
\noindent \msc 15A18, 15A69, 65F15, 90C33 \\

\section{Introduction}\label{sc:intro}

The constrained gradient method (CGM)  was first proposed in \cite{MJ22} for solving the constrained convex optimization problem:
\begin{equation}
\begin{aligned} \label{pro:min_f(x)}
\min_{x\in \mathcal{C}} f(x), 
\end{aligned}
\end{equation}
where $f:\mathbb{R}^n \rightarrow \mathbb{R}$ is a smooth and convex function, and the constraint set $\mathcal{C}\subseteq \mathbb{R}^n$ is defined by $m$ inequality constraints as follows: 
\begin{equation}\label{pro-con}
\mathcal{C}:= \{x\in \mathbb{R}^n \mid g_i(x)\leq 0,\, i  \in [m]\},
\end{equation}
with each $g_i$ being a smooth convex function. Note that throughout this paper, a function is referred to as ``smooth'' if its gradient operator is Lipschitz continuous, and $[m] := 
\{1,\ldots,m\}$. 
This method has been further extended to accelerated variants \cite{MJ25}, online and stochastic nonconvex minimization settings \cite{KMM23, STM22, STMM23}, and notably, to solving variational inequality (VI) problems \cite{ZHM25}. A VI problem seeks an $x^*\in \mathcal{C}$ satisfying
\begin{equation}\label{pro-vi}
\langle F(x^*), x^*-x\rangle \leq 0, \quad \forall  x \in \mathcal{C},
\end{equation}
where $F:\mathbb{R}^n\rightarrow \mathbb{R}^n$ is a monotone operator. 
However, as we will discuss in more detail later, the assumptions in \cite{MJ22} and \cite{ZHM25} are quite restrictive, and certain assumptions in \cite{ZHM25} are even mutually inconsistent. The main purpose of this paper is to establish convergence guarantees for CGM under weaker and more reasonable assumptions, specifically when $f$ in \eqref{pro:min_f(x)} is strongly convex and {$F$ in \eqref{pro-vi} is either strongly monotone or generally monotone. Throughout this paper, we make the following assumption: the solution sets of both \eqref{pro:min_f(x)} and \eqref{pro-vi} are nonempty. We let $x^*\in\mathcal{C}$ denote any solution to both problems. When $f$ in \eqref{pro:min_f(x)} is strongly convex or $F$ in \eqref{pro-vi} is strongly monotone, the underlying problem admits a unique solution, which is also denoted by $x^*$.}

Both problems \eqref{pro:min_f(x)} and \eqref{pro-vi} are of great interest. For instance, problem \eqref{pro:min_f(x)} encompasses a wide variety of applications, including but not limited to operations planning \cite{FJV86, PEM21}, support vector machines \cite{CV95, HDOP98, RN17}, and reinforcement learning \cite{KF11}; and problem \eqref{pro-vi} offers a powerful, unified framework for modeling and analyzing numerous problems, such as equilibrium problems \cite{N50, N51, N98}, optimization problems (e.g., matrix minimization and generative adversarial networks) \cite{H18, TYH11, G18}, and control theory \cite{SB84, NT88}. A more thorough review of algorithms for solving problems \eqref{ass:min_f(x)} and \eqref{pro-vi} is provided in Appendices \ref{appendix:a} and \ref{appendix:b}, respectively.

CGM for solving problem \eqref{pro:min_f(x)} mirrors the streamlined design of traditional projection-based methods in both the generation of descent directions and the choice of line-search-free step sizes. To avoid the notorious difficulty of projecting onto $\mathcal{C}$, it does not project each iterate directly onto the feasible region; instead, it projects the update direction onto a local and sparse velocity polytope. Built around the iterate $x$, this polytope takes the form: 
\begin{equation} \label{def:velocity}
        V_{\alpha}(x) = \{ v \in \mathbb{R}^n \mid \alpha g_i(x) + \nabla g_i(x)^\top v \leq 0, \; \forall i \in I_{x}^{\textbf{act}} \}
        \text{~~with~~} {I_{x}^{\textbf{act}}} := \{ i \in [m] \mid g_i(x) \geq 0 \}.
\end{equation}
Here,  ${I_{x}^{\textbf{act}}}$ denotes the set of constraints {violated or active} at $x$, and $\alpha>0$ controls the tradeoff between optimizing the objective function and maintaining feasibility. 
This design embodies CGM's distinctive hallmark, differentiating it from traditional projection-based methods. Specifically, at the $(t+1)$-th iteration, CGM for solving \eqref{pro:min_f(x)} updates the iterates as follows:
\begin{equation} \label{update-vx}
x^{t+1} = x^{t} + \eta v^t,  \text{~~with~~}
v^t := {\arg \min}_{v \in V_{\alpha}(x^t)}   \bigl\| v + {\nabla f(x^t)} \bigr\|^2.
\end{equation}
Here, $\eta>0$ denotes the step size. 
{With a slight  abuse of terminology,  in this paper we refer to $I_{x}^{\textbf{act}}$ as the set of constraints active at $x$.} 
Regarding the update \eqref{update-vx}, \cite{MJ22} established the following results:

\vspace{1em}
\begin{mdframed}[style=myframe]
\paragraph{Results adapted from \cite{MJ22}.} 
Assume the following assumptions: 
\begin{enumerate}[label = (\roman*)]
    \item The Mangasarian-Fromovitz constraint qualification (MFCQ) holds at every $x\in\mathbb{R}^{n}$\footnote{This condition specifically necessitates that for every $x\in \mathbb{R}^n$ and each $i\in {I_{x}^{\textbf{act}}}$, there must exist a vector $w\in \mathbb{R}^n$ such that $\langle \nabla g_i(x), w\rangle <0$.}\!;
    \item $f$ is closed, proper, $\mu$-strongly convex and smooth; each $g_i$ is smooth;
    \item The feasible set $\mathcal{C}$ is convex and bounded;
    \item The parameters $\alpha$ and $\eta$ satisfy $0<\alpha<\mu$ and $\eta\leq 2/(\ell_{*}+\mu)$, where $\ell_{*}$ denotes an upper bound of the smoothness constant of the Lagrangian function $\mathcal{L}(x,\lambda(x))= f(x)+\sum_{i=1}^{m}\lambda_{i}(x)g_{i}(x)$, $x\!\in\!\mathbb{R}^n$, and $$\lambda(x)\!\in\! \arg\max
    \Big\{\mathcal{L}(x,\lambda)-\frac{1}{2\alpha}\|\nabla_x \mathcal{L}(x,\lambda)\|^2 ~~\big|~~ \lambda\in \mathbb{R}^{m}_{+}, \; \lambda_i=0, \, \forall i \notin {I_x^{\textbf{act}}} \Big\}.$$
\end{enumerate}
\noindent 
Then, the sequence   $\{x^t\}$ generated by \eqref{update-vx} from any  $x^0\in\mathbb{R}^n$ converges to the unique minimizer $x^*$ of \eqref{pro:min_f(x)}. Furthermore,  there exist a constant $C>0$ and a positive integer $N$ such that
\[
\min_{t\in \{0,\ldots,T\}}\|x^t-x^*\|^2 \leq C/T, \quad \forall T\geq N.
\]
\end{mdframed}

\noindent Note that assumption (iv) is apparently impractical, since $\ell_*$ is typically unavailable, and therefore it is not clear how to choose $\eta$ in practice.

Recently, the authors of \cite{ZHM25} extended CGM to solve the VI problem \eqref{pro-vi}, with several modifications made to the update rule \eqref{update-vx}. Specifically, to ensure the boundedness of the sequences $\{x^t\}$ and $\{v^t\}$, they introduced an auxiliary constraint $g_{m+1}(x) = \|x\|^2 - \widehat{D}^2\leq 0$, where $\widehat{D}>0$ denotes an upper bound on the norm of elements in $\mathcal{C}$, i.e., $\|x\|\leq \widehat{D}$ for all $x\in \mathcal{C}$. 
As a result, the set of active constraints at $x$ in \cite{ZHM25} is adjusted accordingly to  ${I_{x}^{\textbf{act}}} = \{ i \in [m+1] \mid g_i(x) \geq 0 \}$.
Their CGM for the VIs iterates as follows:
\begin{equation}  \label{update-vx-vi}
x^{t+1} = x^{t} + \eta_t v^t,  \text{~~with~~}
v^t = \mathop{\arg\min}\nolimits_{v \in V_{\alpha}(x^t)} \bigl\| v + F(x^t) \bigr\|^2.
\end{equation}

{
Note that $V_\alpha(x)$ in \eqref{update-vx-vi} is defined as the first half in \eqref{def:velocity} but with $I_{x}^{\textbf{act}} = \{ i \in [m+1] \mid g_i(x) \geq 0 \}$. This differs from $V_\alpha(x)$ in \eqref{update-vx}, which is fully defined in \eqref{def:velocity} and does not involve the auxiliary constraint. 
}
Given $\mathcal{C}\subset \mathcal{B}(0,\widehat{D}):=\{x\in\mathbb{R}^n\colon\|x\|\le \widehat{D}\}$, 
the convergence of \eqref{update-vx-vi} is analyzed in \cite{ZHM25} 
under both strongly monotone and general monotone settings for $F$. 
The main results for the strongly monotone case are summarized as follows:   

\vspace{1em}
\begin{mdframed}[style=myframe]
\paragraph{Results adapted from \cite{ZHM25}.} Assume the following assumptions:
\begin{enumerate}[label = (\roman*)]
\item $F$ is continuous, $\mu$-strongly monotone, and $L_F$-bounded, i.e., $\|F(x)\|\leq L_F$ for all $x\in \mathbb{R}^n$;
\item All $g_i$'s are convex, $L_g$-Lipschitz continuous, and $\ell_g$-smooth;
\item The feasible set $\mathcal{C}$ is nonempty, closed and bounded, i.e., $\mathcal{C} \subseteq \mathcal{B}(0, \widehat{D})$.
\end{enumerate} 
\noindent {Suppose $x^0 \in \mathcal{C}$}, and let $\{x^t\}$ be the sequence generated by \eqref{update-vx-vi},  then CGM achieves the following convergence results:
\begin{enumerate}[label = {(\alph*)}]
\item When the parameters are chosen as $\eta_t = \frac{1}{\mu(t+1)}$ for $t=0,\ldots,T$, and $\alpha=\frac{\mu(\gamma-1)}{\gamma+1}$ with $\gamma>1$, there hold
\[
\langle F(x), \bar{x}^{T}-x\rangle = \mathcal{O}\left(\frac{1}{T}\right), \; \forall x\in \mathcal{C}, \text{ and } g_i(\bar{x}^{T}) = \mathcal{O}\left(\frac{\gamma^2+\gamma^2\zeta(1+2/(\gamma+1))}{(T+1)^{1-2/(\gamma+1)}}\right), \; \forall i\in [m\!+\!1],
\]
where, for any integer $T\geq 2$, $\bar{x}^{T} = \frac{2}{T(T-1)}\sum_{t=0}^{T-1} tx^t$, and
$\zeta(p)=\sum_{s=1}^{\infty} 1/s^p$ represents the Riemann zeta function.

\item When $F=\nabla f$ with $f$ being a $\mu$-strongly convex and $\ell_f$-smooth function, and the parameters are chosen as $\eta_t\equiv \eta = (\log T)/(\mu T)$, $\alpha=\mu$, where $T$ satisfies $T\geq \max\{3,\ell_f/\mu\}\log T$, then there hold
\[
f(x^{T})-f(x^*)=\mathcal{O}\left(\frac{1}{T}\right),\; \text{ and } g_i(x^{T})=\widetilde{\mathcal{O}}\left(\frac{1}{T}\right), \; \forall i\in [m+1].
\]
\end{enumerate}
\end{mdframed}

\noindent Here, we note that the convergence bound for $g_i(\bar{x}^T)$ in part (a) is weaker than $\mathcal{O}(1/T)$: while taking $\gamma\rightarrow+\infty$ makes the denominator approach $T+1$, the numerator also tends to $+\infty$ in this case, which degenerates the bound. Moreover, assumption (i) above is flawed, as discussed below. 
First, it is not feasible to impose the joint assumptions that the operator $F$ is both $\mu$-strongly monotone and $L_F$-bounded. To verify this point, suppose for contradiction that both conditions hold. Then, choose $x,y\in\mathbb{R}^n$ satisfying $\|y-x\|> 2L_{F}/\mu$; we have 
    \[
    2L_{F}\|y-x\|\geq \|F(y)-F(x)\| \|y-x\|\geq \langle F(y)-F(x), y-x\rangle\geq \mu\| y-x\|^2. 
    \]
This would imply that $\|y-x\|\leq 2L_{F}/\mu$, which cannot be true for arbitrary $x$ and $y$.
Second, for the minimization problem \eqref{pro:min_f(x)}, we have $F = \nabla f$, and 
the requirement that $\|F(x)\|\leq L_F$ for all $x$ simplifies to $\|\nabla f(x)\|\leq L_f$ for all $x$. This, as clearly noted in \cite{G19},  contradicts the $\mu$-strong convexity of $f$.  

{For monotone VIs, it is assumed in \cite{ZHM25} that $F$ is continuous and globally $L_F$-bounded; all $g_i$'s are convex, $L_g$-Lipschitz continuous, and $\ell_g$-smooth, and $\mathcal{C}\subseteq \mathcal{B}(0,\widehat{D})$. Starting from a feasible initial point  $x^0\in\mathcal{C}$, the authors of \cite{ZHM25} establish an $\mathcal{O}(1/\sqrt{T})$ convergence rate for both the Minty gap (see Definition \ref{weak-e-sol}) and the constraint violation of the averaged iterate.}

\paragraph{Motivation and contributions.} 
Based on the above discussions, our primary motivation in this paper is to provide convergence analyses for CGM under weaker and more justifiable assumptions. {Our contributions are summarized below. }
\begin{itemize}
    \item[(i)] For strongly convex optimization problems with general functional constraints, we establish the convergence of CGM without relying on the flawed assumption $\|\nabla f\|\leq L_f$ used in \cite{ZHM25}, while further relaxing other requirements.
    Specifically, we neither assume that the constraint set $\mathcal{C}$ is bounded 
    nor that all $g_i$'s are $L_g$-Lipschitz continuous, {and do not require the initial point
    $x^0$ to be feasible.}
    Additionally, we establish the same convergence rates as in \cite{ZHM25} for both constant and varying step sizes. In contrast to \cite{MJ22}, our analysis does not require dual information, such as $\ell_*$, which is not practically available.
    
    \item[(ii)] For strongly monotone VIs with general functional constraints, we conduct a convergence analysis of CGM under more appropriate and relaxed conditions. 
    Specifically, we replaces the flawed assumption $\|F(x)\|\leq L_F$ for all $x$ with a more reasonable one and {eliminates the requirements} that each $g_i$ be $L_g$-Lipschitz continuous, {that the constraint set $\mathcal{C}$ be bounded, and that $x^0$ be feasible.}
    Moreover, we also improve upon the convergence rates given in \cite{ZHM25}. 

    \item[(iii)] {For monotone VIs with general functional constraints, we develop a horizon-dependent expanding-ball variant of CGM that accommodates a possibly unbounded constraint set $\mathcal{C}$ and an arbitrary starting point $x^0\in\mathbb{R}^n$. Additionally, when the feasible set $\mathcal{C}$ is unbounded,
we show that any prefixed radius of the auxiliary ball is insufficient in the sense that one can construct a simple example dependent on this radius for which the CGM equipped with such an auxiliary ball constraint leads to failure.
This motivates the use of an expanding, rather than fixed, auxiliary ball.}
\end{itemize}

{In summary, we develop convergence analyses of CGM under more realistic and substantially weaker conditions. These refinements yield stronger convergence guarantees and allow the method to be initialized at infeasible starting points and to handle unbounded constraint sets.} 

\paragraph{Organization.} The rest of this paper is organized as follows.
In Section \ref{sec:pre}, we introduce the notation, optimality measures employed throughout the paper, and present several commonly used lemmas.
In Section \ref{sec:min_f}, we provide a refined convergence analysis for the strongly convex optimization problem. We exclude inappropriate assumptions and the auxiliary constraint, {no longer require $x^0$ to be feasible,} and establish convergence rates for both constant and varying step sizes.
In Section \ref{sec:vi}, we conduct a refined convergence analysis for {strongly monotone} VIs under relaxed conditions, {which eliminates the auxiliary constraint, removes the boundedness requirement on the constraint set, the global Lipschitz continuity requirement on the constraint functions, and the feasibility requirement on $x^0$}. We derive a convergence rate for the optimality measure that matches the original result, while achieving an improved rate for the feasibility measure compared to the result in \cite{ZHM25}.
{In Section \ref{sec:monotone-vi-expanding-ball}, we further study the monotone VIs and establish convergence guarantees without requiring a bounded constraint set or a feasible initial point}.
In Section \ref{sec:ne}, we carry out numerical experiments to validate the proposed convergence results.
Finally, we draw some concluding remarks in Section \ref{sec:conclusion}.

\section{Preliminaries} \label{sec:pre}
In this section, we first introduce the notation and optimality measures employed throughout the paper, followed by the presentation of two useful lemmas.

\paragraph{Notation.} 
Throughout the paper, we use the following notation. 
The cardinality of a set $\mathcal{W}$ is denoted by $|\mathcal{W}|$.  
Let $\mathbb{R}^n$ be a finite-dimensional Euclidean space, with the inner product and its induced norm denoted by $\langle\cdot, \cdot\rangle$ and $\|\cdot\| =\sqrt{\langle\cdot,\cdot\rangle}$, respectively.  
Let $\mathbb{R}_{+}^{n}$ denote the set of $n$-dimensional real vectors with nonnegative entries.  The standard basis vector in $\mathbb{R}^n$ is denoted by $\mathbf{e}_i = (\delta_{ij})_{j=1}^n$, where $\delta_{ij}$ is the Kronecker delta, i.e., $\delta_{ij}=1$ if $j=i$, and $\delta_{ij}=0$  otherwise.
Furthermore, $\mathbf{1}_n \in \mathbb{R}^n$ denotes the $n$-dimensional column vector of all ones, and $\mathbf{0}_n \in \mathbb{R}^n$ the $n$-dimensional zero vector;  $I_n$ stands for the identity matrix of order $n$. For a symmetric matrix $M$ of appropriate dimensions, 
$M\succeq 0$ (or $M\succ 0$) denotes that $M$ is positive semidefinite (or positive definite). 

An operator $F:\mathcal{X}\rightarrow \mathbb{R}^n$ defined on a convex set $\mathcal{X}\subseteq\mathbb{R}^n$ is called $\mu$-strongly monotone with parameter $\mu>0$ if $\langle F(x)-F(y),x-y\rangle \geq \mu \|x-y\|^2$ for all $x,y \in \mathcal{X}$. When $\mu=0$, the operator is simply referred to as a monotone operator. 
If there exists $\ell_F>0$ such that $\|F(x)-F(y)\|\leq \ell_{F}\|x-y\|$ for all $x,y \in \mathcal{X}$, then $F$ is $\ell_F$-Lipschitz continuous. 
Similarly, a function $\theta:\mathcal{X}\rightarrow \mathbb{R}$ is said to be $L_{\theta}$-Lipschitz if $|\theta(x)-\theta(y)|\leq L_{\theta}\|x-y\|$ for all $x,y\in \mathcal{X}$. 
When $\theta(\cdot)$ is differentiable, $L_{\theta}$-Lipschitz is equivalent to $\|\nabla\theta (x)\|\leq L_{\theta}$ for all $x\in \mathcal{X}$. 
A function $\theta$ is called $\ell_{\theta}$-smooth if it is differentiable and its gradient operator is $\ell_{\theta}$-Lipschitz, i.e., $\|\nabla \theta(x) - \nabla \theta(y)\|\leq \ell_{\theta}\|x-y\|$  for all $x,y\in \mathcal{X}$. 
We denote the set of optimal solutions to problem \eqref{pro-vi} as $\mathcal{X}^*$, and let $x^*$ be any element in $\mathcal{X}^*$. Finally, $\widetilde{\mathcal{O}}$ is the  big-$\mathcal{O}$ notation with logarithmic factors ignored.

\paragraph{Optimality measures.} 
The VI problem in \eqref{pro-vi} is commonly designated as the Stampacchia VI problem, and a solution to \eqref{pro-vi} is referred to as a strong solution. 
In this paper, our goal is to find a weak solution $x^* \in \mathcal{C}$, which solves the Minty VI problem:
\begin{equation}
\langle F(x),x^*-x\rangle \leq 0, \quad \forall x\in \mathcal{C}. \label{pro-vi-minty}
\end{equation}
This Minty VI framework has been widely adopted; see, e.g., \cite{N04, N07, JNT11, ZHM25}. 
It is well known (see, e.g., \cite{M62, KS20}) that if $F$ is continuous and monotone, the strong and weak solutions coincide. 
For an approximate solution $z\in\mathbb{R}^n$, we use $\max_{x\in \mathcal{C}}\langle F(x),z-x\rangle$ to measure optimality 
and $g_i(z)$ ($i\in[m]$) to measure  feasibility. Specifically,  we adopt the notion of \textit{weak $\epsilon$-solution} as in \cite{ZHM25}.
\begin{definition}[Weak $\epsilon$-solution]\label{weak-e-sol}
Let $\epsilon>0$, and the feasible set $\mathcal{C}$ is defined in \eqref{pro-con}.
A point $z\in \mathbb{R}^n$  is called a weak $\epsilon$-solution of \eqref{pro-vi}  if {the Minty gap} $\mathcal{G}(z) := \max_{x\in \mathcal{C}}\langle F(x),z-x\rangle \leq \epsilon$ and $g_i(z)\leq \epsilon$ for all $i \in [m]$.
\end{definition}

\paragraph{Useful lemmas.} Next, we present two lemmas that encapsulate key properties of the velocity polytope $V_{\alpha}(x)$ defined in \eqref{def:velocity}. {Note that $I_{x}^{\textbf{act}}$ used in these lemmas is also defined in \eqref{def:velocity}; in particular, it does not involve the auxiliary constraint.} 
These properties are implicitly used in \cite{ZHM25} but not explicitly formulated as lemmas. For the sake of completeness, we provide their proofs herein.

\begin{lemma}\label{lemma:depict_veloc_1}
    For all $x\in \mathcal{C}$ and $\alpha>0$, we have $\alpha(x-x^{t})\in V_{\alpha}(x^t)$.
\end{lemma}
\proof{Proof.}
Let $x\in \mathcal{C}$, i.e., $g_i(x)\leq 0$ for all $i\in[m]$. 
Then, for any $i\in {I_{x^t}^{\textbf{act}}}$, the convexity of $g_i$ implies that
$g_i(x^{t}) +  \langle \nabla g_i(x^{t}), x-x^{t}\rangle \leq g_i(x) \leq 0$.
Hence, the conclusion $\alpha(x-x^{t})\in V_{\alpha}(x^t)$ follows from 
the definition of $V_{\alpha}(x^t)$ in \eqref{def:velocity} and $\alpha>0$.
\Halmos \endproof

\begin{lemma}\label{lemma:depict_veloc_2}
    For all $x\in \mathcal{C}$ and $v \in V_{\alpha}(x^t)$, we have 
    $v+(x-x^t)\in V_{\alpha}(x^t)$.
\end{lemma}
\proof{Proof.}
Let $x\in \mathcal{C}$, $v \in V_{\alpha}(x^t)$ and $i\in {I_{x^t}^{\textbf{act}}}$ be arbitrarily fixed. Then, $g_i(x)\leq 0$ and 
$g_i(x^t)\geq 0$. Further considering the convexity of $g_i$, we obtain
\begin{equation} \label{eq-add}
\langle \nabla g_i(x^t), x-x^t\rangle \leq g_i(x^t) + \langle \nabla g_i(x^t), x-x^t\rangle  
\leq g_i(x)\leq 0.
\end{equation}
Since $v \in V_{\alpha}(x^t)$, we have $\alpha g_i(x^t) + \langle \nabla g_i(x^t), v\rangle \leq 0$. Then, by combining this with \eqref{eq-add} and \eqref{def:velocity}, we obtain the desired result.
\Halmos \endproof

\section{CGM for Strongly Convex Optimization} \label{sec:min_f}
In this section, we present an enhanced convergence analysis of CGM for solving the optimization problem \eqref{pro:min_f(x)}, 
where $f$ is $\mu$-strongly convex and $\ell_f$-smooth. Compared with the results in \cite[Thm. 3]{ZHM25}, we 
remove the inappropriate requirement that $f$ is $L_f$-Lipschitz and further relax other assumptions.
Throughout this section, we make the following assumptions.
\begin{assumption} \label{ass:min_f(x)}
Assume that (i) $f$ is $\mu$-strongly convex and $\ell_f$-smooth, (ii) all $g_i$'s are convex and $\ell_g$-smooth,
and (iii) the feasible set $\mathcal{C}$ is nonempty. Here, $\ell_f \geq \mu > 0$ and $ \ell_g > 0$ are some constants. 
\end{assumption}

Note that under Assumption \ref{ass:min_f(x)}, problem \eqref{pro:min_f(x)} has a unique optimal solution, which we denote by $x^*$,
i.e., $x^* :=\arg\min_{x\in\mathcal{C}} f(x)$.  
We further let $x^{\star}:=\arg\min_{x\in\mathbb{R}^n} f(x)$ (the unique optimal solution of $f$ over $\mathbb{R}^n$) and denote the corresponding function value by $f_{\star} := f(x^\star)$. The existence of $x^{\star}$ is also guaranteed by Assumption \ref{ass:min_f(x)}.
Note that $f_{\star}$ is different from $f(x^*)$ and both are finite.

\begin{remark}  
As explained in Section \ref{sc:intro}, the strong convexity of $f$ is inconsistent with its $L_f$-Lipschitz continuity. 
In our analysis, we only assume that $f$ is strongly convex, but not $L_f$-Lipschitz.
Furthermore, we neither impose an $L_g$-Lipschitz condition on each constraint function $g_i$ nor require {$x^0 \in \mathcal{C}$ or boundedness of the feasible set  $\mathcal{C}$. }
Note that all these conditions were assumed in \cite{ZHM25}. 
Thus, compared with the results in \cite{ZHM25}, our analysis is more rigorous and general.
In practice, {obtaining a feasible initial point $x^0 \in \mathcal{C}$ can itself be computationally challenging, especially when projection onto $\mathcal{C}$ is expensive; allowing arbitrary initialization therefore makes CGM more practical while preserving its infeasible-iterate advantage.} Additionally, the feasible set $\mathcal{C}$ can indeed be unbounded. Typical examples include linear inequality constraints $Ax\le b$ and exponential-decay bounds $e^{-x}\leq \bar{C}$, which arise in control systems. 
Even when $\mathcal{C}$ is bounded, it is nontrivial to obtain its diameter in advance---note that this diameter is explicitly required for the algorithmic construction in  \cite{ZHM25}.
\end{remark}

Our CGM for solving problem \eqref{pro:min_f(x)} 
is presented in Algorithm \ref{alg:cgm-f}. 
Note that in Algorithm \ref{alg:cgm-f} {(as well as in all other CGM varaints proposed in this paper)}, we define
\begin{equation} \label{def:new-ixt}
{I_{x} := \{ i \in [m] \mid g_i(x) > 0 \},    }
\end{equation}
which differs from ${I_{x}^{\textbf{act}}}$ in \eqref{def:velocity}. Specifically, our $I_x$ denotes the set of constraints violated at $x$, whereas ${I_{x}^{\textbf{act}}}$ corresponds to the set of constraints active at $x$.  This change stems from our refined analysis. {Despite this minor distinction, Lemmas \ref{lemma:depict_veloc_1} and \ref{lemma:depict_veloc_2} remain valid  when ${I_{x}^{\textbf{act}}}$ is replaced by $I_{x}$.}

\begin{algorithm}
\caption{CGM for Strongly Convex Minimization  (CGM-Min)}
\label{alg:cgm-f}
\begin{algorithmic}[1]
    \State {Let $x^0 \in \mathbb{R}^n$, and set $\alpha = \mu$}. 
\For{$t = 0, 1, 2, \ldots$}
    \State Build the set of constraints violated at $x^t$:    
    $I_{x^t} = \{ i \in [m] \mid g_i(x^t) > 0 \}$.
    \State Construct the velocity polytope
    \begin{equation}\label{def-V-alpha}
        V_{\alpha}(x^t) = \{ v \in \mathbb{R}^n \mid \alpha g_i(x^t) + \nabla g_i(x^t)^\top v \leq 0, \; \forall i \in I_{x^t} \}.
    \end{equation}
    \State Solve the quadratic programming problem
    \begin{equation} \label{update-v-f}
        v^t = {\arg \min}_{v \in V_{\alpha}(x^t)}  \bigl\| v + \nabla f(x^t) \bigr\|^2.
    \end{equation}
    \State Choose $0<\eta_t \leq \min\{1/\ell_f, 1/\alpha\}$ and update the iterate 
    $x^{t+1} = x^{t} + \eta_t v^t$.    
\EndFor
\end{algorithmic}
\end{algorithm}

\begin{remark} \label{remark-alg}
We present the following remarks on Algorithm \ref{alg:cgm-f}.
\begin{enumerate}
    \item Unlike the CGM in \cite{ZHM25}, we remove the auxiliary constraint $g_{m+1}(x)=\|x\|^2 - \widehat{D}^2\leq 0$. 
    This constraint was originally used to ensure the boundedness of sequences $\{x^t\}$ and $\{v^t\}$ in their theoretical analysis. 
    Additionally, we relax the assumptions on $\mathcal{C}$ by permitting it to be unbounded.  {We also allow $x^0$ to be chosen arbitrarily, without requiring feasibility.}

    \item Leveraging our refined analysis, we relax the active constraint set in \eqref{def:velocity} to the violated constraint set $I_{x^t} = \{ i \in [m] \mid g_i(x^t) > 0\}$. A beneficial byproduct of this relaxation is a reduction in gradient computations $\{\nabla g_i(x^t): i \in I_{x^t}\}$, owing to the smaller cardinality of our $I_{x^t}$. Notably, this relaxation does not affect the conclusion of Lemma \ref{lemma:depict_veloc_2}, as $\langle \nabla g_i(x^t), x-x^t \rangle\leq 0$ remains valid. To see this, for all  $i\in I_{x^t}$, since $g_i(x^t) > 0$, \eqref{eq-add} becomes
    $$\langle \nabla g_i(x^t), x-x^t\rangle < g_i(x^t) + \langle \nabla g_i(x^t), x-x^t\rangle  \leq g_i(x)\leq 0.$$ 

    \item When $x^t \in \mathcal{C}$,  $I_{x^t} = \varnothing$ (the empty set) holds, in which case $V_{\alpha}(x^t)$ is the entire space $\mathbb{R}^n$. 
\end{enumerate}
\end{remark}

The core idea of our analysis is to leverage function values to bound the magnitude of the projected velocity $\|v^t\|$. The corresponding lemmas are presented as follows.


\begin{lemma}[Bounding $\|v^t\|^2$ via Function Values] \label{lemma:control_v}
Under Assumption \ref{ass:min_f(x)} for Problem \eqref{pro:min_f(x)},  Algorithm \ref{alg:cgm-f} satisfies, for all $t\geq 0$, the following:
\begin{equation}\label{est-v-f}
\|v^t\|^2 \leq 4(2\ell_f-\alpha)(f(x^t)-f(x^*))+8\ell_f(f(x^{*})-f_{\star}).
\end{equation}
\end{lemma}
\proof{Proof.}
Let $t\geq 0$ be arbitrarily fixed. 
Recall that the optimal function value of $f$ over $\mathbb{R}^n$ is attained at $x^\star$, denoted $f_{\star}$. Given that $f$ is convex and $\ell_f$-smooth, we have
\begin{equation} \label{est_nabla_f}
f(x^t)-f_{\star} \geq \frac{1}{2\ell_f}\|\nabla f(x^t)\|^2.   
\end{equation}
From \eqref{update-v-f}, it follows that $\|v^t+\nabla f(x^t)\|\leq \|v+\nabla f(x^t)\|$ for all $v \in V_{\alpha}(x^t)$. Combining this with Lemma \ref{lemma:depict_veloc_1} and the fact that $x^* \in \mathcal{C}$, we obtain $\alpha(x^*\!-\!x^t) \in V_{\alpha}(x^t)$ and 
\begin{equation}\label{est-v-nabla-f}
\begin{aligned}
\|v^t+\nabla f(x^t)\|^2 & \leq \|\alpha(x^*\!-\!x^t)+\nabla f(x^t)\|^2 
= \alpha^2 \|x^*\!-\!x^t\|^2  +2\alpha \langle x^*\!-\!x^t,\nabla f(x^t)\rangle + \|\nabla f(x^t)\|^2\\
& \leq \alpha^2 \|x^*-x^t\|^2 + 2\alpha\big(f(x^*)-f(x^t)- (\mu/2) \|x^t-x^*\|^2\big) + \|\nabla f(x^t)\|^2 \\
& = \alpha(\alpha-\mu)\|x^t-x^*\|^2 + 2\alpha\big(f(x^*)-f(x^t)\big)+\|\nabla f(x^t)\|^2\\
& \leq 2\alpha(f(x^*)-f(x^t))+ 2\ell_f(f(x^t)-f_{\star})\\
& = 2(\ell_f-\alpha) (f(x^t)-f(x^*))+2\ell_f(f(x^{*})-f_{\star}),
\end{aligned}
\end{equation}
where the second inequality follows from the $\mu$-strong convexity of $f$, and the third from \eqref{est_nabla_f} and $\alpha= \mu$.
Since $\|v^t\|^2  \leq 2\big( \|v^t+\nabla f(x^t)\|^2 +  \|\nabla f(x^t)\|^2\big)$, combining \eqref{est-v-nabla-f} and \eqref{est_nabla_f} yields \eqref{est-v-f}.
\Halmos \endproof

To further bound $\|v^t\|$, the only nontrivial term in \eqref{est-v-f} is $f(x^t)-f(x^*)$. This can be effectively controlled by setting appropriate step sizes $\eta_t$. We now introduce the following lemma, which establishes the boundedness of $\|v^t\|^2$ and $\|x^t - x^*\|$.

\begin{lemma}[Boundedness of $\|v^t\|$ and $\|x^t-x^*\|$] \label{lemma:bound-v-f}
Under Assumption \ref{ass:min_f(x)} for Problem \eqref{pro:min_f(x)},  Algorithm \ref{alg:cgm-f} satisfies, for all $t\geq 0$, the following:
\begin{align}
\|v^t\|  &\leq  C_1 := \sqrt{4(2\ell_f-\alpha){[f(x^0)-f(x^*)]_+}+8\ell_f(f(x^{*})-f_{\star})}, \label{Cv-w-u} \\
\|x^{t}-x^*\| & \leq C_{2} := \frac{1}{\mu}\begin{pmatrix} \|\nabla f(x^*)\|+\sqrt{\|\nabla f(x^*)\|^2+2\mu{[f(x^0)-f(x^*)]_+}} \end{pmatrix}.\label{cor:bound_x} 
\end{align}
\end{lemma}
\proof{Proof.}
Let $t\geq 0$ and $x\in \mathcal{C}$ be arbitrarily fixed. 
Following from $x^{t+1} = x^{t} + \eta_t v^t$, we have
\begin{align}
f(x^{t+1}) & \leq f(x^{t})+\eta_t \langle \nabla f(x^{t}), v^{t}\rangle+\frac{\ell_{f}}{2} \eta_t^{2}\left\|v^{t}\right\|^{2} \nonumber \\
&=f(x^{t})+\frac{\eta_t}{2}\left\|v^{t}+\nabla f(x^{t})\right\|^{2}-\frac{\eta_t}{2}\left\|\nabla f(x^{t})\right\|^{2}-\frac{\eta_t}{2}\left(1-\eta_t \ell_{f}\right)\left\|v^{t}\right\|^{2} \nonumber \\
&\leq f(x^{t})+\frac{\eta_t}{2}\left\|\alpha(x-x^{t})+\nabla f(x^{t})\right\|^{2}-\frac{\eta_t}{2}\left\|\nabla f(x^{t})\right\|^{2}-\frac{\eta_t}{2}\left(1-\eta_t \ell_{f}\right)\left\|v^{t}\right\|^{2} \nonumber \\
&=f(x^{t})+\frac{\eta_t}{2} \alpha^{2}\left\|x-x^{t}\right\|^{2}+\alpha \eta_t \langle\nabla f(x^{t}),x-x^{t}\rangle-\frac{\eta_t}{2}\left(1-\eta_t \ell_{f}\right)\left\|v^{t}\right\|^{2} \nonumber \\
&\leq f(x^{t})-\alpha\eta_t(f(x^t)-f(x))+\frac{\eta_t\alpha}{2}(\alpha-\mu)\left\|x-x^{t}\right\|^{2}-\frac{\eta_t}{2}\left(1-\eta_t \ell_{f}\right)\left\|v^{t}\right\|^{2}, \label{jy-1}
\end{align}
where the first inequality follows from the $\ell_f$-smoothness of $f$,
the second from Lemma \ref{lemma:depict_veloc_1},
and the third from $\langle\nabla f(x^{t}),x-x^{t}\rangle \leq f(x) - f(x^t) - (\mu/2)\|x-x^t\|^2$, which is due to
the $\mu$-strong convexity of $f$. 
Further considering {$\alpha= \mu$} and $\eta_t \leq 1/\ell_f$, we obtain from \eqref{jy-1} that
\begin{equation}\label{fn-value-shrink}
f(x^{t+1})-f(x) \leq (1-\alpha\eta_t)\big(f(x^{t})-f(x)\big). 
\end{equation}
Since $ 0 < \alpha \eta_t \leq 1$, recursive application of \eqref{fn-value-shrink} yields
\begin{equation}\label{pi-f} 
f(x^{t+1})-f(x)   \leq \big(\Pi_{i=0}^{t}(1-\alpha\eta_i)\big) \big(f(x^{0})-f(x)\big) \leq {[f(x^{0})-f(x)]_+}. 
\end{equation}
Note that {$\alpha = \mu \leq \ell_f$}, so $2\ell_f - \alpha \geq \ell_f > 0$.
Substituting $x=x^*$ into \eqref{pi-f} and using \eqref{est-v-f}, we derive the boundedness result for $\|v^t\|$ given in \eqref{Cv-w-u}. 
Moreover, by again utilizing the $\mu$-strong convexity of $f$ and the Cauchy-Schwarz inequality, we have
\[
f(x^{t})-f(x^*)\geq \langle \nabla f(x^{*}),x^{t}-x^*\rangle+\frac{\mu}{2}\|x^t-x^*\|^2\geq -\|\nabla f(x^{*})\|\|x^{t}-x^*\|+\frac{\mu}{2}\|x^t-x^*\|^2.
\]
Further considering \eqref{pi-f}, this implies $\frac{\mu}{2}\|x^t-x^*\|^2-\|\nabla f(x^{*})\|\|x^{t}-x^*\|\leq {[f(x^0)-f(x^*)]_+}$, from which the desired result \eqref{cor:bound_x} follows. 
\Halmos \endproof

Lemma \ref{lemma:bound-v-f} implies that the sequence $\{x^t\}$ generated by Algorithm \ref{alg:cgm-f}  is contained in the bounded region $\mathcal{B}(x^*,C_{2})$, 
a closed ball centered at $x^*$ 
with radius $C_2>0$ defined in \eqref{cor:bound_x}.
This boundedness guarantee allows us to remove the assumption that all $g_i$'s are $L_g$-Lipschitz continuous. 
With a slight abuse of notation, we define $L_g$ hereafter as 
\begin{equation}\label{def-Lg}
    L_g:=\max \{ \|\nabla g_i(x)\|: x\in \mathcal{B}(x^*,C_{2}), \, i \in [m]\}.
\end{equation}
We are now ready to present the main results in this section. 

\begin{theorem}[Convergence Rates]  \label{thom:f}
Consider Problem \eqref{pro:min_f(x)} under Assumption \ref{ass:min_f(x)}. 
Let $\kappa:=\ell_f/\mu\geq 1$ denote the condition number of $f$. 
Recall that 
$\mathcal{G}(\cdot)$ is defined in Definition \ref{weak-e-sol} and 
$C_1$ is given in \eqref{Cv-w-u}. 
Then, for Algorithm \ref{alg:cgm-f}, the following convergence results hold:
\begin{enumerate}
    \item[(i)] \textit{Constant step size}: Let $T> 1$ be any integer satisfying $T\geq \kappa \log T$ and $\eta_t = \eta := \log T/(\mu T)$ for all $t=0,1,\ldots, T-1$. 
    Then, Algorithm \ref{alg:cgm-f} achieves
    
\vspace{1em}  
\begin{tabular}{@{\hspace{-3.2em}}p{3cm}l@{}}
    (Fun. Val. Res.) & $f(x^T)-f(x^*)\leq {[f(x^0)-f(x^*)]_+}/T  \sim \mathcal{O}(1/T)$, \smallskip \\  
    (Opt. Gap) & $\mathcal{G}(x^{T})\leq  {[f(x^0)-f(x^*)]_+}/T  \sim \mathcal{O}(1/T)$, \smallskip  \\
    (Feasibility) & $g_i(x^{T})\leq {\dfrac{[g_i(x^0)]_+}{T}+{}}\dfrac{C_{1}}{\mu}\max\left\{\dfrac{C_{1}\ell_g}{2\mu}, L_g\right\}\dfrac{\log T}{T}\sim  \widetilde{\mathcal{O}}(1/T), \; \forall i \in [m]$.
\end{tabular}
\vspace{1em} 
    \item[(ii)] \textit{Varying step size}: Let $\eta_t = 1/(\mu(t+\kappa))$ for any $t\geq 0$. Then, for all $t\geq 1$, Algorithm \ref{alg:cgm-f} achieves 
    
\vspace{1em}     
\begin{tabular}{@{\hspace{-3.2em}}p{3cm}l@{}}
    (Fun. Val. Res.) & $f(x^t)-f(x^*)\leq \dfrac{ \kappa-1}{t+\kappa-1} {[f(x^{0})-f(x^*)]_+}\sim \mathcal{O}(1/t)$, \smallskip \\
    (Opt. Gap) & $\mathcal{G}(x^t)\leq \dfrac{ \kappa-1}{t+\kappa-1} {[f(x^{0})-f(x^*)]_+}\sim \mathcal{O}(1/t)$, \smallskip \\
    (Feasibility) &    $g_i(x^{t+1})\!\leq \, \!\!{\dfrac{\kappa[g_i(x^0)]_+}{t+\kappa\!+\!1}}\!+\!\dfrac{2C_{1}}{\mu(t+\kappa\!+\!1)}\left(L_g\!+\!\dfrac{\ell_g\!C_{1}}{2\mu}\right) \!+\! \dfrac{\ell_g C_{1}^2{\log(t\!+\!1)}}{\mu^2(t+\kappa+1)}\!\sim \! \widetilde{\mathcal{O}}(1\!/t), \; \forall i \in [m]$.
    \end{tabular}
\vspace{1em} 
\end{enumerate}
\end{theorem}

\proof{Proof.}
First, it is straightforward to verify that $0<\eta_t \leq \min\{1/\ell_f, 1/\alpha\}$ holds for both choices of $\eta_t$. 
Next, we prove the convergence rate results for the function value residual, followed by those for feasibility. 
Let $T\geq 1$ be any integer satisfying $T\geq \kappa \log T$ and $x\in\mathcal{C}$ be arbitrarily fixed.
For the function value residual, using the constant step size $\eta_t=\eta=\log T/(\mu T)$ for all $t=0,1,\ldots, T-1$, we have
\begin{align*}
\langle \nabla f(x),x^{T}-x\rangle 
& \leq  f(x^{T})-f(x)  \leq  f(x^{T})-f(x^{*}) \leq (1-\mu \eta)^{T}\left(f(x^0)-f(x^*)\right)  \\
& \leq \mathrm{e}^{-\mu\eta T}{[f(x^0)-f(x^*)]_+} = {[f(x^0)-f(x^*)]_+} / T,
\end{align*}
where the first ``$\leq$'' follows from the convexity of $f$, the second is because $x\in\mathcal{C}$ and $f(x) \geq f(x^*)$, the third 
follows from $\alpha = \mu$ and \eqref{fn-value-shrink} with $x=x^*$. Since $x\in\mathcal{C}$ is arbitrarily taken, the above inequality further implies 
\[
\mathcal{G}(x^{T})=\max_{x\in \mathcal{C}}\langle \nabla f(x),x^{T}-x\rangle \leq {[f(x^0)-f(x^*)]_+} / T.
\]
When using the varying step size $\eta_t = 1/(\mu(t+\kappa))$, similarly, let $x\in\mathcal{C}$ and $t\geq 1$ be arbitrarily fixed, we have
\begin{align*}
\langle \nabla f(x),x^{t}-x\rangle 
& \leq  f(x^{t})-f(x)    \leq \big(\Pi_{i=0}^{t-1}(1-\mu\eta_i)\big) \big(f(x^{0})-f(x)\big)  \\
& \leq \big(\Pi_{i=0}^{t-1}(1-\mu\eta_i)\big)  {[f(x^{0})-f(x^*)]_+} = \frac{ \kappa-1}{t+\kappa-1} {[f(x^{0})-f(x^*)]_+},
\end{align*}
where the first ``$\leq$'' follows from the convexity of $f$, the second ``$\leq$'' follows from the first inequality in \eqref{pi-f}, the third ``$\leq$'' is because $x\in\mathcal{C}$ and $f(x) \geq f(x^*)$, 
and the ``$=$'' is because $\eta_i=1/(\mu(i+\kappa))$ for $i\geq 0$.
Letting $x=x^*$ in the above inequality, we obtain
\[
f(x^t) - f(x^*)\leq \frac{ \kappa-1}{t+\kappa-1} {[f(x^{0})-f(x^*)]_+}, \quad \forall t\geq 1.
\]
On the other hand, since $x\in\mathcal{C}$ is arbitrarily taken, we obtain 
\[
\mathcal{G}(x^t)=\max_{x\in \mathcal{C}}\langle \nabla f(x),x^t-x\rangle \leq \frac{ \kappa-1}{t+\kappa-1} {[f(x^{0})-f(x^*)]_+}, \quad \forall t\geq 1.
\]

Next, we turn to the convergence rates of the feasibility. 
When using the constant step size $\eta_t=\eta=\log T/(\mu T)$ for all $t=0,1,\ldots, T-1$, we will prove that 
\begin{equation}\label{feas-rate}
\begin{aligned}
g_{i}(x^t)
&\leq {\left(1-\frac{\log T}{T}\right)^t[g_i(x^0)]_+ +{}}
\frac{C_{1}}{\mu}\max\left\{\frac{C_{1}\ell_g}{2\mu}, L_g\right\}\frac{\log T}{T}, \;\; \forall t=0,1,\ldots,{T}, \;\; i\in[m], 
\end{aligned}
\end{equation}
where $C_1$ is defined in \eqref{Cv-w-u}. 
To this end, we consider the following two cases.
\begin{itemize}
    \item[(a)] If $i \notin I_{x^{t}}$, then $
g_{i}(x^{t})\leq 0$. Following from the convexity of $g_i$, we have
\[
\begin{aligned}
g_{i}(x^{t+1})  \leq g_{i}(x^{t}) + \eta \nabla g_{i}(x^{t+1})^{\top} v^{t} 
 \leq\eta \|\nabla g_{i}(x^{t+1})\|\|v^{t}\|\leq \frac{C_{1} L_{g} \log T}{\mu T},
\end{aligned}
\]
where the last inequality follows from \eqref{Cv-w-u} and \eqref{def-Lg}. 
\item[(b)] If $i \in I_{x^{t}}$, then
$g_{i}(x^{t}) > 0$. From \eqref{def-V-alpha}, we have $\eta \alpha g_{i}(x^{t})+\nabla g_{i}(x^{t})^{\top}\left(x^{t+1}-x^{t}\right) \leq 0$. Following from the $\ell_g$-smoothness of $g_i$, we have
\begin{equation}\label{thm3.1caseb}
g_{i}(x^{t+1})  \leq\left(1-\alpha \eta\right) g_{i}(x^{t})+\frac{\ell_{g}}{2} \eta^{2} \|v^t\|^2 
 \leq \left(1-\frac{\log T}{T}\right)g_{i}(x^{t}) +\frac{\ell_g C_{1}^2  (\log T)^2}{2\mu^2 T^2}.
\end{equation}
\end{itemize}
Now we prove \eqref{feas-rate} by induction. 
{For $t=0$, we have $g_i(x^0)\leq[g_i(x^0)]_+$, so \eqref{feas-rate} holds trivially.}
Assume that \eqref{feas-rate} holds for some $t\ge 0$. For $(t+1)$,  in case (a), we have $g_{i}(x^{t+1}) \leq \frac{C_{1} L_{g} \log T}{\mu T}$, 
and in case (b), we have
\begin{equation*}
\begin{aligned}
g_{i}(x^{t+1}) &  \leq \left(1-\frac{\log T}{T}\right)g_{i}(x^{t}) +\frac{\ell_g C_{1}^2 (\log T)^2}{2\mu^2 T^2}\\
& \leq {\left(1-\frac{\log T}{T}\right)^{t+1}[g_i(x^0)]_+ +{}}
\left(1-\frac{\log T}{T}\right) \frac{C_{1}}{\mu}\max\left\{\frac{C_{1}\ell_g}{2\mu}, L_g\right\}\frac{\log T}{T}+\frac{\ell_g C_{1}^2  (\log T)^2}{2\mu^2 T^2} \\
& \leq {\left(1-\frac{\log T}{T}\right)^{t+1}[g_i(x^0)]_+ +{}}
\frac{C_{1}}{\mu}\max\left\{\frac{C_{1}\ell_g}{2\mu}, L_g\right\}\frac{\log T}{T},
\end{aligned}
\end{equation*}
where the first ``$\leq$" follows from \eqref{thm3.1caseb}, and the second is because the induction hypothesis. 
This proves \eqref{feas-rate}, and thus the desired result 
$g_i(x^{T})\leq {\dfrac{[g_i(x^0)]_+}{T}+{}}\dfrac{C_{1}}{\mu}\max\left\{\dfrac{C_{1}\ell_g}{2\mu}, L_g\right\}\dfrac{\log T}{T} \sim  \widetilde{\mathcal{O}}(1/T)$, $\forall i \in [m]$, {where we use $\left(1-\log T/{T}\right)^T\leq \mathrm{e}^{-\log T}=1/T$.}

For the varying step size $\eta_t=1/(\mu(t+\kappa))$, we similarly consider the following two cases. 
\begin{itemize}
    \item[(a)] If $i \notin I_{x^{t}}$, then $g_{i}(x^{t})\leq 0$. Following from the convexity of $g_i(x)$, we have
\begin{equation} 
\begin{aligned} \label{f-case-1}
g_{i}(x^{t+1}) & \leq g_i(x^t) + \nabla g_{i}(x^{t+1})^{\top} (x^{t+1}-x^{t}) \leq  \eta_{t} \nabla g_{i}(x^{t+1})^{\top} v^{t}  \\
&  \leq \eta_{t}\|\nabla g_{i}(x^{t+1})\|\|v^{t} \| \leq \frac{C_{1} L_{g}}{\mu(t+\kappa)}.
\end{aligned}
\end{equation}
\item[(b)] If $i \in I_{x^{t}}$, then 
$g_{i}(x^{t}) > 0$.  Since $x^{t+1} = x^t + \eta_t v^t$ and $\alpha = \mu$, it follows from \eqref{def-V-alpha} that $\eta_{t} \mu g_{i}(x^{t})+\langle \nabla g_{i}(x^{t}), x^{t+1}-x^{t}\rangle \leq 0$. Further considering the $\ell_g$-smoothness of $g_i$, we obtain
\[
\begin{aligned}
g_{i}(x^{t+1})  \leq\left(1-\mu \eta_{t}\right) g_{i}(x^{t})+\frac{\ell_{g}}{2} \eta_{t}^{2} \|v^t\|^2 
 \leq\frac{t+\kappa-1}{t+\kappa} g_{i}(x^{t})+\frac{\ell_g C_{1}^2}{2\mu^2(t+\kappa)^{2}}.
\end{aligned}
\]
Multiplying both sides by $(t+\kappa+1)$ and noting $(t+\kappa+1) \leq 2(t+\kappa)$, we obtain
\begin{equation}\label{f-recursion}
(t+\kappa+1)g_{i}(x^{t+1})  \leq \frac{(t+\kappa)^2-1}{t+\kappa} g_{i}(x^{t})+\frac{\ell_g C_{1}^2}{\mu^2(t+\kappa)}  \leq (t+\kappa)g_{i}(x^{t})+\frac{\ell_g C_{1}^2}{\mu^2(t+\kappa)}.
\end{equation}
{For any $s=0,\ldots,t$, if $g_i(x^s)>0$, then  \eqref{f-recursion} holds with $t$ replaced by such $s$.
Now, suppose $g_i(x^s)>0$ for all  $s=0,\ldots,t$. Then,  by replacing the running index $t$ in \eqref{f-recursion} with $s$ and taking a sum  over $s=0, \ldots, t$, we obtain}
\[
{
(t+\kappa+1)g_i(x^{t+1})
\leq \kappa g_i(x^0) 
+\frac{\ell_g C_1^2}{\mu^2}\sum_{s=0}^{t}\frac{1}{s+\kappa}
\leq \kappa[g_i(x^0)]_+
+\frac{\ell_g C_1^2}{\mu^2}\bigl(1+\log(t+1)\bigr).
}
\]
{Otherwise, there must be some $ s\in \{0,1,\ldots, t-1\}$ such that $g_i(x^s)\leq 0$. In this case, define}
\begin{equation}\label{def-theta-i-t}
    \theta_{i}(t) := \max\{ s\in \{0,1,\ldots, t-1\} \mid g_i(x^s)\leq 0\}.
\end{equation}
{By the definition of  $ \theta_{i}(\cdot)$ and recalling $g_i(x^t)>0$}, we have $0\leq \theta_i(t+1)<t$, $g_i(x^{\theta_i(t+1)})\leq 0$ and $g_i(x^{\theta_i(t+1)+w})>0$ for $w=1,\ldots, t-\theta_i(t+1)$. This implies $i \notin I_{x^{\theta_i(t+1)}}$ and $i \in I_{x^{s}}$ for any $s=\theta_i(t+1) + 1, \ldots, t$. Note that \eqref{f-recursion} holds for any $t$ as long as $i \in I_{x^t}$. Hence, we have 
\[
(s+\kappa+1)g_{i}(x^{s+1})-(s+\kappa)g_{i}(x^{s}) \leq \frac{\ell_g C_{1}^2}{\mu^2(s+\kappa)}, 
\quad s=\theta_i(t+1) + 1,\ldots,t.
\]
Summing the above inequality from $s=\theta_i(t+1) + 1$ to $t$ yields
\begin{equation*}
\begin{aligned}
(t+\kappa+1)g_{i}(x^{t+1}) & \leq  (\theta_i(t+1)+1+\kappa)g_{i}(x^{\theta_i(t+1)+1})+ \sum_{s=\theta_i(t+1)+1}^{t} \frac{\ell_g C_{1}^2}{\mu^2(s+\kappa)}\\
& \leq \frac{C_{1} L_{g}(\theta_i(t+1)+1+\kappa)}{\mu(\theta_i(t+1)+\kappa)} +  \frac{\ell_g C_{1}^2}{\mu^2}\sum_{s=1}^{t} \frac{1}{s}\\
& \leq \frac{2 C_{1} L_{g}}{\mu} +  \frac{\ell_g C_{1}^2}{\mu^2}(1+\log t) = \frac{2C_{1}}{\mu}\left(L_g+\frac{\ell_g C_{1}}{2\mu}\right) + \frac{\ell_g C_{1}^2\log t}{\mu^2},
\end{aligned}
\end{equation*}
where the second inequality uses \eqref{f-case-1}, since we have $i \notin I_{x^{\theta_i(t+1)}}$. 
{Combining the two possibilities above}, we obtain
\begin{equation} \label{f-case-2}
g_{i}(x^{t+1}) \leq {\frac{\kappa[g_i(x^0)]_+}{t+\kappa+1}+{}}\frac{2C_{1}}{\mu(t+\kappa+1)}\left(L_g+\frac{\ell_gC_{1}}{2\mu}\right) + \frac{\ell_g C_{1}^2{\log(t+1)}}{\mu^2(t+\kappa+1)}, \quad \forall i \in I_{x^{t}}. 
\end{equation}
\end{itemize}
Since the term on the far right-hand side of \eqref{f-case-1} is smaller than the right-hand side of \eqref{f-case-2}, we conclude that 
\eqref{f-case-2} holds for all $i \in [m]$. This completes the proof of the theorem.
\Halmos \endproof

{
\begin{remark} [About the choice of $\alpha$ in Algorithm \ref{alg:cgm-f}] \label{remark:alpha1}
The choice $\alpha=\mu$ in Algorithm \ref{alg:cgm-f} and Theorem \ref{thom:f} is made to simplify the statement.
In fact, the convergence guarantees can be extended to any  $\alpha\in(0,\mu]$, provided that the step sizes are rescaled properly. Specifically, for any integer $T$ satisfying $T\ge(\ell_f/\alpha)\log T$, one may take $\eta_t=\log T/(\alpha T)$ under the constant step size setting,   or $\eta_t=1/(\alpha t+\ell_f)$ under the varying step size setting.
For either choice, since $0<\alpha\leq\mu$ and $\eta_t\leq1/\ell_f$, the bounds given in \eqref{Cv-w-u}–\eqref{cor:bound_x}, as well as the function‑value recursion \eqref{fn-value-shrink}, remain valid.
By following an analogous proof, we can derive the same convergence results as in Theorem \ref{thom:f}, with only some constants in the resulting bounds modified.
\end{remark}
}

\section{CGM for Strongly Monotone VIs} \label{sec:vi}

This section focuses on addressing the VI problem \eqref{pro-vi}, for which $F$ is strongly monotone.
{We improve the convergence rate for constraint violation reported in \cite{ZHM25} from $ \mathcal{O}\left(\frac{\gamma^2+\gamma^2\zeta(1+2/(\gamma+1))}{(T+1)^{1-2/(\gamma+1)}}\right)$ for fixed $\gamma>1$ to $
\widetilde{\mathcal{O}}(1/T)$, while retaining the $\mathcal{O}(1/T)$ rate for the Minty gap. We also establish pointwise convergence of the iterates to the unique solution.}
Throughout this section, we make the following assumptions.

\begin{assumption} \label{assmp:vi}
Assume that (i) 
the operator $F: \mathbb{R}^n\rightarrow\mathbb{R}^n$ in \eqref{pro-vi} is continuous, $\mu$-strongly monotone over $\mathbb{R}^n$, and satisfies the relaxed Lipschitz condition:
\begin{equation} \label{relax-con}
\|F(x)-F(y)\|^2\leq \ell_F^2\|x-y\|^2 + B, \quad  \forall  x,y\in \mathbb{R}^n,
\end{equation}
where $\ell_F\geq \mu$ and $B\geq 0$ are constants; 
(ii) all $g_i$'s are convex and $\ell_g$-smooth over $\mathbb{R}^n$; and (iii) the feasible set $\mathcal{C}$ is nonempty.  
\end{assumption}

Rather than assuming that $F$ is $L_F$-bounded over the entire space, we instead adopt the more relaxed condition given in \eqref{relax-con}, for which we provide the following remarks. 
\begin{enumerate}
\item[(i)] The use of weakened conditions traces back to stochastic gradient descent methods, which typically adopt boundedness conditions such as
$\mathbb{E}_{\xi}\|\nabla f(x;\xi)\|^{2} \leq B$ or $\mathbb{E}_{\xi}\|s(x; \xi)\|^{2}\leq B$,
with $s(x; \xi)$ denoting an unbiased estimator of a subgradient.
For the smooth case, Blum and Gladyshev \cite{B54, G65} proposed a relaxed condition 
$\mathbb{E}_{\xi}\|{\nabla} f(x;\xi)\|^{2} \leq L^{2}\|x\|^{2}+B$, which further implies the existence of $\hat{L}, \hat{B} >0$ such that
$\mathbb{E}_{\xi}\|{\nabla} f(x;\xi)-{\nabla} f(x^\star)\|^{2} \leq \hat{L}^{2}\|x-x^\star\|^{2}+\hat{B}$, a form analogous to that of \eqref{relax-con}.
For the nonsmooth case, \cite{G19} assumed $\mathbb{E}_{\xi}\|s(x; \xi)\|^{2} \leq B+L\left(f(x)-f(x^\star)\right)$. 
\item[(ii)] In our case, condition \eqref{relax-con} is different from \cite[Eq. (1.3)]{DQM24} and \cite[Eq. (4)]{JNT11}: the latter 
controls $\|F(x)-F(y)\|$ (rather than its square) and only requires satisfaction for $x,y \in \mathcal{X}$, where $\mathcal{X}$ is a simple convex compact set 
admitting easy projection. Additional forms of relaxed conditions are available in \cite{KR23, AYS25} and related works. 

\item[(iii)] When $B=0$, condition \eqref{relax-con} reduces to the $\ell_F$-Lipschitz continuity of $F$ over $\mathbb{R}^n$. Thus, for $B>0$, this assumption 
is less restrictive than requiring $F$ to be both $\ell_F$-Lipschitz and $L_F$-bounded. Furthermore,  it is consistent with the strong monotonicity of
$F$ when $\ell_F \geq \mu$. 
\end{enumerate}

We retain the notation $\kappa:=\ell_F/\mu\geq 1$ in this section. Our CGM for solving the strongly monotone VI problem \eqref{pro-vi} is presented in Algorithm \ref{alg:cgm-op}.

\begin{algorithm}[!hbpt]
\caption{CGM for Strongly Monotone VI (CGM-VI)}
\label{alg:cgm-op}
\begin{algorithmic}[1]
    \State Let {$x^0 \in \mathbb{R}^n$}, and set parameters: $\alpha=\mu$,  $\eta_t = \frac{1}{\mu(t+{8}\kappa^2)}$ for all $t\geq 0$.
\For{$t = 0, 1, 2, \ldots $}
    
    \State {Build the set of constraints violated at $x^t$:  $I_{x^t} = \{ i \in [m] \mid g_i(x^t) > 0 \}$. }
    \State Construct the velocity polytope 
    $ V_{\alpha}(x^t) = \{ v \in \mathbb{R}^n \mid \alpha g_i(x^t) + \nabla g_i(x^t)^\top v \leq 0, \; \forall  i \in I_{x^t} \}.$
    \State Solve the quadratic programming problem
    \begin{equation} \label{update-v}
        v^t = {\arg \min}_{v \in V_{\alpha}(x^t)}   \bigl\| v + {F(x^t)} \bigr\|^2.
    \end{equation}
    \State Update the iterate $x^{t+1} = x^{t} + \eta_t v^t$.
\EndFor
\end{algorithmic}
\end{algorithm}

Next, we present a key lemma that underpins our main convergence results.

\begin{lemma}[Boundedness of {$\|x^t-x^*\|$} and $\|v^t\|$] \label{lemma:op-bound-x-v}
Under Assumption \ref{assmp:vi} for Problem \eqref{pro-vi}, Algorithm \ref{alg:cgm-op} satisfies, for all $t\geq 0$, the following:
\vspace{-1em}
\begin{subequations} \label{op-bound-x-v}
\begin{align}
{\|x^t-x^*\|}
& \leq C_{3} :=
{\sqrt{\tilde{R}}},\label{op-bound-x-v-1}\\
\|v^t\|
& \leq C_{4} :=
{
\sqrt{
8\ell_F^2
\tilde{R} + 10\big(\|F(x^*)\|^2+B\big)},
}\label{op-bound-x-v-2}
\end{align}
\end{subequations}
with $\tilde{R}:=\max\left\{
\|x^0-x^*\|^2,\,
\frac{4\|F(x^*)\|^2}{\mu^2}
+\frac{5(\|F(x^*)\|^2+B)}{2\ell_F^2}
\right\}$.
\end{lemma}
\proof{Proof.}
For any $x\in \mathcal{C}$, by Lemma \ref{lemma:depict_veloc_1} and the definition of $v^t$, we have 
\begin{align}\label{lem41a}
\left\|v^{t}+F(x^t)\right\|^{2} 
& \leq \left\|\alpha(x-x^t)+F(x^t)\right\|^{2} = \alpha^2 \|x-x^t\|^2 +2\alpha \langle x-x^t, F(x^t)\rangle + \|F(x^t)\|^2 \nonumber \\
& = \alpha^2 \|x-x^t\|^2   +2\alpha \langle x-x^t, F(x^t)-F(x)\rangle + 2\alpha\langle x-x^t, F(x)\rangle + \|F(x^t)\|^2 \nonumber\\
& \leq (\alpha^2-2\mu\alpha)\|x-x^t\|^2 + 2\alpha\langle x-x^t, F(x)\rangle + \|F(x^t)\|^2 \nonumber\\ 
& \leq -\mu^2\|x-x^t\|^2 + 2\mu \Big(\frac{1}{2\mu}\|F(x)\|^2+\frac{\mu}{2}\|x-x^t\|^2\Big) + \|F(x^t)\|^2 \nonumber\\
& = \|F(x)\|^2 + \|F(x^t)\|^2,  
\end{align}
where the second inequality follows from the $\mu$-strong monotonicity of $F$, and the third inequality follows from 
Young's inequality and the fact that $\alpha=\mu$.
Using the inequality $\|a\|^2 \leq 2(\|a-b\|^2 + \|b\|^2)$ (which holds for any vectors $a$ and $b$ of the same dimension), we further derive
\begin{align*}
\left\|v^{t}\right\|^{2} &  \leq 2\|v^{t}+F(x^t)\|^2 + 2\|F(x^t)\|^2 \stackrel{\eqref{lem41a}}{\leq} 2\|F(x)\|^2 + 4\|F(x^t)\|^2\\
& \leq 10\|F(x)\|^2 + 8\|F(x^t)-F(x)\|^2 \stackrel{\eqref{relax-con}}{\leq} 10\|F(x)\|^2 + 8\ell_F^2\|x^t-x\|^2 + 8B\\
& \leq 8\ell_F^2\|x^t-x\|^2 + 10(\|F(x)\|^2+B).
\end{align*} 
{By letting $x=x^*\in\mathcal{C}$ in the above inequality, we obtain}
\begin{equation}\label{control-v-x}
{
\|v^{t}\|^{2}
\leq
8\ell_F^2\|x^t-x^*\|^2
+10\big(\|F(x^*)\|^2+B\big).
}
\end{equation}
{Additionally, using $\eta_t=1/(\mu(t+8\kappa^2))$ for all $t\geq0$ and $\kappa=\ell_F/\mu$, we derive
\begin{equation} \label{control-stepsize}   
\eta_t
\leq
\mu / (8\ell_F^2)  
\text{~~and~~} 
8\ell_F^2\eta_t^2\leq\mu\eta_t, \quad
\forall t\geq0.
\end{equation}
We next bound $\|x^t-x^*\|$. Since $x^*\in\mathcal{C}$, Lemma \ref{lemma:depict_veloc_2} gives $v^t+x^*-x^t\in V_{\mu}(x^t)$. The first-order optimality condition of the velocity subproblem therefore implies $\langle v^t+F(x^t),x^t-x^*\rangle\leq0$.}
{Using the $\mu$-strong monotonicity of $F$, we obtain}
\[
{
\begin{aligned}
\langle v^t,x^t-x^*\rangle
&\leq-\langle F(x^t),x^t-x^*\rangle=-\langle F(x^t)-F(x^*),x^t-x^*\rangle
-\langle F(x^*),x^t-x^*\rangle\\
&\leq-\mu\|x^t-x^*\|^2
+\|F(x^*)\|\|x^t-x^*\|.
\end{aligned}
}
\]
{Consequently, the update $x^{t+1}=x^t+\eta_tv^t$ and \eqref{control-v-x} yield}
\begin{equation}\label{op-distance-recursion}
{
\begin{aligned}
\|x^{t+1}-x^*\|^2
&=\|x^t-x^*\|^2
+2\eta_t\langle v^t,x^t-x^*\rangle
+\eta_t^2\|v^t\|^2\\
&\leq
\bigl(1-2\mu\eta_t+8\ell_F^2\eta_t^2\bigr)
\|x^t-x^*\|^2
+2\eta_t\|F(x^*)\|\|x^t-x^*\|
+10\eta_t^2\bigl(\|F(x^*)\|^2+B\bigr).
\end{aligned}
}
\end{equation}
{Moreover, combining $2\|F(x^*)\|\|x^t-x^*\|\leq
\frac{\mu}{2}\|x^t-x^*\|^2  + \frac{2}{\mu}\|F(x^*)\|^2$ with \eqref{control-stepsize}-\eqref{op-distance-recursion}, we obtain}
{
\begin{align}
\|x^{t+1}-x^*\|^2
&\leq
\left(1-\frac{\mu\eta_t}{2}\right)\|x^t-x^*\|^2
+\frac{2\|F(x^*)\|^2}{\mu}\eta_t
+10\bigl(\|F(x^*)\|^2+B\bigr)\eta_t^2 \nonumber \\
&\leq
\left(1-\frac{\mu\eta_t}{2}\right)\|x^t-x^*\|^2+\frac{\mu\eta_t}{2}
\left(
\frac{4\|F(x^*)\|^2}{\mu^2}
+\frac{5\bigl(\|F(x^*)\|^2+B\bigr)}{2\ell_F^2}
\right) \nonumber \\
&\leq
\left(1-\frac{\mu\eta_t}{2}\right)\|x^t-x^*\|^2+\frac{\mu\eta_t}{2}\tilde{R}.
\label{jy-260919}
\end{align}
Here the last ``$\leq$" uses the definition of $\tilde{R}$. 
Again,   by definition of $\tilde{R}$, we have  $\|x^0-x^*\|^2\leq \tilde{R}$.
Note that $0<\mu\eta_t/2<1$. 
Then, it follows from \eqref{jy-260919} by induction that \eqref{op-bound-x-v-1} holds for all $t\geq0$.  Finally, substituting \eqref{op-bound-x-v-1} into \eqref{control-v-x} yields \eqref{op-bound-x-v-2}.}
\Halmos \endproof

Recall that the optimality gap $\mathcal{G}(\cdot)$ is defined in Definition \ref{weak-e-sol},
$C_3$ and $C_4$ are defined in \eqref{op-bound-x-v}, and $\kappa = \ell_F/\mu$. 
Now we are ready to present our main convergence results for Algorithm \ref{alg:cgm-op}. 

\begin{theorem}[Convergence Rates] \label{thom:op}
Consider Problem \eqref{pro-vi} under Assumption \ref{assmp:vi}. Let $\{x^t\}$ be the sequence generated by 
Algorithm \ref{alg:cgm-op}. Define {$\bar{x}^{T}=\sum_{t=0}^{T-1} (t+8\kappa^2-1) x^{t} \big/ \sum_{t=0}^{T-1} (t+8\kappa^2-1)$ for any integer $T\ge 1$}. 
We then have the following convergence rate results:  
\begin{align}
\mathcal{G}(\bar{x}^{T})
&\!\leq
{
\frac{256\mu\kappa^4C_3^2}{T^2} + \frac{C_4^2}{\mu(T+16\kappa^2-3)}}
\sim\mathcal{O}\left(\frac{1}{T}\right),\label{op-op-gap}\\
g_i(\bar{x}^T)
&\!\leq
{
\frac{2}{T\!+\!16\kappa^2\!-\!3}
\Bigg[
\max\!\left\{
(8\kappa^2\!-\!1)[g_i(x^0)]_+,
\frac{C_4\tilde{L}_g}{\mu}
\right\}+
\frac{\ell_g C_4^2\log(T\!+\!1)}{2\mu^2}
\Bigg]
\!\sim\!\widetilde{\mathcal{O}}\!\left(\frac{1}{T}\right), \, i\in{[m]},}\label{op-vr-gap}
\end{align}
{where $\tilde{L}_g
:=
\max\left\{
\|\nabla g_i(x)\|:
\|x-x^*\|\leq C_3,\ i\in[m]
\right\}$.}
Additionally, for the constraint violation, we also have the following non-ergodic convergence rate for {all $t\geq 0$}:
\begin{equation}\label{op-vr-gap-nonergodic}
{
[g_i(x^{t+1})]_+
\leq
\frac{
\max\left\{
(8\kappa^2-1)[g_i(x^0)]_+,\,
\dfrac{C_4\tilde{L}_g}{\mu}
\right\}
+\dfrac{\ell_gC_4^2}{2\mu^2}
\displaystyle\sum_{r=0}^{t}\frac{1}{r+8\kappa^2}
}{
t+8\kappa^2
}
}
\sim\widetilde{\mathcal{O}}\left(\frac{1}{t}\right),
\quad i\in{[m]}.
\end{equation}
\end{theorem}

\proof{Proof.}
Using the optimality conditions of \eqref{update-v}, we obtain $\langle v^t + F(x^t), v^t-v\rangle \leq 0$ for any $v \in V_{\alpha}(x^t)$. 
Let $x \in \mathcal{C}$ be arbitrarily fixed. Then, we have $v=v^t+x-x^t \in V_{\alpha}(x^t)$ from Lemma \ref{lemma:depict_veloc_2}.
Hence, we have $\langle v^t + F(x^t), x^t-x\rangle \leq 0$, which further implies
\begin{equation}\label{thm41a}
\langle F(x^{t}),x^{t}-x\rangle  \leq \langle v^t,x-x^t\rangle = \frac{1}{\eta_t}\langle x^{t+1}-x^{t}, x-x^t\rangle 
=\frac{1}{2 \eta_{t}} \big(\|x-x^{t}\|^{2} -  \|x-x^{t+1}\|^{2} \big) + \frac{\eta_{t}}{2} \|v^{t}\|^{2}.
\end{equation}
{Moreover, it follows from the $\mu$-strong monotonicity of $F$ that}
\[
{
\begin{aligned}
\langle F(x),x^{t}-x\rangle
&\leq
\langle F(x^{t}),x^{t}-x\rangle
-\mu\|x^{t}-x\|^{2}\\
&\stackrel{\eqref{thm41a}}\leq
\left(\frac{1}{2\eta_t}-\mu\right)\|x-x^t\|^2
-\frac{1}{2\eta_t}\|x-x^{t+1}\|^2
+\frac{\eta_t}{2}\|v^t\|^2\\
&\stackrel{\eqref{op-bound-x-v-2}}\leq
\frac{\mu(t+8\kappa^2-2)}{2}\|x-x^t\|^2
-\frac{\mu(t+8\kappa^2)}{2}\|x-x^{t+1}\|^2
+\frac{C_4^2}{2\mu(t+8\kappa^2)}.
\end{aligned}
}
\]
{Multiplying this inequality by $(t+8\kappa^2-1)$ and summing it over $t=0,\ldots,T-1$ yields}
\[
{
\begin{aligned}
&\sum_{t=0}^{T-1}
(t+8\kappa^2\!-\!1)
\langle F(x),x^t-x\rangle\\
\leq\, & 
\frac{\mu}{2}
\big[
(8\kappa^2\!-\!2)(8\kappa^2\!-\!1)\|x\!-\!x^0\|^2
-(T+8\kappa^2\!-\!2)(T+8\kappa^2\!-\!1)\|x\!-\!x^T\|^2
\big]
+\frac{C_4^2}{2\mu}
\bigg(
T\!-\!\sum_{t=0}^{T-1}\frac{1}{t+8\kappa^2}
\bigg).
\end{aligned}
}
\]
{Since  $(T+8\kappa^2-2)(T+8\kappa^2-1)
-(8\kappa^2-2)(8\kappa^2-1)
=T(T+16\kappa^2-3)>0$, then from}
{
\[
\begin{aligned}
&(8\kappa^2-2)(8\kappa^2-1)\|x-x^0\|^2
-(T+8\kappa^2-2)(T+8\kappa^2-1)\|x-x^T\|^2\\
&=
-T(T+16\kappa^2-3)
\left\|
x-
\frac{
(T+8\kappa^2-2)(T+8\kappa^2-1)x^T
-(8\kappa^2-2)(8\kappa^2-1)x^0
}{
T(T+16\kappa^2-3)
}
\right\|^2\\
&\quad+
\frac{
(8\kappa^2-2)(8\kappa^2-1)
(T+8\kappa^2-2)(T+8\kappa^2-1)
}{
T(T+16\kappa^2-3)
}
\|x^T-x^0\|^2,
\end{aligned}
\]
}
{we further have}
\[
{
\begin{aligned}
&\sup_{x\in\mathcal{C}}
\left\{
(8\kappa^2-2)(8\kappa^2-1)\|x-x^0\|^2
-(T+8\kappa^2-2)(T+8\kappa^2-1)\|x-x^T\|^2
\right\}\\
\leq \, &
\frac{
(8\kappa^2-2)(8\kappa^2-1)
(T+8\kappa^2-2)(T+8\kappa^2-1)
}{
T(T+16\kappa^2-3)
}
\|x^T-x^0\|^2 \\
\leq \, & \frac{64\kappa^4(T+16\kappa^2-3)^2}{T(T+16\kappa^2-3)}\|x^T-x^0\|^2 = \frac{64\kappa^4(T+16\kappa^2-3)}{T}\|x^T-x^0\|^2 \\
\leq \, & 256\kappa^4 C_3^2(T+16\kappa^2-3) / T, 
\end{aligned}
}
\]
{where the last ``$\leq$'' uses $\|x^T-x^0\|\leq\|x^T-x^*\|+\|x^0-x^*\|\leq 2C_3$ from \eqref{op-bound-x-v-1}.}
{Moreover, using 
$\sum_{t=0}^{T-1}(t+8\kappa^2-1)=T(T+16\kappa^2-3)/2$, taking the supremum over $x\in\mathcal{C}$ in
the summed inequality, and dividing by
$T(T+16\kappa^2-3)/2$, we derive}
\[
{
\begin{aligned}
\mathcal{G}(\bar{x}^T)
&\leq
\frac{2}{T(T+16\kappa^2\!-\!3)}
\left[
\frac{\mu}{2}
\frac{256\kappa^4 C_3^2(T\!+\!16\kappa^2\!-\!3)}{T}
+
\frac{TC_4^2}{2\mu}
\right]\!=\!
\frac{256\mu\kappa^4C_3^2}{T^2} \!+\! \frac{C_4^2}{\mu(T+16\kappa^2\!-\!3)},
\end{aligned}
}
\]
{which proves \eqref{op-op-gap}.}
To prove \eqref{op-vr-gap} and \eqref{op-vr-gap-nonergodic}, we consider the following two cases:
\begin{itemize}
\item[(a)] If $i\notin I_{x^t}$, then $g_i(x^t)\leq0$. Using the convexity of $g_i$, the definition of $\tilde{L}_g$, and \eqref{op-bound-x-v-2}, we obtain
\begin{equation}\label{op-case-1}
{
\begin{aligned}
\relax [g_i(x^{t+1})]_+
&\leq
\eta_t\|\nabla g_i(x^{t+1})\|\|v^t\|\leq
\frac{C_4\tilde{L}_g}{\mu(t+8\kappa^2)}.
\end{aligned}
}
\end{equation}
\item[(b)] If $i\in I_{x^t}$, then $g_i(x^t)>0$. The definition of $V_\mu(x^t)$ and the $\ell_g$-smoothness of $g_i$ give
\[
{
\begin{aligned}
g_i(x^{t+1})
&\leq
\left(1-\mu\eta_t\right)g_i(x^t)
+\frac{\ell_g}{2}\eta_t^2\|v^t\|^2\leq
\left(1-\frac{1}{t+8\kappa^2}\right)g_i(x^t)
+\frac{\ell_gC_4^2}{2\mu^2(t+8\kappa^2)^2}.
\end{aligned}
}
\]
{Multiplying this inequality by $(t+8\kappa^2)$ yields}
\begin{equation}\label{op-recursion}
{
(t+8\kappa^2)[g_i(x^{t+1})]_+
\leq
(t+8\kappa^2-1)[g_i(x^t)]_+
+\frac{\ell_gC_4^2}{2\mu^2(t+8\kappa^2)}.
}
\end{equation}
{
For any $s=0,\ldots,t$, if $g_i(x^s)>0$, then $i\in I_{x^s}$ and hence \eqref{op-recursion} holds with $t$ replaced by such $s$.
Now, suppose $g_i(x^s)>0$ for all  $s=0,\ldots,t$. 
Then,  by replacing the running index $t$ in \eqref{op-recursion} with $s$ and taking a sum  over $s=0, \ldots, t$, we obtain
\begin{equation} \label{po-1}
(t+8\kappa^2)[g_i(x^{t+1})]_+
\leq
(8\kappa^2-1)[g_i(x^0)]_+ +
\frac{\ell_gC_4^2}{2\mu^2}
\sum_{s=0}^{t}\frac{1}{s+8\kappa^2}.
\end{equation}
Otherwise, there must be some $ s\in \{0,1,\ldots, t-1\}$ such that $g_i(x^s)\leq 0$.  We
 inherit the definition of $\theta_i(\cdot)$ from \eqref{def-theta-i-t}. 
Since $g_i(x^t)>0$, we have $0\leq \theta_i(t+1)<t$, $g_i(x^{\theta_i(t+1)})\leq 0$ and $g_i(x^{\theta_i(t+1)+w})>0$ for $w=1,\ldots, t-\theta_i(t+1)$. This implies $i \notin I_{x^{\theta_i(t+1)}}$ and $i \in I_{x^{s}}$ for any $s=\theta_i(t+1) + 1, \ldots, t$. 
Note that \eqref{op-recursion} holds for any $t$ such that $i \in I_{x^t}$. Therefore, we have
\[
(s+8\kappa^2)[g_i(x^{s+1})]_+
\leq
(s+8\kappa^2-1)[g_i(x^s)]_+
+\frac{\ell_gC_4^2}{2\mu^2(s+8\kappa^2)}, \quad s=\theta_i(t+1)+1,\ldots,t.
\]
}
Summing the above inequality for $s=\theta_i(t+1)+1,\ldots,t$ yields
{
\begin{equation}
\begin{aligned} \label{po-2}
(t+8\kappa^2)[g_i(x^{t+1})]_+
&\leq
\bigl(\theta_i(t+1)+8\kappa^2\bigr)
\bigl[g_i\bigl(x^{\theta_i(t+1)+1}\bigr)\bigr]_+ +
\frac{\ell_gC_4^2}{2\mu^2}
\sum_{s=\theta_i(t+1)+1}^{t}
\frac{1}{s+8\kappa^2}\\
&\leq 
\frac{C_4\tilde{L}_g}{\mu}
+
\frac{\ell_gC_4^2}{2\mu^2}
\sum_{s=\theta_i(t+1)+1}^{t}
\frac{1}{s+8\kappa^2},
\end{aligned}
\end{equation}
}
where the second inequality follows from \eqref{op-case-1}, since $i \notin I_{x^{\theta_i(t+1)}}$.
\end{itemize}

{
Combining  all these possibilities considered in \eqref{op-case-1}, \eqref{po-1} and \eqref{po-2}, and enlarging the summation range yields}
\[
{
\begin{aligned}
(t+8\kappa^2)[g_i(x^{t+1})]_+
\leq
\max\left\{
(8\kappa^2-1)[g_i(x^0)]_+,\,
\frac{C_4\tilde{L}_g}{\mu}
\right\}+
\frac{\ell_gC_4^2}{2\mu^2}
\sum_{s=0}^{t}\frac{1}{s+8\kappa^2} \sim\widetilde{\mathcal{O}}(1),\\
\end{aligned}
}
\]
{since 
$\sum_{s=0}^{t}\frac{1}{s+8\kappa^2}\leq
\int_{-1}^{t}\frac{1}{s+8\kappa^2}\,\mathrm{d}s=
\log\left(
\frac{t+8\kappa^2}{8\kappa^2-1}
\right)
\leq
\log(t+1)$ for $t\geq 1$.
}
{Thus, dividing both sides by $(t+8\kappa^2)$ proves
\eqref{op-vr-gap-nonergodic} for all
$i\in[m]$.}

{Finally, we prove the ergodic constraint-violation bound \eqref{op-vr-gap}.  
For any $s \in [T-1]$, it follows from setting $t=s-1$ in \eqref{op-vr-gap-nonergodic} and further upper bounding  $\sum_{r=0}^{s-1}\frac{1}{r+8\kappa^2}$ by $\sum_{r=0}^{T-1}\frac{1}{r+8\kappa^2}$ that 
\[
(s+8\kappa^2-1)[g_i(x^s)]_+
\leq
\max\left\{
(8\kappa^2-1)[g_i(x^0)]_+,\,
\frac{C_4\tilde{L}_g}{\mu} 
\right\}+
\frac{\ell_gC_4^2}{2\mu^2}
\sum_{r=0}^{T-1}\frac{1}{r+8\kappa^2}, \; \forall s \in [T-1].
\]
For $s=0$, the same upper bound follows directly from the
first term in the maximum above. Since all weights
$(s+8\kappa^2-1)$, $s=0,\ldots, T-1$, are positive, the convexity of $g_i$ implies}
\[
{
\begin{aligned}
\relax [g_i(\bar{x}^T)]_+
&\leq
\frac{
\sum_{s=0}^{T-1}(s+8\kappa^2-1)[g_i(x^s)]_+
}{
\sum_{s=0}^{T-1}(s+8\kappa^2-1)
}\\
&\leq
\frac{2}{T+16\kappa^2-3}
\left[
\max\left\{
(8\kappa^2-1)[g_i(x^0)]_+,\,
\frac{C_4\tilde{L}_g}{\mu} 
\right\}
+
\frac{\ell_gC_4^2}{2\mu^2}
\sum_{r=0}^{T-1}\frac{1}{r+8\kappa^2}
\right],
\end{aligned}
}
\]
{where we use
$\sum_{s=0}^{T-1}(s+8\kappa^2-1)
=T(T+16\kappa^2-3)/2$ and $
\sum_{s=0}^{T-1}\frac{1}{s+8\kappa^2}
\leq
\log(T+1)$.
}
{This establishes \eqref{op-vr-gap}, and the proof is completed.}
\Halmos \endproof

{
\begin{remark} 
 [About the choice of $\alpha$ in Algorithm \ref{alg:cgm-op}]  
Similar to Remark \ref{remark:alpha1}, the choice of $\alpha=\mu$ in Algorithm~\ref{alg:cgm-op} can also be extended to $\alpha\in(0,\mu]$ by letting $\eta_t=1/(\alpha t+8\mu\kappa^2)$ and replacing the weight $(t+8\kappa^2-1)$ in $\bar{x}^T$ by $(t+8\mu\kappa^2/\alpha-1)$. Since $\eta_t\leq\mu/(8\ell_F^2)$ and $\alpha/(2\mu-\alpha)\leq1$, the bounds given in Lemma~\ref{lemma:op-bound-x-v} remain valid. Utilizing the adjusted step size and averaging weights and following an analogous proof, we can derive similar convergence results as in Theorem~\ref{thom:op} for both the ergodic Minty gap and constraint violations.
\end{remark}
}

{Next, we establish pointwise convergence for Algorithm \ref{alg:cgm-op} under strong monotonicity.}

\begin{theorem}[Pointwise Convergence]
\label{thom:op-last-iterate}
Consider Problem \eqref{pro-vi} under Assumption \ref{assmp:vi}. Let $\{x^t\}$ be the sequence generated by Algorithm \ref{alg:cgm-op}. Then, we have $\lim_{t\to\infty}x^t=x^*$.
\end{theorem}
\proof{Proof.}
By Lemma \ref{lemma:op-bound-x-v}, the sequence $\{x^t\}$ is bounded. Moreover, it follows from \eqref{op-vr-gap-nonergodic} that  $[g_i(x^t)]_+\to0$ for all $i\in[m]$.
We first show that
\begin{equation}
\label{op-last-iterate-boundary-term}
\lim_{t\to\infty} \left[
-\left\langle F(x^{*}),x^t-x^{*}\right\rangle
\right]_+
=0.
\end{equation}
Suppose for contradiction that there exist $\varepsilon>0$ and a subsequence $\{x^{t_j}\}$ such that $\left[
-\left\langle F(x^*),x^{t_j}-x^*\right\rangle
\right]_+
\geq\varepsilon$
for all $j$. By passing to a further subsequence if necessary, we may assume without loss of generality that 
$\{x^{t_j}\}$ converges to some $\widehat{x}$. 
Hence, we obtain $\left[
-\left\langle F(x^*), \widehat{x}-x^*\right\rangle
\right]_+
\geq\varepsilon$.
On the other hand, the nonergodic feasibility estimate \eqref{op-vr-gap-nonergodic} and the continuity of $g_i$ imply $g_i(\widehat{x})\leq0$ for all $i\in[m]$, and hence $\widehat{x}\in\mathcal{C}$. Since $x^*$ is the unique solution of Problem \eqref{pro-vi}, we have $\left\langle F(x^*),\widehat{x}-x^*\right\rangle\geq0$, contradicting $\left[
-\left\langle F(x^*), \widehat{x}-x^*\right\rangle
\right]_+
\geq\varepsilon$. Hence,  \eqref{op-last-iterate-boundary-term} holds.
Also, from the proof of \eqref{op-distance-recursion}, it is easy to derive
\begin{equation}
\begin{aligned} \label{op-last-iterate-recursion}
\|x^{t+1}\!-\!x^*\|^2
&\leq
\big(1\!-\!2\mu\eta_t\!+\!8\ell_F^2\eta_t^2\big)
\|x^t\!-\!x^*\|^2+2\eta_t
\left[
-\left\langle F(x^*),x^t\!-\!x^*\right\rangle
\right]_+
+10\eta_t^2\bigl(\|F(x^*)\|^2+B\bigr)\\
& \leq 
(1-\mu\eta_t)\|x^t-x^{*}\|^2
+ 2\eta_t
\Big(
\left[
-\left\langle F(x^{*}),x^t-x^{*}\right\rangle
\right]_+
+
5\eta_t\bigl(\|F(x^*)\|^2+B\bigr)
\Big),
\end{aligned}
\end{equation}
where the second inequality uses $8\ell_F^2\eta_t^2\leq\mu\eta_t$.
From \eqref{op-last-iterate-boundary-term} and $\lim_{t\to\infty} \eta_t=0$, for any fixed $\varepsilon>0$, there exists an integer $t_0$ such that   
\[
\left[
-\left\langle F(x^{*}),x^t-x^{*}\right\rangle
\right]_+
+
5\eta_t\bigl(\|F(x^{*})\|^2+B\bigr)
\leq \mu \varepsilon / 2, \quad \forall t\geq t_0. 
\]
Plugging the above inequality into \eqref{op-last-iterate-recursion}, we obtain $\|x^{t+1}-x^{*}\|^2 \leq (1-\mu\eta_t)\|x^t-x^{*}\|^2 +\mu\eta_t\varepsilon$, or equivalently, $\|x^{t+1}-x^{*}\|^2-\varepsilon \leq (1-\mu\eta_t)(\|x^t-x^{*}\|^2 -\varepsilon)$ for all $t\geq t_0$.
Applying this inequality recursively, we obtain
$    \|x^t-x^*\|^2
    \leq
    \varepsilon
    +
    \left(
        \|x^{t_0}-x^*\|^2-\varepsilon
    \right)\prod\nolimits_{s=t_0}^{t-1}
    \left(1-\mu\eta_s\right).
$ 
Further note that $\mu\eta_s \in [0,1)$, $ \sum_{s=t_0}^{\infty}\mu\eta_s=\sum_{s=t_0}^{\infty}\frac{1}{s+8\kappa^2}=\infty$, and 
$\prod_{s=t_0}^{t-1}\left(1-\mu\eta_s\right)\leq \exp\bigl( -\sum_{s=t_0}^{t-1}\mu\eta_s \bigr)$, 
we derive $\lim_{t\to\infty}\prod_{s=t_0}^{t-1}    \left(1-\mu\eta_s\right)=0$.  
The conclusion $\lim_{t\to\infty}x^t=x^*$ then follows from the fact that $\varepsilon>0$ is arbitrary. 
\Halmos \endproof

{
\begin{remark}[Why analyze optimization and VIs separately]
\label{remark:min-vi-separate}
Note that both Algorithm \ref{alg:cgm-f} for strongly convex optimization and Algorithm \ref{alg:cgm-op} for strongly monotone VI
employ the same velocity‑polytope mechanism and are largely similar.
Even so, it is more appropriate to analyze the two settings separately.
Consider Algorithm \ref{alg:cgm-f} with varying step size (case (ii) in Theorem \ref{thom:f}) and Algorithm \ref{alg:cgm-op}. Both algorithms adopt step sizes of the form $\eta_t=1/(\mu(t+c))$ with some $c>0$. However, the former takes $c=\kappa$, while the latter exhibits a different dependence on the order of $\kappa$ and takes $c = 8\kappa^2$.
%
%
To see the necessity of such a difference, consider VI problem defined by 
\begin{equation} \label{def:operator-J}
F(x)=\mu  
\begin{pmatrix}
1 & -\sqrt{\kappa^2-1}\\
\sqrt{\kappa^2-1} & 1
\end{pmatrix}x, \;\; \kappa \geq 1, \text{~~and~~} \mathcal{C} = \mathbb{R}^2.    
\end{equation}
This operator $F$ is $\mu$-strongly monotone with Lipschitz constant $\ell_F = \mu\kappa$, and the unique solution of this VI problem is $x^* = 0$. 
For  $\eta_t=1/(\mu(t+c))$ with any $c>0$, the CGM recursion satisfies
\begin{equation}\label{jy-260920a}
\|x^{t+1}\|^2
    =
    \left(
        1-\frac{2}{t+c}
        +\frac{\kappa^2}{(t+c)^2}
    \right)\|x^t\|^2, \quad \forall t\geq 0.
\end{equation}
Let $\kappa$ be an integer satisfying $\kappa\geq 8$. Then, by setting $c=\kappa$, it follows from \eqref{jy-260920a} that $\|x^{t+1}\|^2 \geq \frac{9}{8}\|x^t\|^2$ for $t = 0, 1, \ldots, \kappa-1$. Hence, $\|x^\kappa\|\geq\left({9}/{8}\right)^{\kappa/2}\|x^0\|$.
It is also worth noting that the step size adopted in \cite{ZHM25} for this problem is $\eta_t=1/(\mu(t+1))$, and the resulting amplification effect persists and becomes even more severe.
For visualization purposes, we plot the iteration trajectories in Figure \ref{fig:rotational-trajectories} for both $c=1$ and $c=\kappa$, with $\mu=1$, $\kappa = 100$ and $x^0=(1,0)^\top$.
It is evident from Figure \ref{fig:rotational-trajectories} that transient amplification occurs for both cases, and the amplification is more pronounced for $c=1$.
In contrast, by setting $c=8\kappa^2$ as in Algorithm \ref{alg:cgm-op}, \eqref{jy-260920a} yields
\[
\|x^{t+1}\|^2
=
\left(
1-\frac{2}{t+8\kappa^2}
+\frac{\kappa^2}{(t+8\kappa^2)^2}
\right)\|x^t\|^2\leq 
\left(
1-\frac{15}{8(t+8\kappa^2)}
\right)\|x^t\|^2
<
\|x^t\|^2,
\]
which renders every iteration contractive and guarantees $x^t\to x^*$ as $t\to+\infty$.
Note that the coefficient matrix defining $F$ in \eqref{def:operator-J} possesses a skew‑symmetric structure and thus $F$ is not the gradient operator of any function. This leads to the key difference between optimization and VI problems in step size choices and necessitates our separate analysis.

\begin{figure}[!htbp]
    \centering
    \begin{minipage}[t]{0.4\textwidth}
        \centering
        \includegraphics[width=\linewidth]{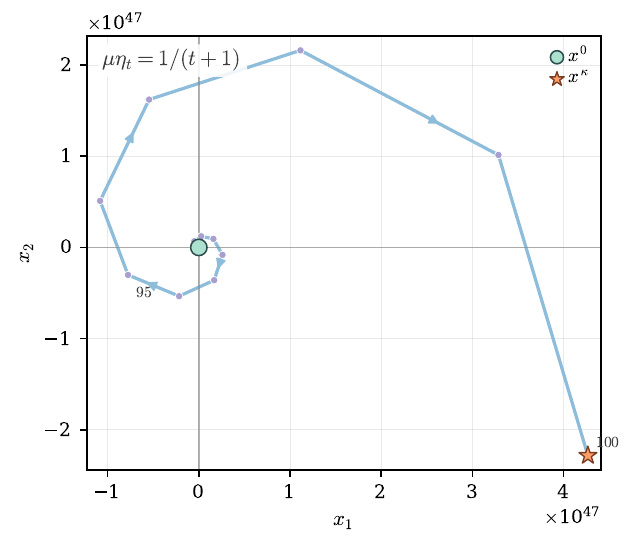}
    \end{minipage}
    \begin{minipage}[t]{0.4\textwidth}
        \centering
        \includegraphics[width=\linewidth]{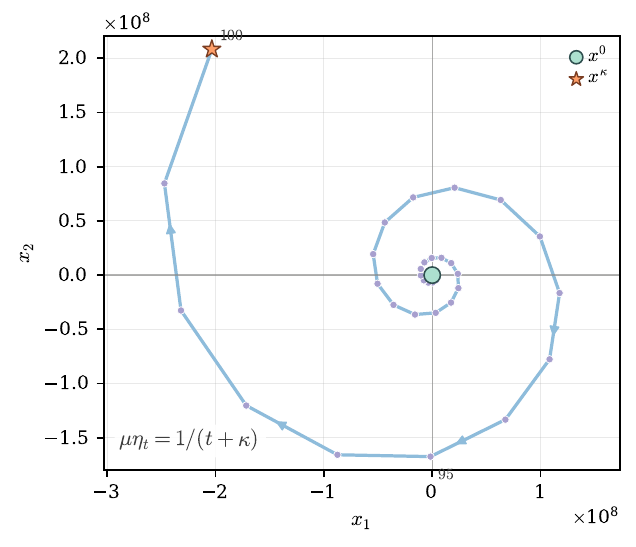}
    \end{minipage}
    \caption{   
    {Transient amplification: trajectories of the iterates $\{x^t\}_{t=0}^{\kappa}$ generated by CGM applied to \eqref{def:operator-J} with $\mu=1$, $\kappa=100$, $\eta_t= 1/(\mu(t+c))$, and $x^0=(1,0)^\top$. Left: $c=1$; Right: $c=\kappa$.
    }}
    \label{fig:rotational-trajectories}
\end{figure} 
\end{remark}
}

\section{CGM for Monotone VIs} \label{sec:monotone-vi-expanding-ball}  
{
In this section, we propose a CGM variant for solving the VI problem \eqref{pro-vi}, where $F$ is merely monotone (not strongly monotone). Specifically, we impose the following assumptions.

\begin{assumption}
\label{assmp:monotone-vi-expanding-ball}
Assume that (i) $F:\mathbb{R}^n\to\mathbb{R}^n$ is continuous and monotone, and there exists a constant $L_F>0$ such that $\|F(x)\|\leq L_F$ for all $x\in\mathbb{R}^n$; (ii) each function $g_i:\mathbb{R}^n\to\mathbb{R}$ is convex, $L_g$-Lipschitz continuous, and $\ell_g$-smooth; and (iii) the feasible set $\mathcal{C}$ is nonempty.
\end{assumption}
}

{ 
The proposed CGM variant uses a predetermined total iteration horizon, denoted by $T$, and employs a horizon‑dependent expanding‑ball defined via an auxiliary constraint $g_{m+1}(x)\leq 0$. 
The algorithm is presented in Algorithm \ref{alg:cgm-op-expanding-ball}.
}

\begin{algorithm}[!htbp]
\caption{{CGM for Monotone VI (CGM-MVI)}}
\label{alg:cgm-op-expanding-ball}
\begin{algorithmic}[1]
{ 
    \State Let $x^0\in\mathbb{R}^n$ and $D_0\geq\max\{1,\|x^0\|,\operatorname{dist}(0,\mathcal{C})\}$. Set the total iteration horizon $T$, and let $D_T=D_0+\log(T+1)$, $\alpha_T=L_F/D_T$, and $\eta_T = \sqrt{D_T/T} \big/(10L_F)$. Define the auxiliary constraint function $g_{m+1}(x):=\|x\|^2-D_T^2$.
    \For{$t=0,1,\ldots,T-1$}
        \State Build the set of constraints violated at $x^t$: $I_{x^t} = \{ i \in [m+1] \mid g_i(x^t) > 0 \}$.
        \State Construct the velocity polytope
        $V_{\alpha_T}(x^t) = \left\{v\in\mathbb{R}^n: \alpha_Tg_i(x^t) + \nabla g_i(x^t)^\top v \leq 0, \; \forall i\in I_{x^t} \right\}.$
        \State Solve the quadratic programming problem
            $v^t={\arg\min}_{v\in V_{\alpha_T}(x^t)} \bigl\|v+F(x^t)\bigr\|^2$. 

        \State Update the iterate $x^{t+1}=x^t+\eta_Tv^t$.
    \EndFor
}
\end{algorithmic}
\end{algorithm}

{
\begin{remark} 
We present the following remarks on Algorithm \ref{alg:cgm-op-expanding-ball}.
\begin{enumerate}
    \item[(i)] Unlike the CGM for monotone VIs in \cite{ZHM25}, we do not require the initial point $x^0\in\mathbb{R}^n$ to be feasible, nor do we require the feasible set $\mathcal{C}$ to be bounded.

    \item[(ii)] The auxiliary constraint $g_{m+1}(x)=\|x\|^2-D_T^2\le 0$ describes a Euclidean ball constraint with 
    a horizon‑dependent radius $D_T$ determined by the iteration budget $T$. In contrast, the radius $\widehat{D}$ in \cite{ZHM25} is a fixed a priori bound arising from the boundedness assumption on $\mathcal{C}$.

    \item[(iii)] $V_{\alpha_T}(x^t)$ is guaranteed to be nonempty. Indeed, since $\mathcal{C}$ is nonempty, closed, and convex, there exists a unique $x^{\mathrm{min}}\in\mathcal{C}$ such that $\|x^{\mathrm{min}}\| = \operatorname{dist}(0,\mathcal{C})$. By the definitions of $D_0$ and $D_T$, $g_{m+1}(x^{\mathrm{min}})\leq 0$. Hence, $V_{\alpha_T}(x^t)\neq \varnothing$, because  Lemma~\ref{lemma:depict_veloc_1} implies $\alpha_T(x^{\mathrm{min}}-x^t)\in V_{\alpha_T}(x^t)$.

\end{enumerate}
\end{remark}
}

{
Let $R\geq\operatorname{dist}(0,\mathcal{C})$. 
Given the potential unboundedness of $\mathcal{C}$, in this section we adopt the following restricted Minty gap as the optimality measure:
\begin{equation}
\label{def:restricted-minty-gap}
\mathcal{G}_R(y)
:=
\max_{x\in\mathcal{C},\;\|x\|\leq R}
\left\langle F(x),y-x\right\rangle.
\end{equation}
}

{
\begin{lemma}[Boundedness of $\|x^t\|$ and $\|v^t\|$]
\label{lemma:expanding-ball-bounds}
Under Assumption \ref{assmp:monotone-vi-expanding-ball} for Problem \eqref{pro-vi}, 
the sequences $\{x^t\}$ and $\{v^t\}$ generated by Algorithm \ref{alg:cgm-op-expanding-ball} satisfy
$\|x^t\| \leq 2D_T$ for $t=0,1,\ldots,T$, and 
$\|v^t\| \leq 5L_F$ for $t=0,1,\ldots,T-1$.
\end{lemma}

\proof{Proof.}
Recall that
there exists a unique $x^{\mathrm{min}}\in\mathcal{C}$ such that $\|x^{\mathrm{min}}\| = \operatorname{dist}(0,\mathcal{C})$.
From the optimality of $v^t$ and $\alpha_T(x^{\mathrm{min}}-x^t)\in V_{\alpha_T}(x^t)$, we have $\|v^t+F(x^t)\|\leq \|\alpha_T(x^{\mathrm{min}}-x^t)+F(x^t)\|$.
Using $\|F(x^t)\|\leq L_F$ and $\|x^{\mathrm{min}}\|\leq D_T$, we obtain
\begin{equation}
\label{bound:expanding-ball-velocity-direct}
\begin{aligned}
    \|v^t\|
    &\leq
    \|v^t+F(x^t)\|+\|F(x^t)\|
    \leq
    \alpha_T\|x^t-x^{\mathrm{min}}\|+2L_F\\
    &\leq
    \alpha_T(\|x^t\|+D_T)+2L_F=
    L_F\left( \|x^t\| / D_T+3\right),
\end{aligned}
\end{equation}
where the last ``$=$" uses $\alpha_T=L_F/D_T$.
We next prove $\|x^t\|\leq2D_T$ for $t=0,\ldots,T$ by induction. First, from the definitions of $D_0$ and $D_T$, we have $\|x^0\|\leq D_0\leq D_T$. Suppose $\|x^t\|\leq2D_T$ for some $t\in\{0,\ldots,T-1\}$. We consider the following two cases:
(i) $\|x^t\|\leq D_T$, or equivalently, $g_{m+1}(x^t)\leq 0$; and (ii) $D_T< \|x^t\|\leq 2D_T$, which implies $g_{m+1}(x^t)> 0$.

For case (i), using the update rule $x^{t+1} = x^t + \eta_T v^t$, we derive
\[
    \|x^{t+1}\|
    \leq \|x^t\| + \eta_T\|v^t\| \leq \|x^t\| + 
    \frac{1}{10}\sqrt{D_T/T}
    \left(\|x^t\|/D_T+3\right) 
    \leq 7D_T/5 < 2 D_T,
\]
where the second ``$\leq$" uses $\eta_T= \sqrt{D_T/T} \big/ (10L_F)$ and \eqref{bound:expanding-ball-velocity-direct}.
On the other hand, for case (ii) we have $g_{m+1}(x^t) > 0$ and thus $(m+1) \in I_{x^t}$. 
By the definitions of $g_{m+1}$, $V_{\alpha_T}(x^t)$ and $v^t \in V_{\alpha_T}(x^t)$, we have  
\begin{equation} \label{def-V-alpha-expanding-ball}
    \alpha_T(\|x^t\|^2-D_T^2)
    +2\left\langle x^t,v^t\right\rangle
    \leq0.
\end{equation}
Using the update rule $x^{t+1}=x^t+\eta_Tv^t$ again, we derive
\begin{align*}
\|x^{t+1}\|^2
&=
\|x^t\|^2 + \eta_T^2\|v^t\|^2
+2\eta_T\langle x^t,v^t\rangle
\stackrel{\eqref{def-V-alpha-expanding-ball}}{\leq}
\|x^t\|^2 + \eta_T^2\|v^t\|^2
-\alpha_T\eta_T
\left(\|x^t\|^2-D_T^2\right)\\
&\stackrel{\eqref{bound:expanding-ball-velocity-direct}}{\leq}
\left(1-\alpha_T\eta_T\right)\|x^t\|^2
+\eta_T^2L_F^2
\left(
\|x^t\| / D_T +3
\right)^2
+\alpha_T\eta_TD_T^2
\nonumber\\
& \leq 
4\left(1-\alpha_T\eta_T\right)D_T^2
+ 25 \eta_T^2L_F^2
+ \alpha_T\eta_TD_T^2
\nonumber\\
& = 
D_T^2
\left(
4 +\frac{1}{4 D_T T}
-\frac{3}{10\sqrt{D_TT}}
\right)\leq 4D_T^2,
\end{align*}
where we use $\|x^t\|\leq 2D_T$, $\alpha_T\eta_T \in (0,1)$, and $D_TT\geq 1$. 
Thus, $\|x^{t+1}\|\leq2D_T$ in both cases, which completes the induction.
Finally, substituting $\|x^t\|\leq 2D_T$ into
\eqref{bound:expanding-ball-velocity-direct} directly yields $\|v^t\|\leq5L_F$.
This completes the proof.
\Halmos \endproof
}

{
We are now ready to derive the convergence results of Algorithm \ref{alg:cgm-op-expanding-ball}.
}
{
\begin{theorem}[Convergece Rates]
\label{thom:monotone-vi-expanding-ball}
Consider Problem \eqref{pro-vi} under Assumption \ref{assmp:monotone-vi-expanding-ball}. Let $\{x^t\}_{t=0}^T$ be the sequence generated by Algorithm \ref{alg:cgm-op-expanding-ball}, and define $\bar{x}^{T}=T^{-1}\sum_{t=0}^{T-1}x^t$. Then, we have
\begin{align}
    \relax[g_i(x^t)]_+
    &\leq
    (1\!-\!\alpha_T\eta_T)^t[g_i(x^0)]_+ \!+\!
    \max\left\{
        \frac{1}{2}L_g,\!
        \frac{5}{4}\ell_gD_T
    \right\}\!  \sqrt{\frac{D_T}{T}} \!\sim\! \widetilde{\mathcal{O}}\left(\frac{1}{\sqrt{T}}\right), \; \forall i\in[m], 0\!\leq t\!\leq T, \label{bound:expanding-ball-nonergodic-feasibility}\\
    [g_i(\bar{x}^T)]_+
    &\leq
    \max\left\{
        \frac{1}{2}L_g,
        \frac{5}{4}\ell_gD_T
    \right\} \sqrt{\frac{D_T}{T}}
    +
    10[g_i(x^0)]_+\sqrt{\frac{D_T}{T}}\sim \widetilde{\mathcal{O}}\left(\frac{1}{\sqrt{T}}\right),  \; \forall i\in[m], \label{bound:expanding-ball-ergodic-feasibility} \\
    \mathcal{G}_{D_T}(\bar{x}^T)
    & \leq
    \frac{5L_F(D_T+D_0)^2}{\sqrt{D_TT}}
    +
    \frac{5L_F}{4} \sqrt{\frac{D_T}{T}}
    \sim \widetilde{\mathcal{O}}\left(\frac{1}{\sqrt{T}}\right), \label{bound:expanding-ball-minty-gap}
\end{align} 
where $\mathcal{G}_{D_T}(\cdot)$ is defined in \eqref{def:restricted-minty-gap}.
\end{theorem}
\proof{Proof.}
Fix any $i\in[m]$, and consider the following two cases.
\begin{itemize}
    \item[(a)] If $i\notin I_{x^t}$, then $g_i(x^t)\leq0$. The
    $L_g$-Lipschitz continuity of $g_i$, the updating rule
    $x^{t+1}=x^t+\eta_Tv^t$, and Lemma
    \ref{lemma:expanding-ball-bounds} give
    \[
        \relax [g_i(x^{t+1})]_+
        \leq
        L_g\|x^{t+1}-x^t\|\leq
        \eta_TL_g\|v^t\|\leq
        \frac{L_g}{2}  \sqrt{\frac{D_T}{T}}.
    \]

    \item[(b)] If $i\in I_{x^t}$, then $g_i(x^t)>0$. From
    \eqref{def-V-alpha-expanding-ball}, we have $\left\langle\nabla g_i(x^t),v^t\right\rangle \leq -\alpha_Tg_i(x^t)$.
    It follows from the $\ell_g$-smoothness of $g_i$ and $\|v^t\|\leq 5L_F$ given in Lemma
    \ref{lemma:expanding-ball-bounds} that
    \begin{equation}
    \label{bound:expanding-ball-feasibility-recursion}
    \begin{aligned}
        \relax [g_i(x^{t+1})]_+
        &\leq
        (1-\alpha_T\eta_T)[g_i(x^t)]_+
        +\frac{\ell_g}{2}\eta_T^2\|v^t\|^2 \leq
        (1-\alpha_T\eta_T)[g_i(x^t)]_+
        +\frac{25\ell_gL_F^2\eta_T^2}{2}\\
        &=
        (1-\alpha_T\eta_T)[g_i(x^t)]_+
        +\alpha_T\eta_T
        \frac{5\ell_gD_T}{4}    \sqrt{\frac{D_T}{T}},
    \end{aligned}
    \end{equation}
    where the equality uses
    $\alpha_T=L_F/D_T$ and
    $\eta_T=\sqrt{D_T/T}/(10L_F)$.
\end{itemize}

We now prove \eqref{bound:expanding-ball-nonergodic-feasibility} by 
induction. For $t=0$, we have
\[
\begin{aligned}
    \relax [g_i(x^0)]_+
    &\leq
    [g_i(x^0)]_+
    +
    \max\left\{
        \frac{1}{2}L_g,
        \frac{5}{4}\ell_gD_T
    \right\}  \sqrt{\frac{D_T}{T}},
\end{aligned}
\]
so the claim holds trivially. 
Suppose that
\eqref{bound:expanding-ball-nonergodic-feasibility} holds for some
$t\in\{0,\ldots,T-1\}$.  For $(t+1)$, in case (a), we have $[g_i(x^{t+1})]_+\leq \frac{1}{2}L_g \sqrt{ D_T / T}$,
and in case (b), 
\eqref{bound:expanding-ball-feasibility-recursion} and the induction
hypothesis yield
\[
\begin{aligned}
    & \relax [g_i(x^{t+1})]_+
    \leq
    (1-\alpha_T\eta_T)[g_i(x^t)]_+
    +\alpha_T\eta_T
    \frac{5\ell_gD_T}{4}
    \sqrt{\frac{D_T}{T}}\\
    \leq &
    (1-\alpha_T\eta_T)^{t+1}[g_i(x^0)]_+ +
    \max
    \left\{
        \frac{1}{2}L_g,
        \frac{5}{4}\ell_gD_T
    \right\}  \sqrt{\frac{D_T}{T}}.
\end{aligned}
\]
Therefore, the desired estimate also holds at $(t+1)$ in case~(b).
This completes the induction and proves \eqref{bound:expanding-ball-nonergodic-feasibility}.
We then establish the ergodic feasibility estimate \eqref{bound:expanding-ball-ergodic-feasibility}. Since $g_i$ is
convex, the function $x\mapsto[g_i(x)]_+$ is also convex. Therefore,
the definition of $\bar{x}^T$ and
\eqref{bound:expanding-ball-nonergodic-feasibility} imply
\[
\begin{aligned}
    \relax [g_i(\bar{x}^T)]_+
    &\leq
    \frac{1}{T}\sum_{t=0}^{T-1}[g_i(x^t)]_+ \leq
    \frac{[g_i(x^0)]_+}{T}
    \sum_{t=0}^{T-1}(1-\alpha_T\eta_T)^t +
    \max
    \left\{
        \frac{1}{2}L_g,
        \frac{5}{4}\ell_gD_T
    \right\}  \sqrt{\frac{D_T}{T}}\\
    &\leq
    \frac{[g_i(x^0)]_+}{\alpha_T\eta_TT}
    +
    \max\left\{
        \frac{1}{2}L_g,
        \frac{5}{4}\ell_gD_T
    \right\}   \sqrt{\frac{D_T}{T}}\\
    &=
    10[g_i(x^0)]_+\sqrt{\frac{D_T}{T}}
    +
    \max\left\{
        \frac{1}{2}L_g,
        \frac{5}{4}\ell_gD_T
    \right\} \sqrt{\frac{D_T}{T}},
\end{aligned}
\]
where the third inequality uses $\sum_{t=0}^{T-1}(1-\alpha_T\eta_T)^t\leq {1}/{\alpha_T\eta_T}$.
This proves \eqref{bound:expanding-ball-ergodic-feasibility}.

We next prove the restricted Minty-gap estimate \eqref{bound:expanding-ball-minty-gap}. Fix any 
$x\in\mathcal{C}$ satisfying $\|x\|\leq D_T$.  Then, from Lemma \ref{lemma:depict_veloc_2} we have $v=v^t+x-x^t \in V_{\alpha_T}(x^t)$, and hence $\langle v^t + F(x^t), x^t-x\rangle \leq 0$. Further considering the monotonicity of $F$ and $x^{t+1}=x^t+\eta_Tv^t$, it is elementary to verify that
\[
\begin{aligned}
    2\eta_T \langle F(x),x^t-x\rangle \leq 
    2\eta_T \left\langle F(x^t),x^t-x\right\rangle \leq
    \|x^t-x\|^2 -\|x^{t+1}-x\|^2 +\eta_T^2\|v^t\|^2.
\end{aligned}
\]
Summing the above inequality from $t=0$ to $T-1$, dividing both sides by $2\eta_T T$, invoking Lemma~\ref{lemma:expanding-ball-bounds}, and substituting $\eta_T = \sqrt{D_T/T}\big/(10L_F)$, we obtain
\begin{align} \label{restri-minty}
    \left\langle F(x),\bar{x}^T-x\right\rangle
    &\leq
    \frac{\|x^0-x\|^2}{2\eta_TT}
    +
    \frac{\eta_T}{2T}
    \sum_{t=0}^{T-1}\|v^t\|^2 \leq
    \frac{5L_F\|x-x^0\|^2}{\sqrt{D_TT}}
    +
    \frac{5L_F}{4}    \sqrt{\frac{D_T}{T}}.
\end{align}
The desired result \eqref{bound:expanding-ball-minty-gap} follows immediately from the facts that \eqref{restri-minty} holds for any $x\in\mathcal{C}$ satisfying $\|x\|\leq D_T$, that $\|x-x^0\|\leq D_T+D_0$ whenever $x\in\mathcal{C}$ and $\|x\|\leq D_T$, and the definition of the restricted Minty gap in \eqref{def:restricted-minty-gap}.
\Halmos \endproof
}

{
Algorithm \ref{alg:cgm-op-expanding-ball} adopts a predetermined yet arbitrary iteration budget $T$.
Given $x^0$ and $D_0$, the parameters $(D_T,\alpha_T,\eta_T)$, the finite‑length sequence $\{x^t\}_{t=1}^T$, and hence $\bar{x}^{T}=T^{-1}\sum_{t=0}^{T-1}x^t$ are all determined by $T$.
Next, we consider varying $T$ and examine the properties of the infinite sequence $\{\bar{x}^{T}: T = 1, 2, \ldots\}$.
Note that when considering varying $T$, each $x^t$ from  $\{x^t\}_{t=1}^T$ is dependent on $T$ and should more appropriately be denoted by $x_T^t$. However, since only $\{\bar{x}^{T}: T = 1, 2, \ldots\}$, and not $\{x^t\}_{t=1}^T$ for any fixed $T$, will be involved in the subsequent analysis, we do not indicate the dependence of $x^t$ on $T$ in our discussions.
}

{ 
\begin{theorem}[Properties of $\bar{x}^{T}$ with Large $T$]
Consider Problem \eqref{pro-vi} under Assumption~\ref{assmp:monotone-vi-expanding-ball}.
Given $x^0$ and $D_0$, for each prescribed horizon $T\geq 1$, Algorithm~\ref{alg:cgm-op-expanding-ball} yields a $\bar{x}^T$, thus defining a sequence $\{\bar{x}^{T}: T = 1, 2, \ldots\}$.
Then, for every $\varepsilon>0$ and every fixed $x\in\mathcal{C}$, there exists an integer $T_{\varepsilon,x}$ such that for all $T\geq T_{\varepsilon,x}$, it holds that
\begin{equation} \label{thom:assert1}
\left\langle F(x),\bar{x}^T-x\right\rangle\leq\varepsilon, \text{~~and~~}
[g_i(\bar{x}^T)]_+\leq\varepsilon, \quad \forall i\in[m].    
\end{equation}
If, in addition, the solution set $\mathcal{X}^*$ of \eqref{pro-vi} is nonempty and bounded, then  $\{\bar{x}^T: T=1,2,\ldots\}$ is bounded, and moreover, there holds $\lim_{T\rightarrow\infty}\operatorname{dist}(\bar{x}^T,\mathcal{X}^*) = 0$.
\end{theorem}
}

{
\proof{Proof.}
Fix $\varepsilon>0$ and $x\in\mathcal{C}$ arbitrarily. Since $D_T=D_0+\log(T+1)$, there exists an integer $T_x$ such that $\|x\|\leq D_T$ for all $T\geq T_x$. Thus, \eqref{restri-minty} holds for this $x$ whenever $T\geq T_x$. Note that the right-hand sides of both  \eqref{restri-minty} and \eqref{bound:expanding-ball-ergodic-feasibility} are of order $\mathcal{O}\big(\log^{3/2}(T+1)/\sqrt{T}\big)$. Hence, there exists an integer $T_{\varepsilon,x}\geq T_x$ such that the right-hand sides of both \eqref{restri-minty} and \eqref{bound:expanding-ball-ergodic-feasibility} are no greater than $\varepsilon$ for all $T\geq T_{\varepsilon,x}$. This proves \eqref{thom:assert1}.

Assume that $\mathcal{X}^*$ is nonempty and bounded. Then, $\sup_{z\in\mathcal{X}^*}\|z-x^*\|$ is finite.
We next show that $\{\bar{x}^T: T=1,2,\ldots\}$ is bounded.
Fix any $x^*\in\mathcal{X}^*$ and let $R>\sup_{z\in\mathcal{X}^*}\|z-x^*\|$.
Suppose, on the contrary, that $\{\bar{x}^T: T=1,2,\ldots\}$ is unbounded; then there exists a subsequence of positive integers $\{T_j: j=1,2,\ldots\}$ such that $\|\bar{x}^{T_j}-x^*\| > R$ for all $j\geq 1$ and $\|\bar{x}^{T_j}-x^*\|\to\infty$ as $j\to \infty$. 
Define $\theta_j:=R/\|\bar{x}^{T_j}-x^*\| \in (0,1)$ and $y^j:=(1-\theta_j)x^*+\theta_j\bar{x}^{T_j}$ for $j\geq 1$.
Then,  $\theta_j\to0$,  $\|y^j-x^*\| =\theta_j\|\bar{x}^{T_j}-x^*\| \equiv R$ for all $j\geq 1$, and  $\{y^j: j=1,2,\ldots\} \subseteq \{z\in\mathbb{R}^n:\|z-x^*\|=R\}$.
By taking a subsequence if necessary, we can assume without loss of generality that $\lim_{j\to\infty} y^j = y$ for some $y\in\mathbb{R}^n$ satisfying $\|y-x^*\| = R$.
For any $i\in[m]$, the convexity of $g_i$ and $x^*\in\mathcal{C}$ ensure that
\[
g_i(y) = \lim_{j\rightarrow\infty} g_i(y^j) \leq  \lim_{j\rightarrow\infty} [(1-\theta_j)g_i(x^*) + \theta_j g_i(\bar{x}^{T_j})] 
\leq  \lim_{j\rightarrow\infty}   \theta_j[g_i(\bar{x}^{T_j})]_+ = 0,
\] 
where the last equality follows from $\theta_j\to 0$ and \eqref{bound:expanding-ball-ergodic-feasibility}. 
Hence,  $y\in\mathcal{C}$. 
Recall that $F$ is monotone. Then,  $x\in\mathcal{C}$ and $x^*\in\mathcal{X}^*$ imply
$\langle F(x),x^*-x\rangle\leq \langle F(x^*),x^*-x\rangle \leq 0$.  
Furthermore, from $0<\theta_j \rightarrow 0$ and $y^j=(1-\theta_j)x^*+\theta_j\bar{x}^{T_j}\rightarrow y$, we have
\[\langle F(x),y-x\rangle  =  \lim_{j\rightarrow\infty} \langle F(x), (1-\theta_j)(x^*-x)+\theta_j(\bar{x}^{T_j} -x)\rangle 
\leq\lim_{j\to\infty}\theta_j
\langle F(x),\bar{x}^{T_j}-x\rangle \leq 0,\]
where the second ``$\leq$" follows from the first part in \eqref{thom:assert1}.
Since $F$ is continuous, \eqref{pro-vi-minty} and \eqref{pro-vi} are equivalent. Thus, we conclude that $y\in\mathcal{X}^*$, contradicting $\|y-x^*\|=R>\sup_{z\in\mathcal{X}^*}\|z-x^*\|$.
Therefore, the sequence $\{\bar{x}^T: T=1,2,\ldots\}$ must be bounded.

Finally, we show that $\lim_{T\rightarrow\infty}\operatorname{dist}(\bar{x}^T,\mathcal{X}^*) = 0$, again under the assumption that $\mathcal{X}^*$ is nonempty and bounded. In this case, the sequence $\{\bar{x}^{T}: T=1,2,\ldots\}$ is bounded. 
Let $\{\bar{x}^{T_j}: j=1,2,\ldots\}$ be any subsequence converging to $\widehat{x}$, and assume that $\{T_j: j=1,2,\ldots\}$ is strictly increasing.
Then, for any $\delta>0$ and $x\in\mathcal{C}$,   $T_j\geq T_{\delta,x}$ for all $j$ sufficiently large, where $T_{\delta,x}$ is determined in the first assertion of this theorem. Hence,  it is easy to derive from \eqref{thom:assert1} that 
$[g_i(\widehat{x})]_+\leq\delta$ for all $i\in[m]$ and $\langle F(x),\widehat{x}-x\rangle\leq\delta$.
Since $\delta>0$ and $x\in\mathcal{C}$ are arbitrary, we conclude that $\widehat{x}\in\mathcal{C}$ and $\langle F(x),\widehat{x}-x\rangle\leq0$ for any $x\in\mathcal{C}$. As a result, from  the equivalence between \eqref{pro-vi-minty} and \eqref{pro-vi}, we confirm that $\widehat{x}\in\mathcal{X}^*$. 
Suppose, on the contrary, that $\operatorname{dist}(\bar{x}^T,\mathcal{X}^*)\not\to0$. Then, there exist $\varepsilon_0>0$ and, based on the preceding argument, a subsequence $\{\bar{x}^{T_j}: j=1,2,\ldots\}$ converging to some $\widehat{x} \in \mathcal{X}^*$ such that
$0<\varepsilon_0
\leq
\operatorname{dist}(\bar{x}^{T_j},\mathcal{X}^*)
\leq
\|\bar{x}^{T_j}-\widehat{x}\|
\to0$, an apparent contradiction.
Therefore, we must have $\lim_{T\rightarrow\infty}\operatorname{dist}(\bar{x}^T,\mathcal{X}^*) = 0$, completing the proof.
\Halmos \endproof
}

{
\begin{remark}[Necessity of our expanding auxiliary ball]
When the feasible set $\mathcal{C}$ is bounded, one can choose a fixed radius $\widehat D<\infty$ such that $\mathcal{C}\subseteq\{x\in\mathbb{R}^n:\, \|x\|\leq\widehat D\}$ and introduce an auxiliary ball constraint $g_{m+1}(x) = \|x\|^2 - \widehat D^2 \leq 0$ to establish boundedness of the iterates $\{x^t\}$ and velocities $\{v^t\}$, as is done in \cite{ZHM25}.
This construction cannot be directly extended to an unbounded feasible set, since no auxiliary ball constraint with finite radius can be guaranteed a priori to contain a solution to the VI problem. 
To handle an unbounded feasible set, Algorithm~\ref{alg:cgm-op-expanding-ball} therefore employs a horizon-dependent auxiliary ball with radius $D_T=D_0+\log(T+1)$, which contains every fixed point in $\mathcal{C}$ for sufficiently large $T$ while growing slowly enough to preserve the convergence bounds.
The following simple example shows that such an expanding mechanism is necessary: for any prescribed finite radius $\widehat D$, one can construct a problem, depending on $\widehat D$, for which the CGM in \cite{ZHM25}
converges to an incorrect limit point.

\emph{Example:} 
Given any $\widehat{D} > 0$, we consider the VI problem with $F(x) = \tanh(x-2\widehat{D})$ and $\mathcal{C}=\mathbb{R}$ represented by the void constraint $g_1(x)\equiv-1\leq0$. 
Note that $F$ is continuous, monotone, globally $1$-Lipschitz, and globally bounded, while the VI admits a unique solution $x^*=2\widehat{D}$.
Suppose we add the auxiliary ball constraint $g_2(x)=x^2-\widehat D^2\leq0$, 
set $V_{\alpha}(x) = \{v:\, \alpha g_i(x) + \nabla g_i(x)^\top v \leq 0, \forall i\in 
I_{x}^{\textbf{act}}\}$ with $I_{x}^{\textbf{act}}=\{i\in\{1,2\}: \, g_i(x)\geq0\}$, 
and initialize the CGM in \cite{ZHM25}  at $x^0=\widehat{D}$.
Then, it is easy to verify that $I_{x^0}^{\textbf{act}} = \{2\}$ and $V_{\alpha}(x^0)= (-\infty, 0]$ for any $\alpha \geq 0$. 
Since $F(\widehat D)<0$, the velocity update $v^0=\arg\min_{v\leq0}|v+F(\widehat D)|^2$ yields $v^0=0$. 
Hence, we  have $x^1 = x^0 + \eta_0 v^0 = x^0$ for any $\eta_0\geq 0$.
Then, by induction, for every $\alpha\geq0$ and every nonnegative step size sequence $\{\eta_t\}$, 
we have $v^t\equiv 0$ and $x^t \equiv x^0 = \widehat D$ for all $t\geq0$. 
Consequently, $\bar{x}^T\equiv \widehat D$ and $|\bar{x}^T-x^*|\equiv \widehat D$ for all $T\geq1$.  Moreover, there holds 
$\mathcal{G}(\bar{x}^T)= \sup_{x\in\mathbb{R}} F(x)\big(\bar{x}^T-x\big)\geq F\big(3\widehat{D}/2\big) \big(\widehat{D}- 3\widehat{D}/2 \big)= \frac{1}{2}\widehat{D}\tanh\big( \widehat{D}/2 \big)>0$.
Hence, the Minty gap $\mathcal{G}(\bar{x}^T)$ does not converge to zero either.
\end{remark}
}

\section{Numerical Experiments} \label{sec:ne}
In this section, we present numerical experiments to validate the convergence results of our CGMs with general functional constraints, {applied to optimization problems and VI problems under both strongly monotone and general monotone settings.} 
 All implementations were conducted in Python 3.11.0, with quadratic programming (QP) subproblems \eqref{update-v-f} and \eqref{update-v} solved using \verb|cvxpy|. A fixed random seed (\verb|np.random.seed(42)|) was used across all experiments to ensure reproducibility.

\subsection{Resource Allocation Problem}
In this section, we consider the following resource allocation problem (RAP): 
\begin{equation} \label{pro:rap}
\begin{aligned}
\min_{x\in \mathbb{R}^d} & \quad f(x):=\frac{1}{2} x^{\top} \Sigma x + a^{\top} x\\
\text{s.t.} & \quad x\in \mathcal{C}:=\{ x \mid x\geq 0, \; \mathbf{1}^{\top} x = 1,\; r^{\top} x \leq R_{\max},\; x^{\top} E x \leq E_{\max}\},
\end{aligned}
\end{equation}
where $a, r\in \mathbb{R}_{+}^{d}$, $\Sigma\succ 0$, $E\succeq 0$, and $E_{\max}, R_{\max}>0$ are positive constants.
Here, the linear constraint $r^{\top} x \leq R_{\max}$ defines the resource budget limit, the quadratic constraint $x^{\top} E x \leq E_{\max}$ regulates the allowable risk threshold, and the objective function quantifies the allocation cost. Notably, the objective function $f$ is $\lambda_{\max}(\Sigma)$-smooth and $\lambda_{\min}(\Sigma)$-strongly convex. 
To cast the set $\mathcal{C}$ in \eqref{pro:rap} into the form \eqref{pro-con}, we introduce constraint functions $g_i$ for $i=1,2,\ldots, d+4$ as follows: 
$g_i(x)=-\mathbf{e}_{i}^{\top}x$ for $i\in [d]$, $g_{d+1}(x)=\mathbf{1}^{\top} x - 1$, $g_{d+2}(x)= 1-\mathbf{1}^{\top} x$, $g_{d+3}(x)=r^{\top} x -R_{\max} $, and $g_{d+4}(x)=x^{\top} E x - E_{\max}$. 
Let $\mathcal N(0,1)$ denote the normal distribution with mean $0$ and variance $1$, and $\mathcal U(0,1)$ denote the uniform distribution over $[0,1]$.
In our numerical setup, we generate two independent $d\times 10$ Gaussian matrices $G_1$ and $G_2$ with entries sampled from $\mathcal N(0,1)$, and set $\Sigma =G_1 G_{1}^{\top} + 5I_d$ and $E = G_{2} G_{2}^{\!\top} + 10I_d$. 
We define $\bar\sigma = \frac{1}{d}\sum_{i=1}^{d}\sqrt{\Sigma_{ii}}$, sample $u\sim\mathcal U(0,1)^d$, and set $a=\bar{\sigma}u$. 
For generating $r$, we sample each component $r_i\sim|\mathcal N(0,1)|+0.1$ for $i=1,\ldots,d$. 
Additionally, $R_{\max}$ and $E_{\max}$ are computed as $R_{\max}= \frac{1}{d}\sum_{i=1}^{d} r_i$ and $E_{\max}= \tfrac{1}{d^2}{\mathbf 1}^{\top}\!E{\mathbf 1}$,  respectively. The initial point is set to $x^0 = \frac{1}{d}\mathbf 1$ with $d=50$, while the optimal solution $x^*$ and optimal value $f(x^*)$ are obtained using \verb|cvxpy|.

We first evaluate the performance of Algorithm \ref{alg:cgm-f} for solving \eqref{pro:rap} with a constant step size $\eta_t=\eta=\log T/(\mu T)$ for all $t=0,1,\ldots, T-1$,
considering both small iteration counts $T=\{100,150,200,250\}$ and large iteration counts $T=\{1500,2000,2500,3000\}$. 
The results are reported in Figure \ref{f-T}. 

\begin{figure}[!hbtp]
\centering
\subfloat
{\includegraphics[scale=0.35]{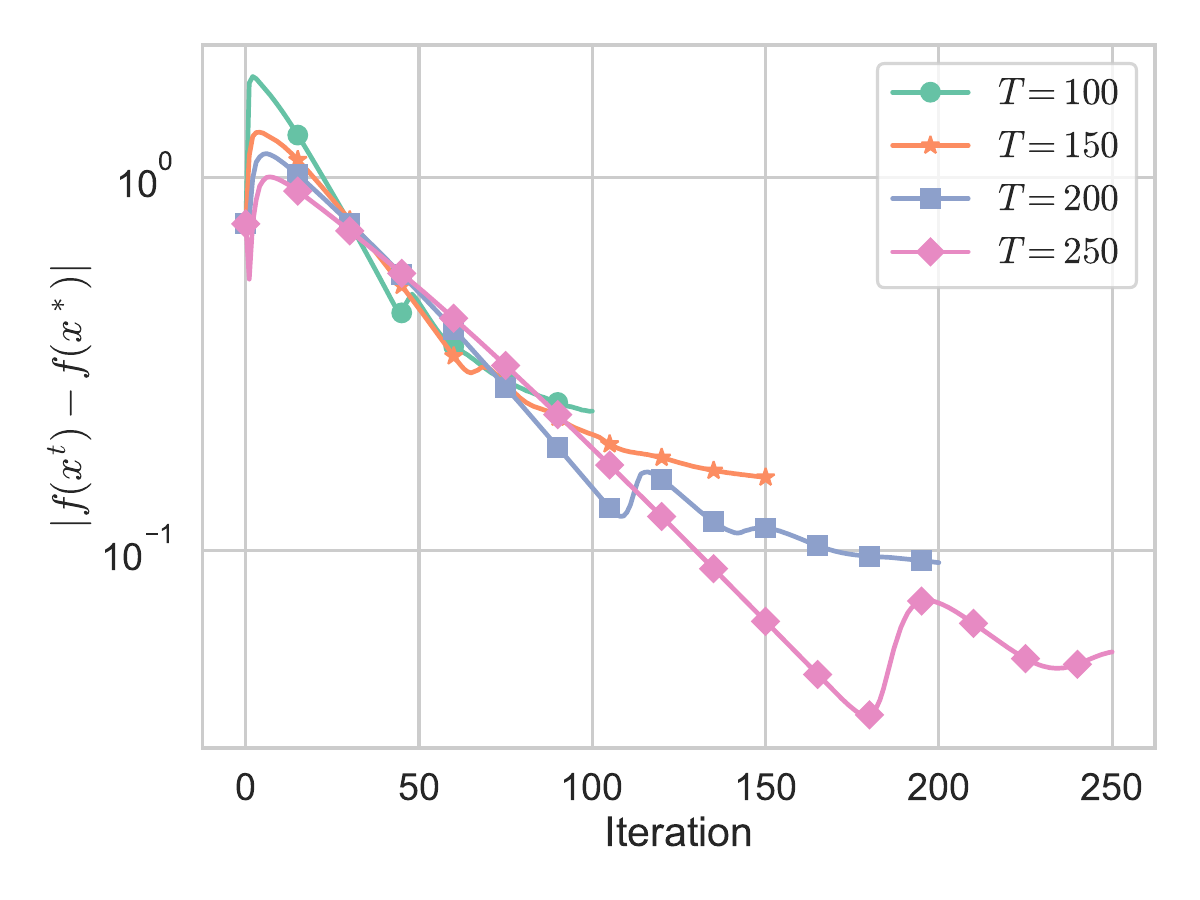}}
\subfloat
{\includegraphics[scale = 0.35]{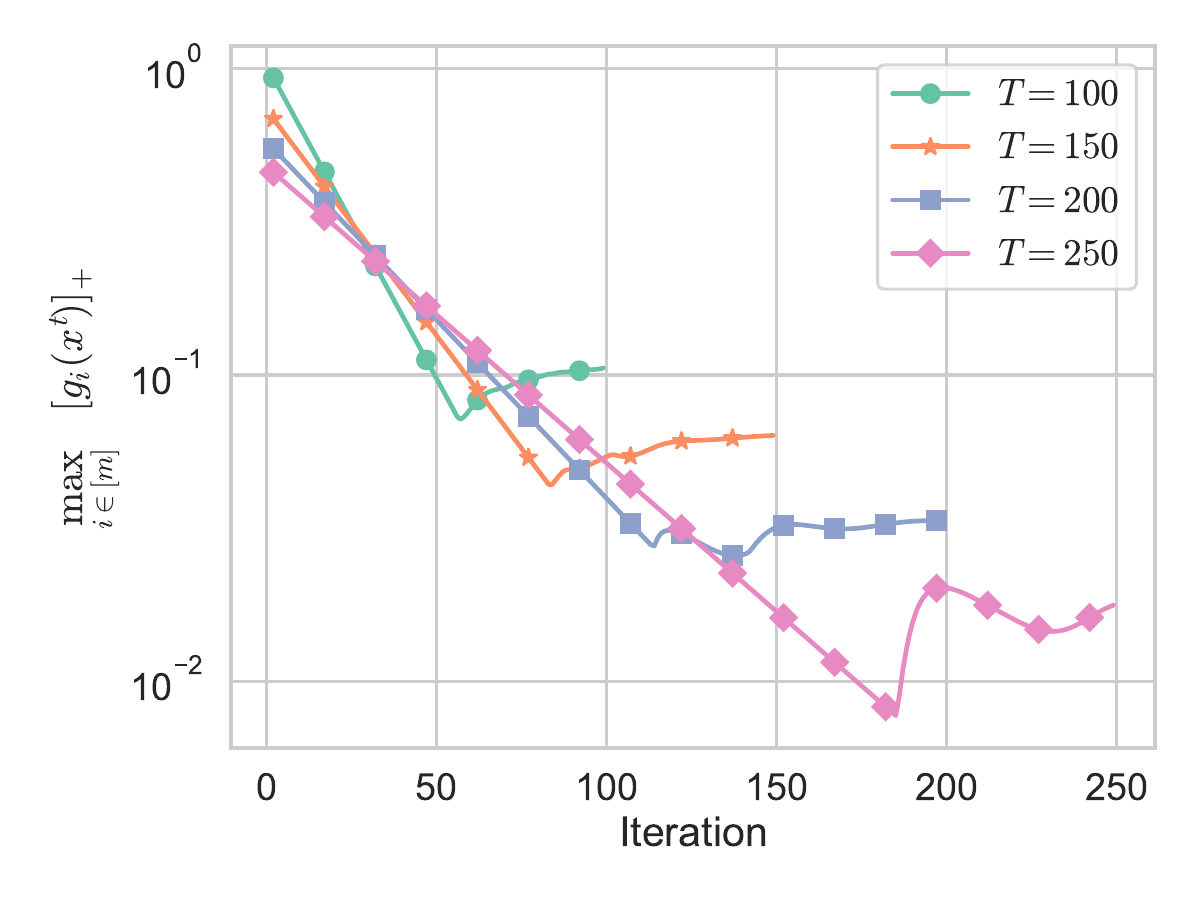}}\\
\subfloat
{\includegraphics[scale=0.35]{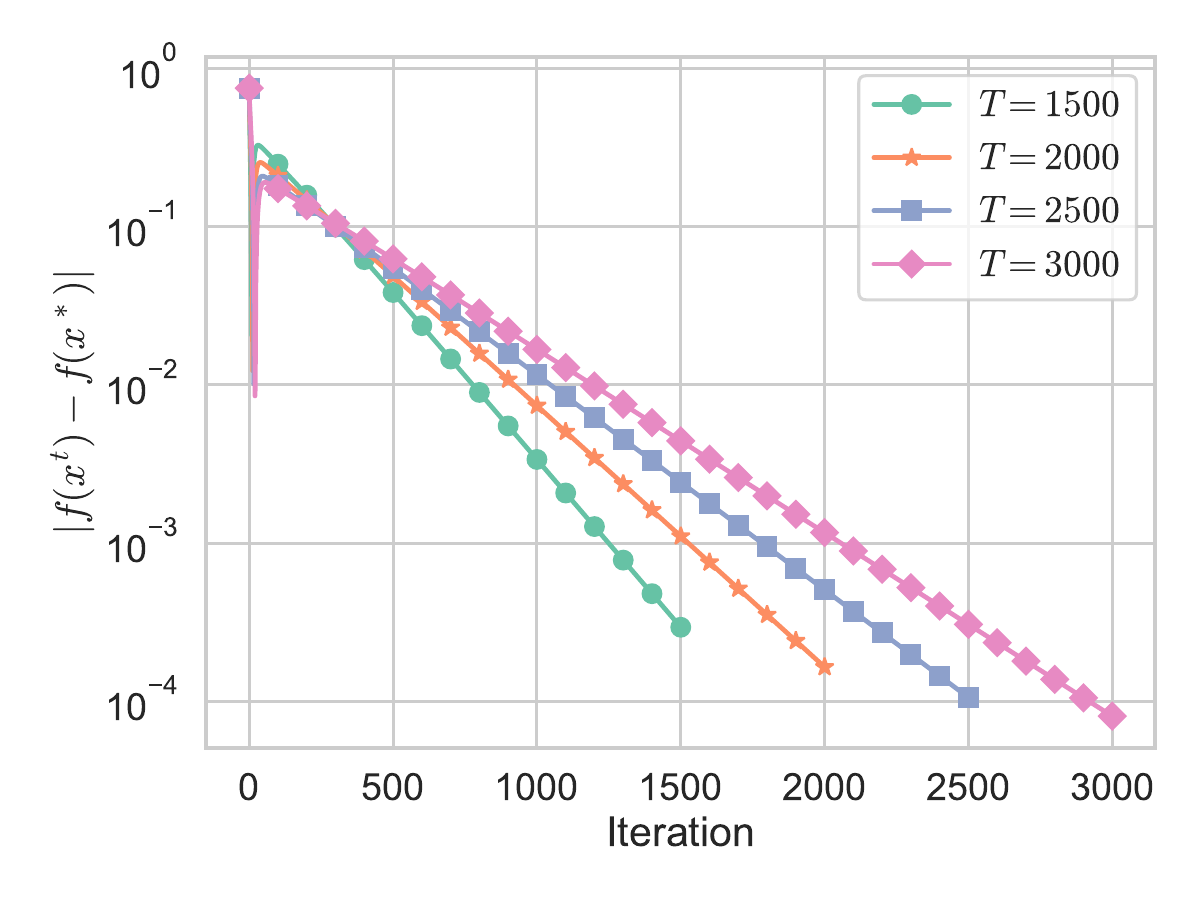}}
\subfloat
{\includegraphics[scale = 0.35]{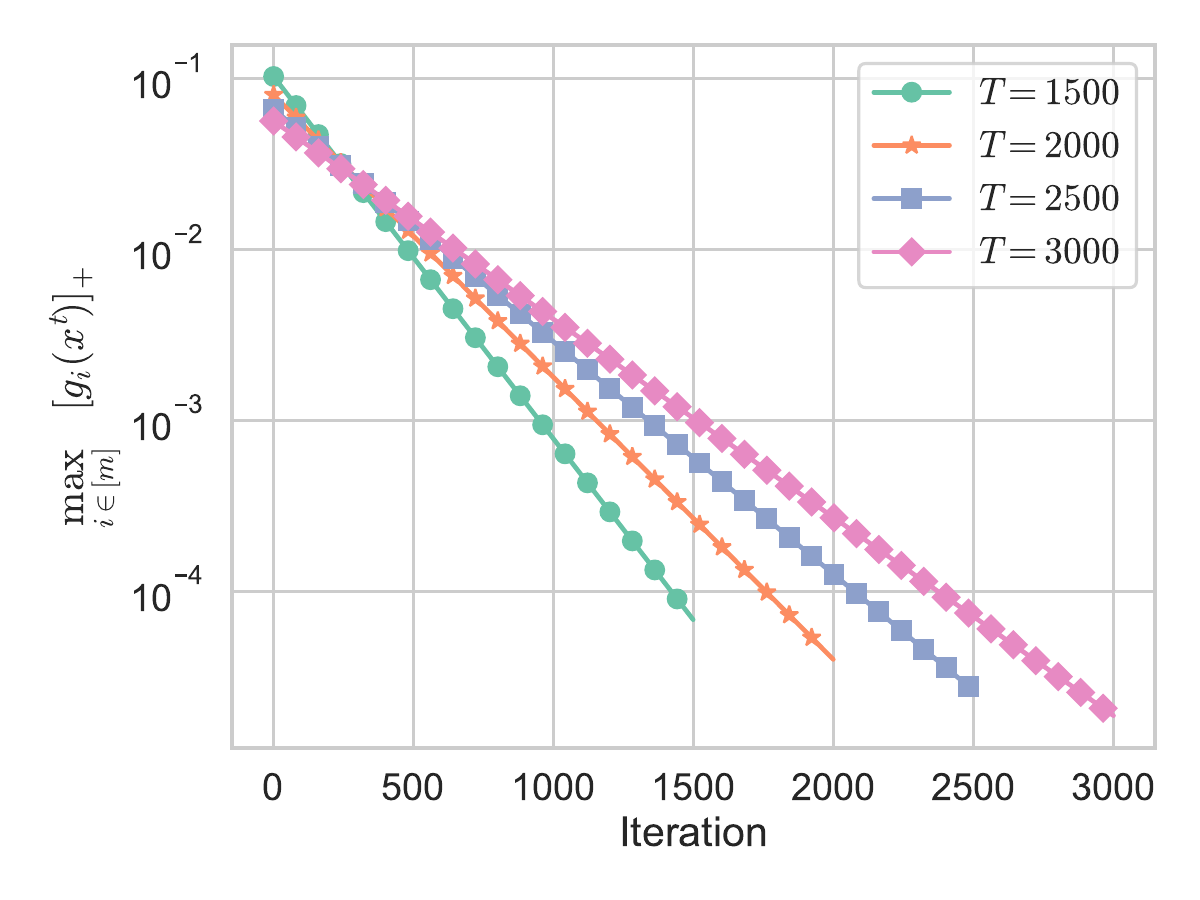}}
\caption{Convergence results of CGM-Min for the RAP with varying iteration counts: small $T=\{100,150,200,250\}$ and large $T=\{1500,2000,2500,3000\}$.
Left: absolute function value residual; Right: constraint violation with $m=d+4$. }
\label{f-T}
\end{figure}

From Figure \ref{f-T}, we observe that as $T$ increases, the step size decreases and convergence slows (this trend is more pronounced in the second row), while the resulting solution becomes more accurate. Furthermore, by comparing the first row (corresponding to smaller $T$) and the second row (corresponding to larger $T$) of Figure \ref{f-T}, we find that setting $T$ too small may lead to unsatisfactory solution accuracy—whether measured by function value residual or feasibility. 

Secondly, we compare the efficacy of Algorithm \ref{alg:cgm-f} with two step size strategies: a constant step size $\eta_t=\eta=\log T/(\mu T)$ with $T=2000$ and a diminishing step size $\eta_t = 1/\mu(t+\kappa)$. The comparison results are presented in Figure \ref{f-t-T}. From the second subfigure in Figure \ref{f-t-T}, the initial step sizes of the varying strategy may be too large, causing the iterates to exit the constraint set $\mathcal{C}$ and resulting in objective values significantly smaller than $f(x^*)$.  Both approaches for setting the step sizes eventually converge to the optimal value. Additionally, while the constant step size yields better performance than the varying step size in Figure \ref{f-t-T}, it requires predefining an appropriate $T$. As illustrated in Figure \ref{f-T} (particularly in the first row), an improperly chosen $T$ will limit the precision of both the function value residual and constraint violation.

\begin{figure}[!htbp]
\centering
\subfloat
{\includegraphics[scale=0.28]{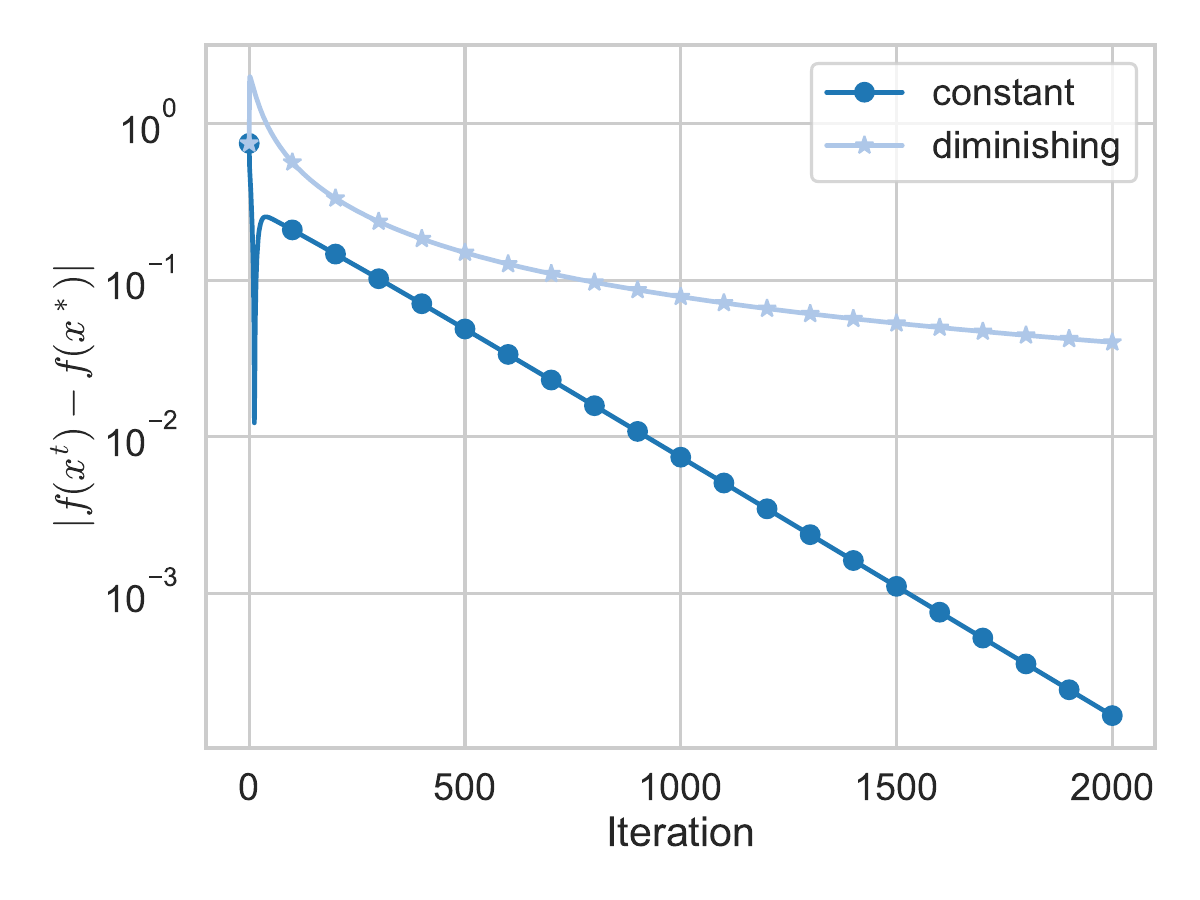}}
\subfloat
{\includegraphics[scale = 0.28]{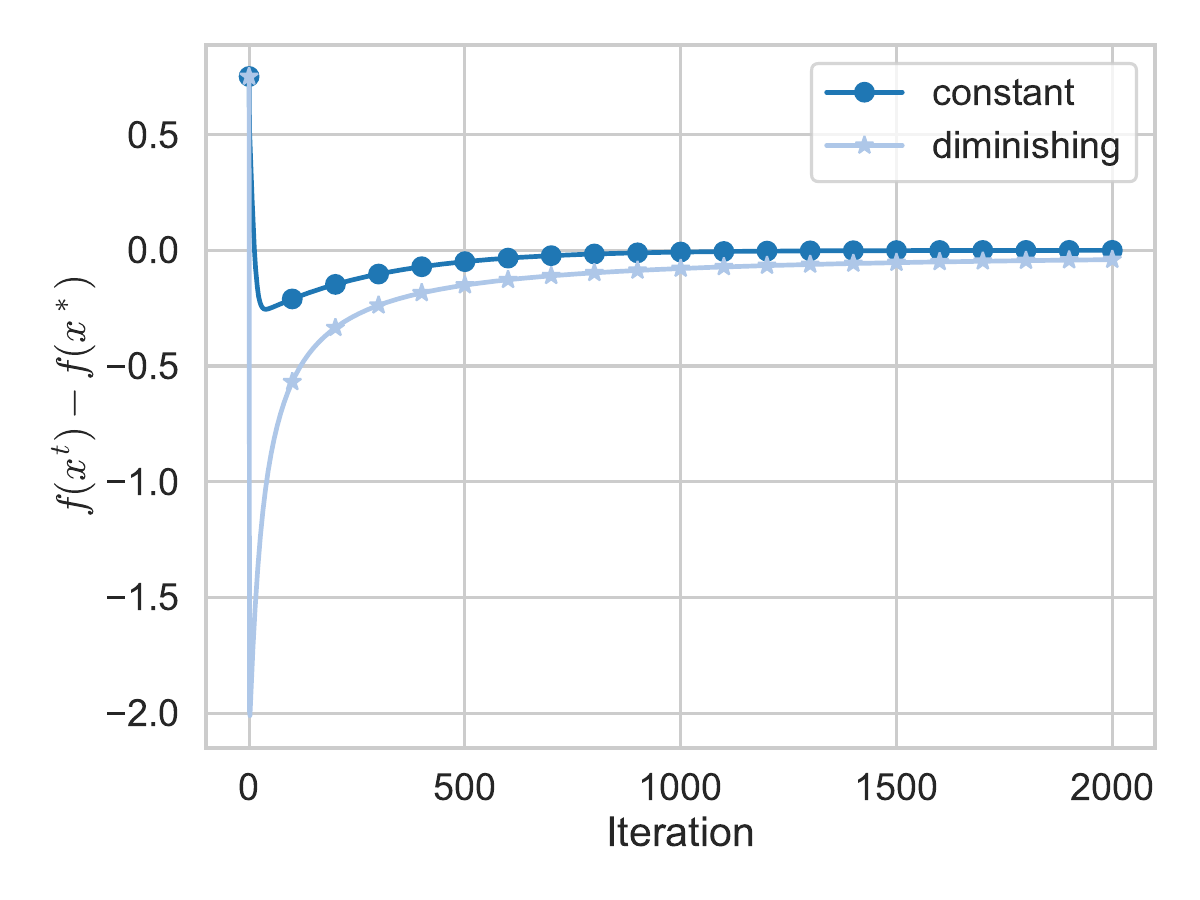}}
\subfloat
{\includegraphics[scale = 0.28]{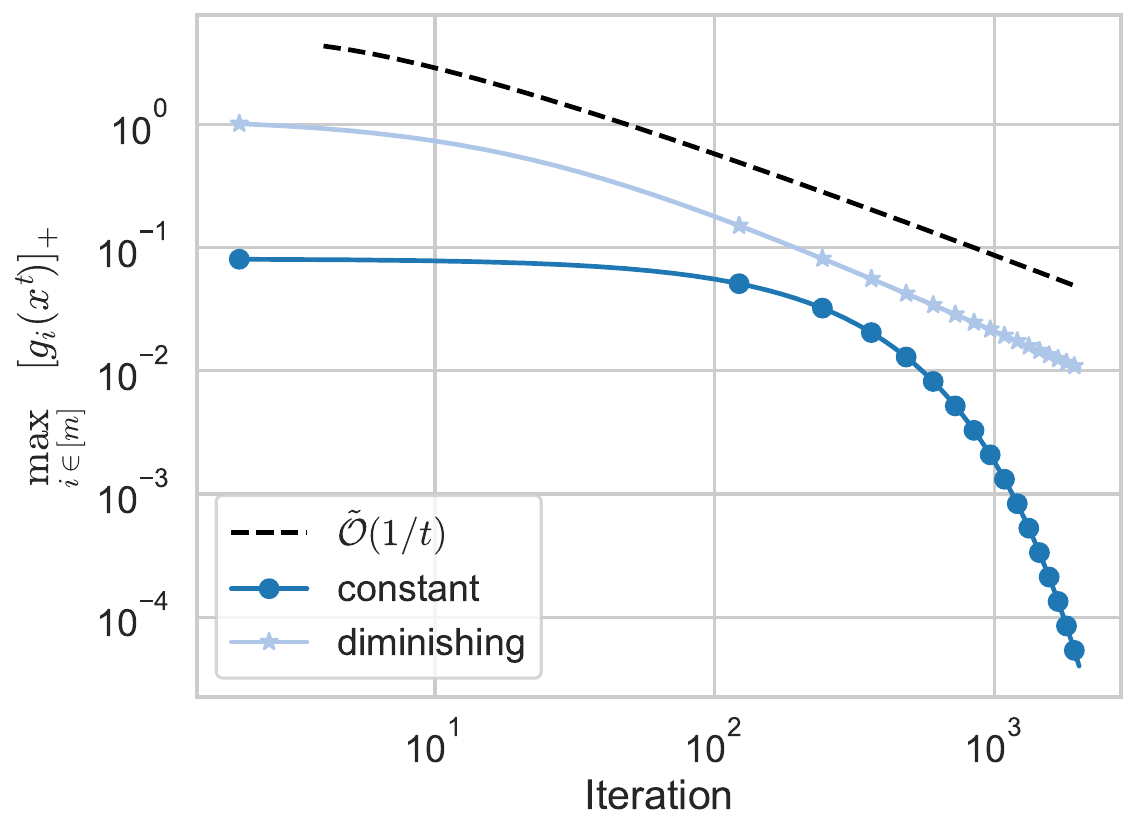}}
\caption{Convergence results of CGM-Min for the RAP: constant step size $\eta_t=\eta=\log T/(\mu T)$ (with $T=2000$) versus 
diminishing step size $\eta_t = 1/\mu(t+\kappa)$. 
Left: absolute function value residual; 
Middle: function value residual; 
Right: constraint violation with $m=d+4$. }
\label{f-t-T}
\end{figure}

{We also investigate whether CGM-Min remains effective when initialized outside the feasible region by applying CGM-Min with the following four infeasible initial points:
\[
x_1^0=0,\qquad
x_2^0=\frac{2}{d}\mathbf{1},\qquad
x_3^0=\mathbf{e}_{j^*},\qquad
x_4^0=\frac{1}{d}\mathbf{1}+0.08(\mathbf{e}_1-\mathbf{e}_2),
\]
where $j^*\in\arg\max_{j\in[d]}\max_{i\in[m]}[g_i(\mathbf{e}_j)]_+$.}
{We test both the constant step size with $T=2000$ and the diminishing step size. }
{The numerical results are given in Figure \ref{f-infeasible}.
It can be seen that for all four infeasible initial points and under both step size choices,
the constraint violation improves steadily (Figure \ref{f-infeasible}, first column) and the objective residual improves steadily as well (Figure \ref{f-infeasible}, second column)  after an oscillation in the first few initial iterations.
This observation is consistent with our convergence analysis, which permits an infeasible starting point $x^0$.}
\begin{figure}[!hbtp]
\centering
\subfloat
{\includegraphics[scale=0.3]{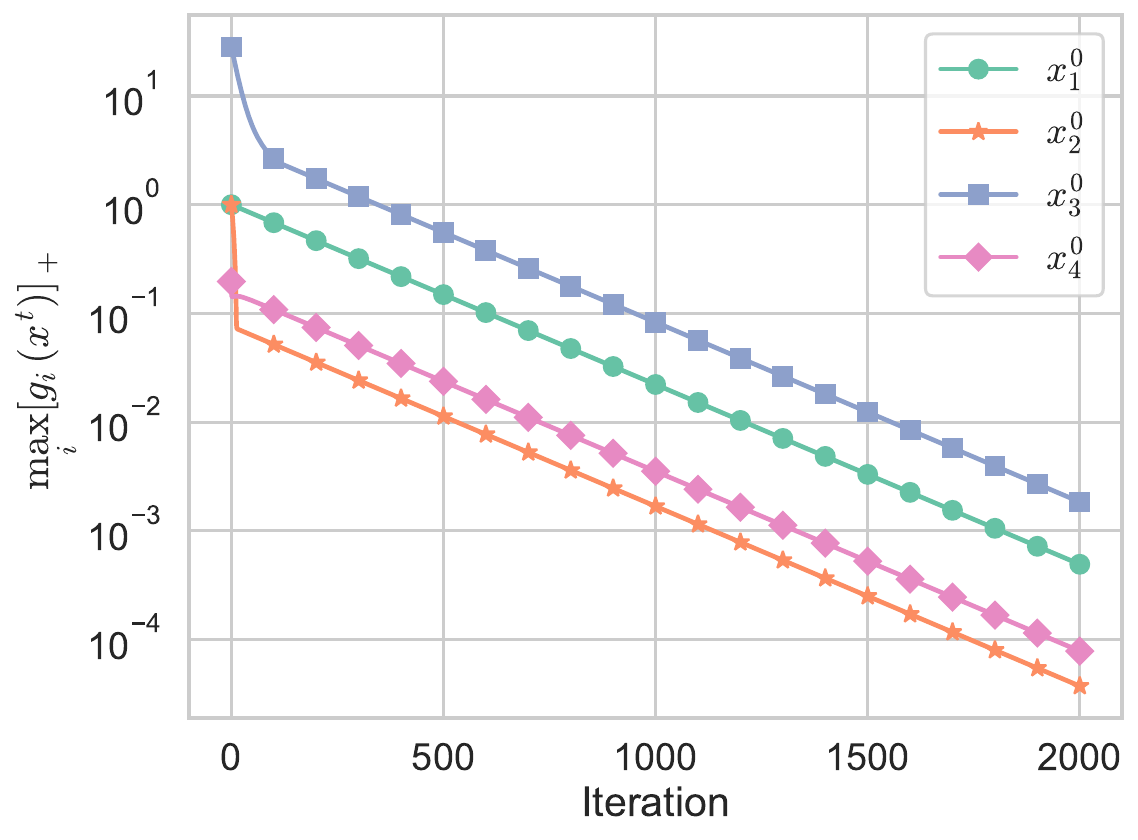}}
\subfloat
{\includegraphics[scale = 0.3]{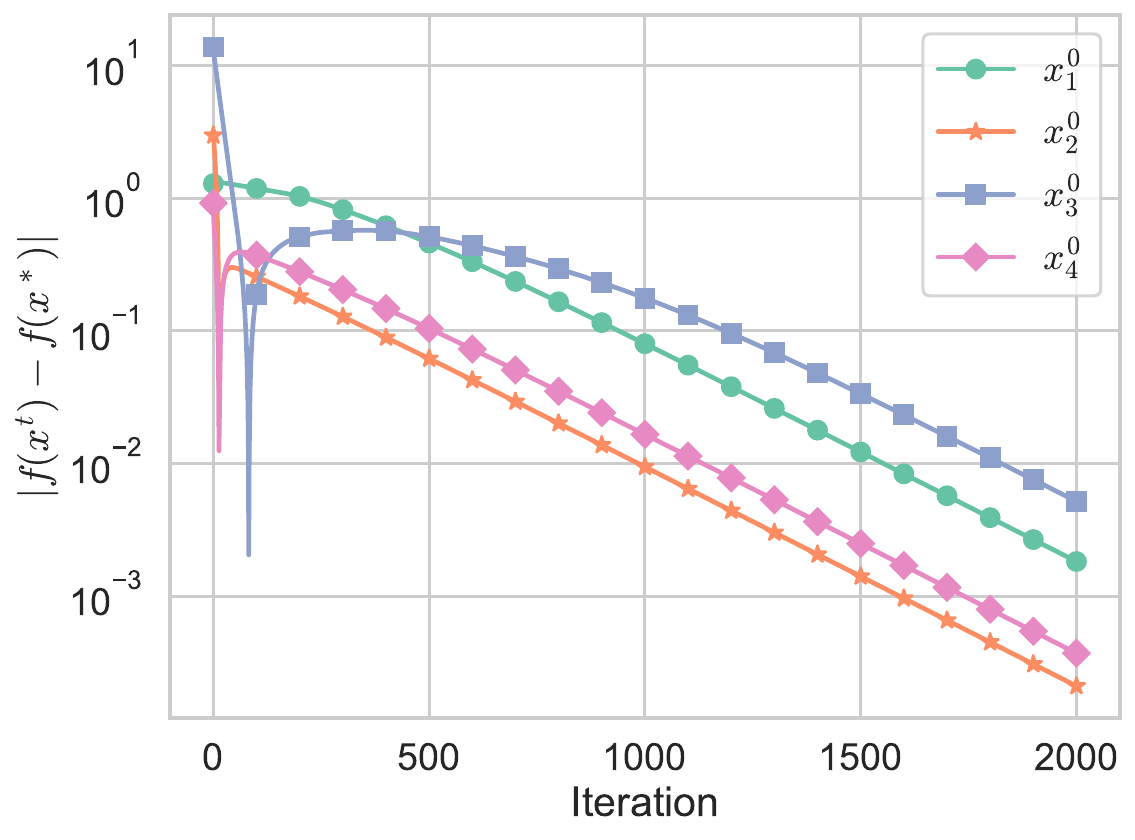}}\\
\subfloat
{\includegraphics[scale=0.3]{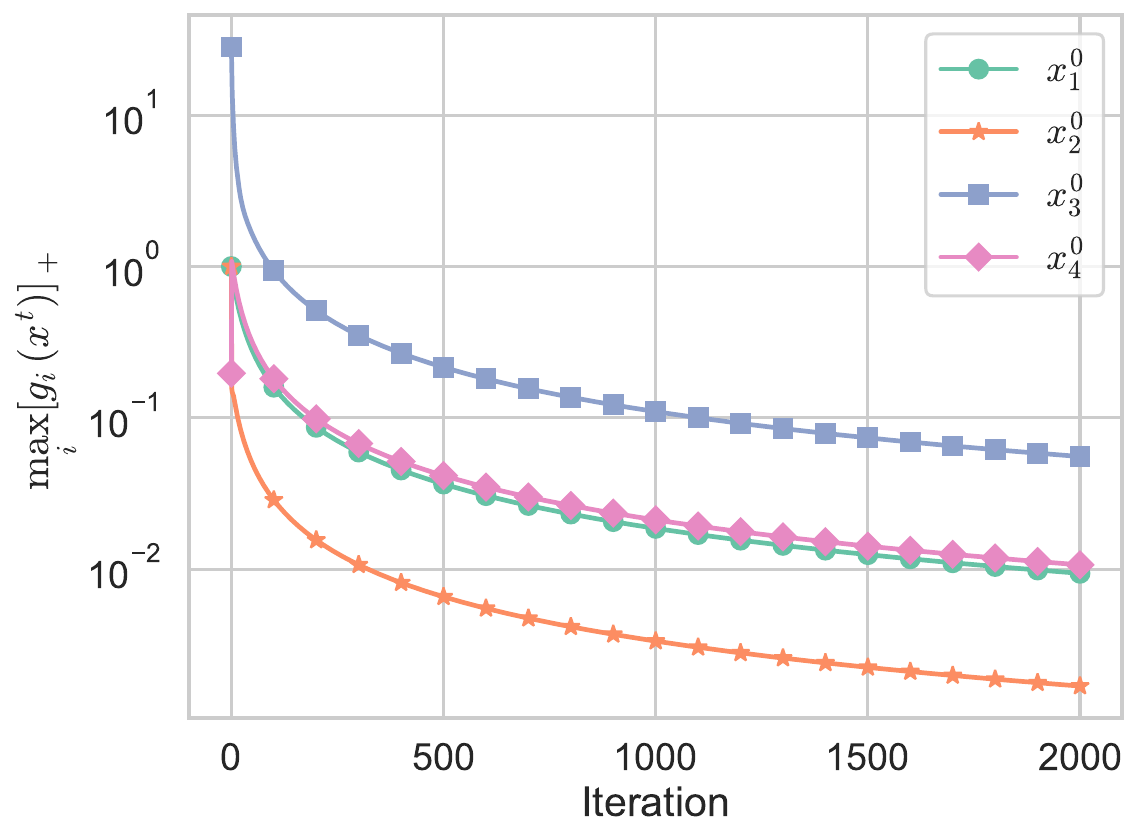}}
\subfloat
{\includegraphics[scale = 0.3]{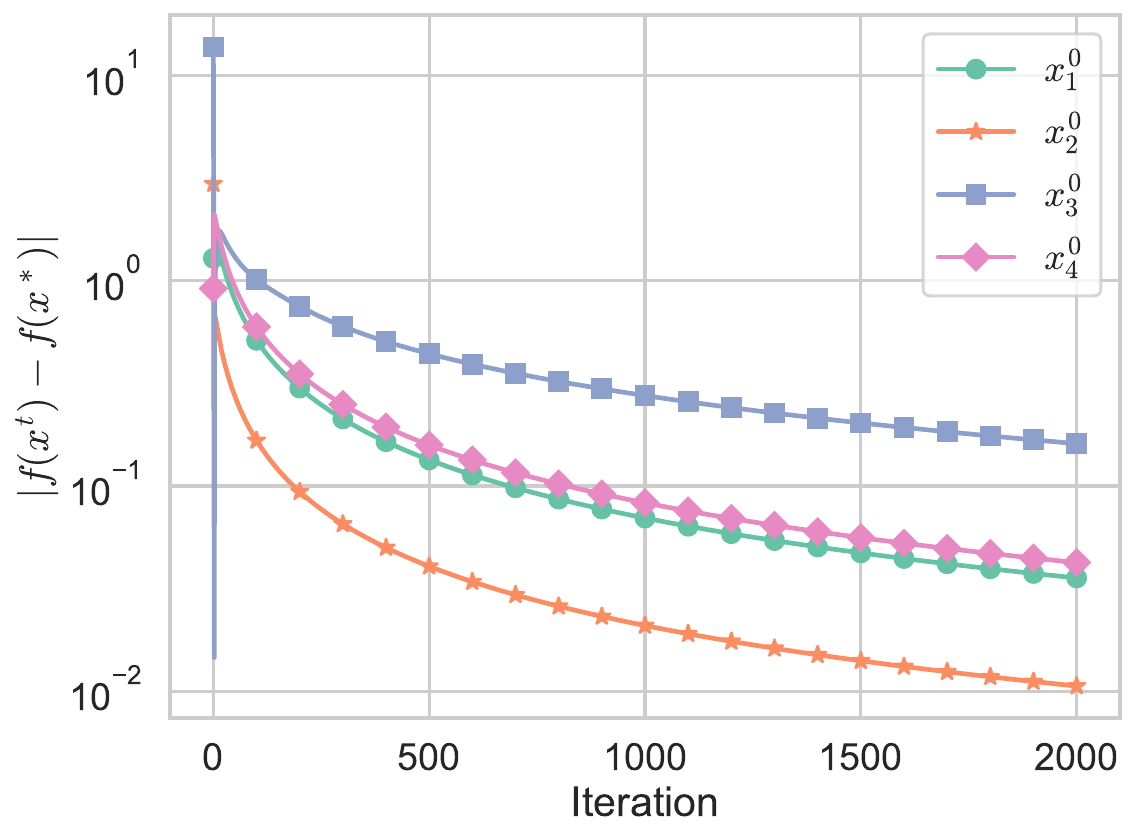}}
\caption{{Convergence results of CGM-Min for the RAP problem with four infeasible initial points.
Left: constraint violation with $m=d+4$;
Right: absolute function value residual. 
First row: constant step size $\eta_t=\eta=\log T/(\mu T)$ with $T=2000$; 
Second row: diminishing step size $\eta_t = 1/\mu(t+\kappa)$.}}
\label{f-infeasible}
\end{figure}

\subsection{High-Dimensional Bilinear Game}
In this section, we consider a high-dimensional bilinear game (HBG) problem formulated as the following min-max problem \cite{CYPJ24}:
\begin{equation} \label{saddle-hbg}
\begin{aligned}
\min_{x_1\in \mathbb{R}^d} \max_{x_2 \in \mathbb{R}^d} & \quad \beta x_1^\top x_1  + (1-\beta)  x_1^\top x_2 - \beta  x_2^\top x_2, \\
\text{s.t.} & \quad  x\in \mathcal{C}:=\{x = [x_1^{\top}, x_2^{\top}]^{\top} \mid x \geq \mathbf{0}_{2d},\; \mathbf{1}_{d}^\top x_1 = 1,\; \mathbf{1}_{d}^\top x_2 = 1\},
\end{aligned}
\end{equation}
where $x=[x_1^{\top}, x_2^{\top}]^{\top}$ encompasses the strategies of two players (i.e., $x_1$ and $x_2$), 
and the parameter $\beta \in (0,1)$ modulates the rotational influence within the game.
Problem \eqref{saddle-hbg} can be reformulated as the VI problem \eqref{pro-vi}, with $F$ and constraint functions $g_i$ ($i=1,\ldots, 2d+4$) defined as follows:
\begin{equation*}
F(x) = \begin{bmatrix}
2\beta I_d & (1-\beta)I_d\\
-(1-\beta)I_d & 2\beta I_d
\end{bmatrix}x, \quad  g_i(x)=-\mathbf{e}_{i}^{\top}x \leq 0, \;\forall  i\in [2d],
\end{equation*}
along with $g_{2d+1}(x)=\mathbf{1}_{d}^\top x_1  - 1 \leq 0$, $g_{2d+2}(x) = 1-\mathbf{1}_{d}^\top x_1  \leq 0$, 
$g_{2d+3}(x)=\mathbf{1}_{d}^\top x_2  - 1 \leq 0$ and $g_{2d+4}(x)=1-\mathbf{1}_{d}^\top x_2 \leq 0$.
In this problem, $F$ satisfies \eqref{relax-con} with $\ell_F=\sqrt{5\beta^2-2\beta+1}$ and $B=0$. Furthermore, $F$ is $\mu$-strong monotone with $\mu=2\beta$. 
The initial point $x^0$ is constructed as follows:  first generate a $2d$-dimensional vector sampled from  $\mathcal U(0,1)$, then normalize its first $d$ components and last $d$ components to sum to $1$ respectively—this ensures both $x_{1}^0$ and $x_2^0$ lie in the probability simplex. 
In this case, $x^*$ is given by $\frac{1}{d}[\mathbf{1}_d^{\top}, \mathbf{1}_d^{\top}]^{\top}$, and we set $d=500$ in our experiment.


{First, we validate Algorithm \ref{alg:cgm-op} with $\alpha=\mu$ and $\eta_t = 1/{\mu(t+8\kappa^2)}$ for $t\geq 0$.} 
{For completeness, the associated Minty gap function $\mathcal{G}(x^t) =
\max_{x\in\mathcal C} \left\langle F(x),x^t-x\right\rangle$ can be evaluated as follows:
\[
\begin{aligned}
\mathcal{G}(x^t)
=
\left\langle 2\beta x_1^t-(1-\beta)x_2^t, \, \widehat{x}_1^{\,t}\right\rangle
-2\beta\left\|\widehat{x}_1^{\,t}\right\|^2
+\left\langle (1-\beta)x_1^t+2\beta x_2^t, \, \widehat{x}_2^{\,t}\right\rangle
-2\beta\left\|\widehat{x}_2^{\,t}\right\|^2,
\end{aligned}
\]
where $x^t=\left[(x_1^t)^\top,(x_2^t)^\top\right]^\top$, $\widehat{x}_1^{\,t}=
\text{Proj}_{\Delta_d}\left({x_1^t/2\!-\!(1\!-\!\beta)x_2^t/4\beta}\right)$, $\widehat{x}_2^{\,t}=
\text{Proj}_{\Delta_d}\left({(1-\beta)x_1^t/4\beta\!+\!x_2^t/2}\right)$, and $\Delta_d
=\left\{x\in\mathbb{R}^d:x\geq0,\ \mathbf{1}^\top x=1\right\}$.
}

{
The evolution of $\mathcal{G}(x^t)$ and $\max_{i\in [m]}[g_i(x^t)]_+$ over iterations is presented in Figure \ref{op-validity}. These numerical results align with our theoretical findings in Theorem \ref{thom:op}.
}

\begin{figure}[!htbp]
\centering
\subfloat
{\includegraphics[scale=0.4]{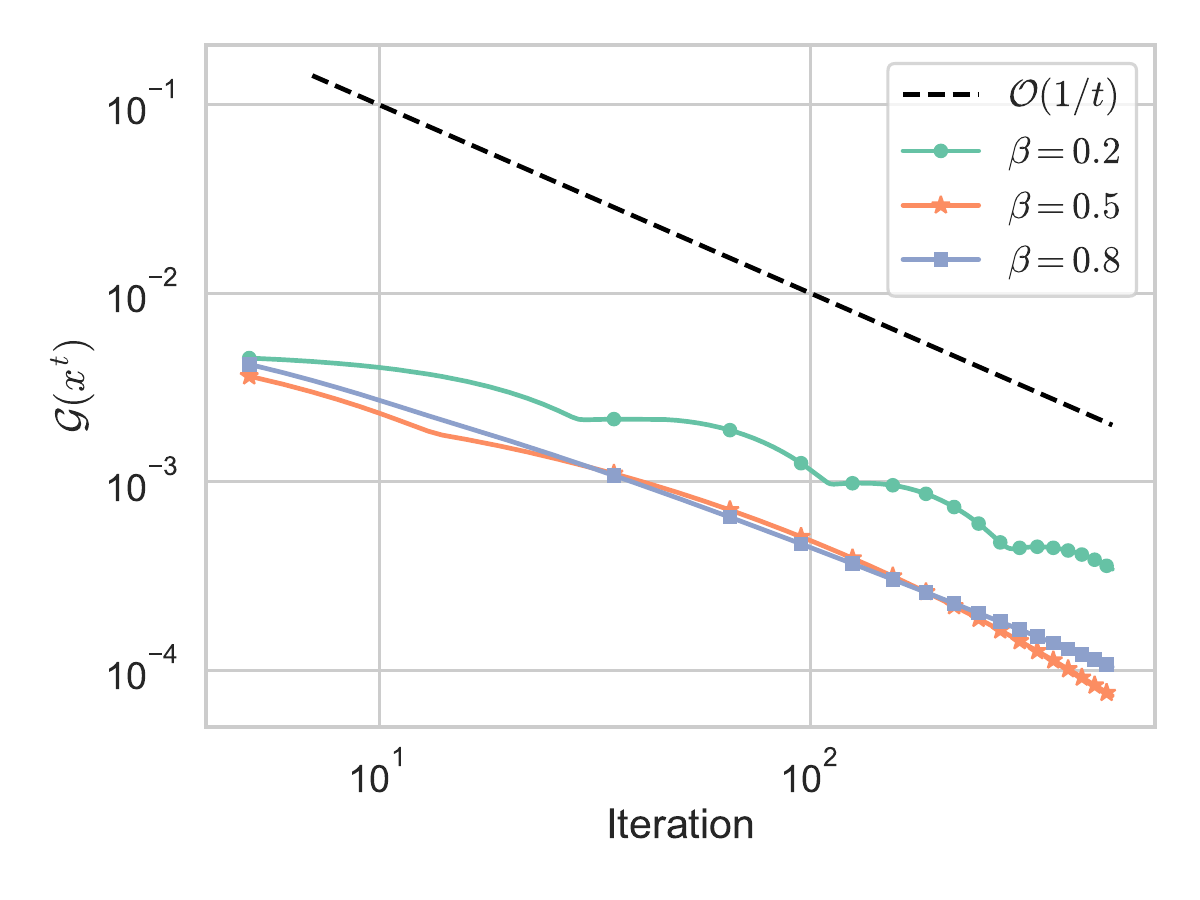}}
\subfloat
{\includegraphics[scale = 0.4]{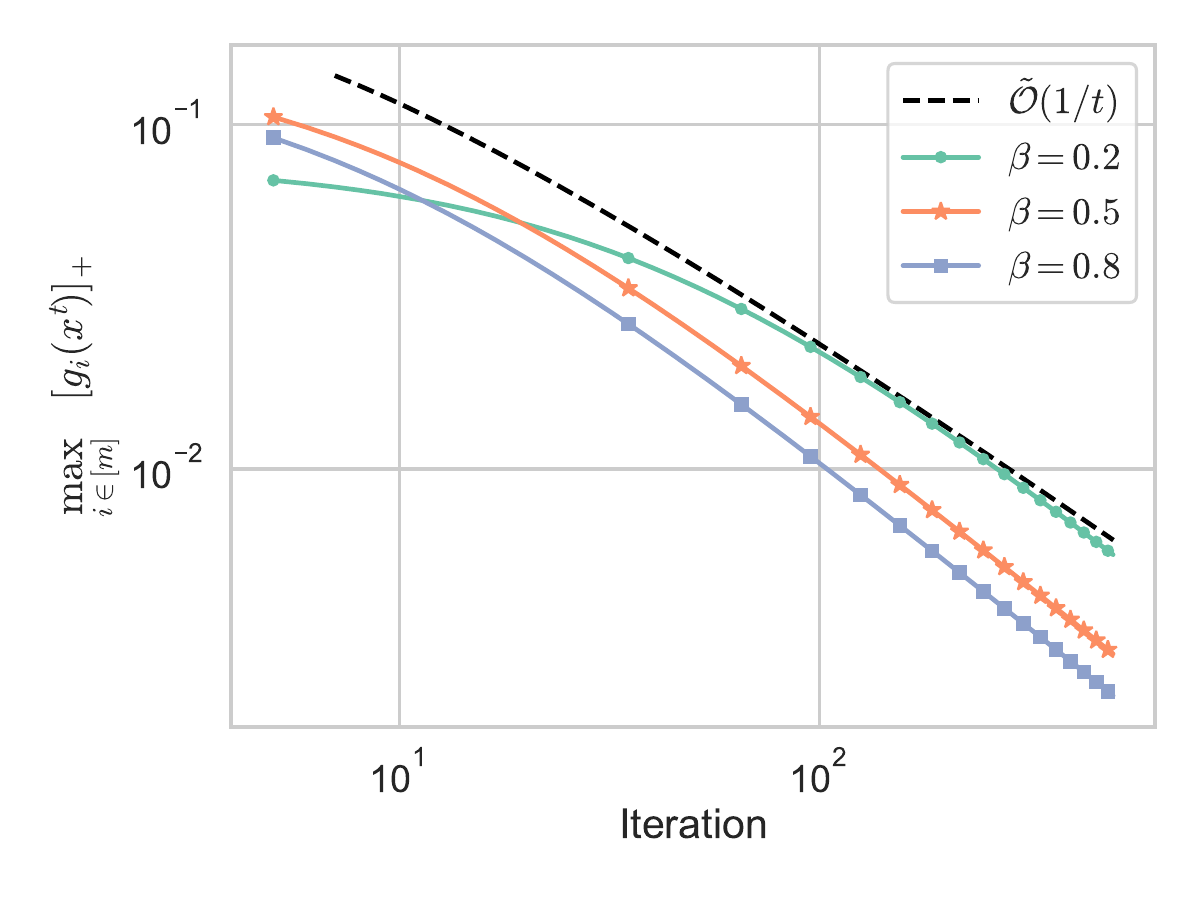}}
\caption{Convergence results of the CGM-VI for the HBG problem with different $\beta$ values. Left: Minty gap residual; Right: constraint violation with $m=2d+4$.}
\label{op-validity}
\end{figure}

\subsection{Risk‑Constrained Non‑Potential Equilibrium Problem}
{
In this section, we compare CGM‑VI, with the same parameter settings as in the previous section, with other projection‑based methods under a fixed CPU time budget.
We consider a shared-constraint generalized Nash game whose variational equilibrium is characterized by a variational inequality \cite{FF07}. Within this framework, we construct a strongly monotone affine instance involving $d$ agents to demonstrate that CGM‑VI can handle an infeasible initial point $x^0$.
In this game, each agent solves a private cost‑minimization problem. 
Agent $j$ selects a signed adjustment $x_j\in\mathbb{R}$ and has the cost
\[
J_j(x_j,x_{-j})
=
\frac{1}{2}A_{jj}(x_j-x_j^*)^2
+
(x_j-x_j^*)\sum\nolimits_{k\neq j}A_{jk}(x_k-x_k^*),
\]
where $x^*=(x_1^*,\ldots,x_d^*)^\top\in\mathbb R^d$ denotes a prescribed target adjustment profile and is constructed as the unique solution of the VI problem. The matrix $A=(A_{jk})\in\mathbb R^{d\times d}$ is the interaction matrix: $A_{jj}>0$ measures agent $j$'s own marginal adjustment cost, and $A_{jk}$ quantifies the effect of agent $k$'s adjustment on agent $j$'s marginal cost.
Stacking the agents’ marginal costs yields the game mapping
\[
F(x)
=
\begin{bmatrix}
\frac{\partial J_1}{\partial x_1}(x), \ldots, \frac{\partial J_d}{\partial x_d}(x)
\end{bmatrix}^\top 
=
A(x-x^*),
\qquad x\in\mathbb{R}^d.
\]
All agents face the common, scenario‑dependent constraint set
\[
\mathcal{C}
=
\left\{
x\in\mathbb{R}^d:
\|U_i x\|^2\leq r_i^2,\ i=1,\ldots,1000,
\ \|x\|^2\leq B_{\text{max}}^2
\right\},
\]
in which each $U_i\in\mathbb{R}^{6\times d}$ maps decisions to the six exposure factors in scenario $i$, $B_{\text{max}}$ specifies the maximum total adjustment, the first $1000$ constraints bound scenario‑specific exposures, and the last constraint enforces the overall total‑adjustment budget.}

{
We set $A=Q\operatorname{diag}\big(1+ 0.45\times [0, 1, \ldots, d-1] /(d-1)\big)Q^\top+0.30{(G-G^\top)}/{\|G-G^\top\|_2}$, $d=40$, and $B_{\text{max}}=1.15$, where the matrices $Z, G\in\mathbb{R}^{d\times d}$ are generated independently, with entries drawn from the standard normal distribution, and $Q$ is the orthogonal matrix obtained from the QR decomposition of $Z$.
By this construction, the operator $F$ is strongly monotone with $\mu=1$, and satisfies \eqref{relax-con} with $\ell_F=\|A\|$ and $B=0$.
We generated the entries of $x^*\in\mathbb{R}^d$ and $U_i\in\mathbb{R}^{6\times d}$ according to the standard Gaussian distribution and rescaled them so that $\|x^*\|=0.45$ and $\|U_i\|_{\mathrm F}=1$. Furthermore, for $i=1,\ldots,1000$, we drew $\xi_i\sim\operatorname{Unif}[1.2,1.8]$ (the uniform distribution over $[1.2,1.8]$) independently and set $r_i^2=\max\left\{\xi_i^2\|U_i x^*\|^2,\,2\times10^{-4}\right\}$  so that $x^*\in\operatorname{int}\mathcal{C}$. 
To generate the infeasible initial point $x^0$, we first sampled $128$ random unit vectors $u$, then formed
$x(u)=1.25\,\rho(u)u$ with 
$\rho(u)
=
\min\big\{
B_{\text{max}},\,
\min_{1\leq i\leq1000}
r_i/\|U_i u\|
\big\}$, 
and finally chose $x^0$ as the one yielding the fewest number of constraint violations, 
breaking ties by the smallest $\|x(u)-x^*\|$. 
}

{
The compared methods are the projected gradient (PG) method: $x^{t+1}=\operatorname{Proj}_{\mathcal{C}}(x^{t}-\eta F(x^t))$ with the standard contraction‑based step size $\eta=\mu/\ell_F^2$, and
the projected extra‑gradient (EG) method \cite{K76}: $x^{t+1}=\operatorname{Proj}_{\mathcal{C}}\left(x^{t}-\eta F\left(\operatorname{Proj}_{\mathcal{C}}\left(x^{t}-\eta F\left(x^{t}\right)\right)\right)\right)$ with the step size $\eta = 1/\ell_F$.
Here, $\operatorname{Proj}_{\mathcal{C}}(\cdot)$ denotes the Euclidean projection onto $\mathcal{C}$. For a given input $y\in\mathbb{R}^d$, we compute this projection by solving $\min_{z\in\mathcal{C}}\|z-y\|^2$, modeled using \texttt{cvxpy} in Python and solved via the \texttt{CLARABEL} solver.
}

\begin{figure}[!htbp]
\centering
\subfloat
{\includegraphics[scale=0.4]{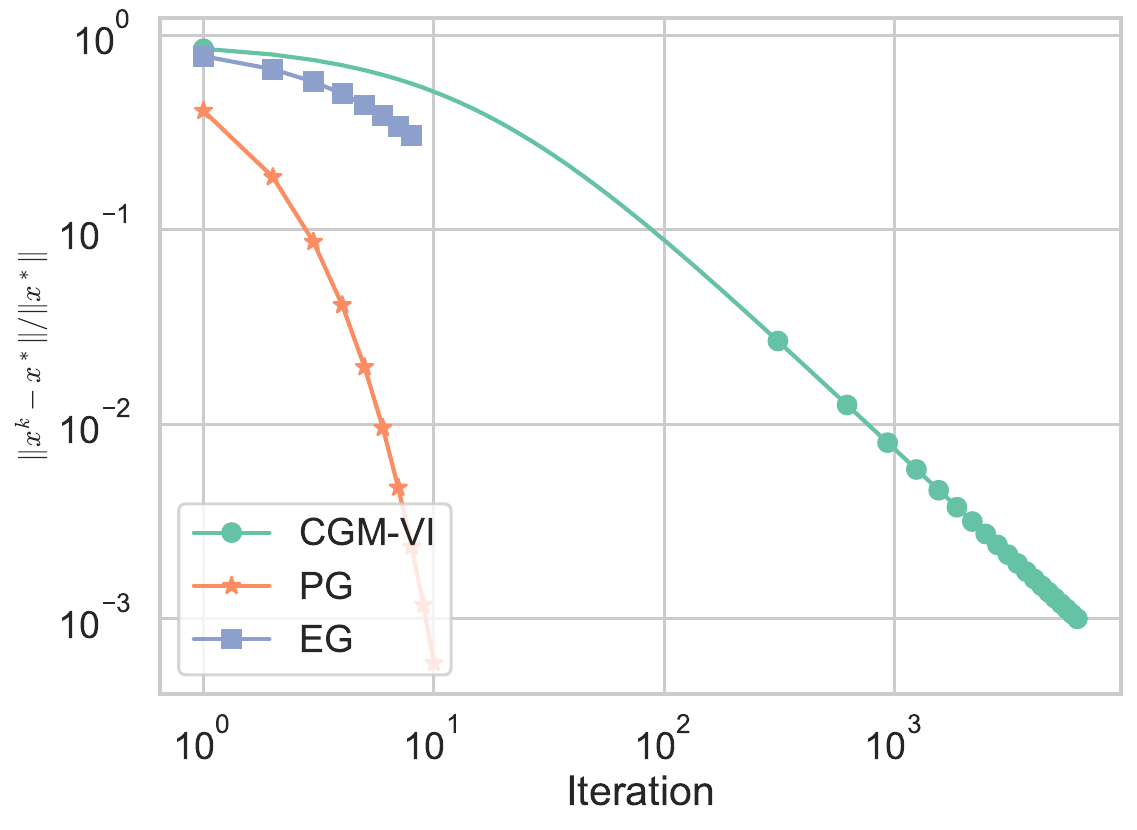}}
\subfloat
{\includegraphics[scale = 0.4]{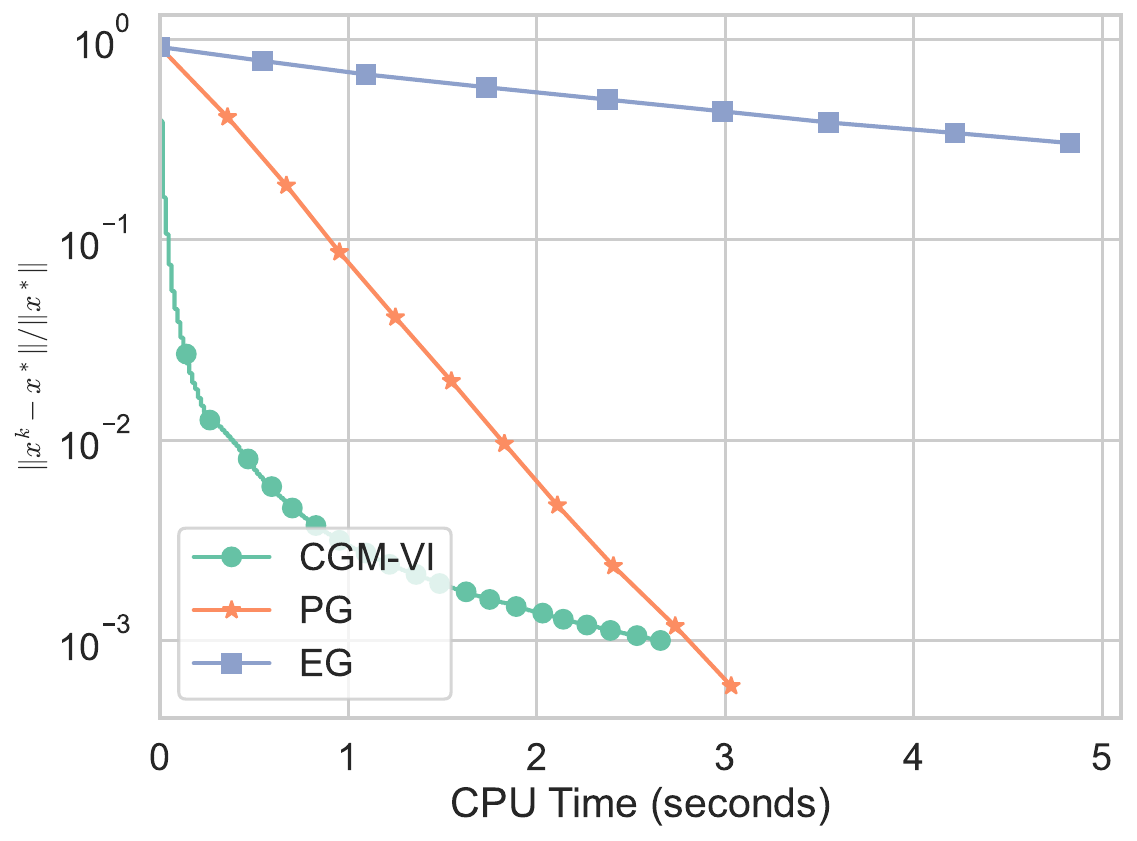}}
\caption{{Comparison results for CGM‑VI versus PG and EG methods. Left: relative error versus iteration count; Right: relative error versus CPU time. Each method is terminated either when the relative error falls below $10^{-3}$ or when the $5$-second CPU‑time limit is reached.}}
\label{op-comparsion}
\end{figure}

{
The results in Figure~\ref{op-comparsion} indicate that CGM-VI converges more slowly than PG in terms of iteration count,
yet outperforms both PG and EG in terms of CPU time for obtaining moderately‑accurate solutions, as measured by the relative error $\|x^t-x^*\|/\|x^*\|$. 
This is because the projection onto $\mathcal{C}$ requires approximately three orders of magnitude more CPU time than the projection onto $V_{\alpha}(x)$ in the CGM-VI subproblem.
The experiment therefore illustrates both the convergence guarantee and the computational benefit of allowing the CGM‑VI to start directly from an infeasible initial point $x^0$ when projection onto the feasible set $\mathcal{C}$ is costly.
When higher accuracy is desired, CGM-VI also falls behind PG in this experiment in terms of CPU time. 
}

\subsection{Testing CGM-MVI with an Unbounded Feasible Set}
{
Consider a one-dimensional VI problem defined by
\[
F(x)=\tanh(x-1) \text{~~and~~}
\mathcal{C}
=
[0,\infty)
=
\left\{
x\in\mathbb{R}:
g_1(x):=\log(1+\exp(-x))-\log 2\leq0
\right\}.
\]
Note that $F$ is monotone,  $1$-Lipschitz continuous, and globally bounded by $L_F=1$. 
The unique solution to this VI problem is $x^*=1$. 
We initialize Algorithm \ref{alg:cgm-op-expanding-ball} at $x^0=-1/2\notin\mathcal{C}$, set $D_0=1$, and run CGM-MVI 
three times, with the prescribed number of iterations $T\in\{100,1000,10000\}$, respectively. 
Since the convergence guarantee under general monotonicity concerns the averaged output, we report the numerical behavior for $\bar{x}^t=(\sum_{s=0}^{t-1}x^s)/t$ for $1\leq t \leq T$ rather than the last iterate $x^T$. 
The restricted Minty gap  is given by $\mathcal{G}_{D_T}(\bar{x}^t)=\max_{0\leq x\leq D_T} \{\phi(x):=\tanh(x-1)(\bar{x}^t-x)\}$. We evaluate this gap via computing the unique maximizer of $\phi$ over $[0, D_T]$, which satisfies the stationarity condition $\bar{x}^t-x-\sinh\bigl(2(x-1)\bigr)/2=0$ and can be obtained via the bisection method.
}

{The numerical results for the restricted Minty gap and constraint violations are presented in 
Figure \ref{mono-op-validity}. These results show a clear decrease in both quantities as the iteration horizon increases, providing numerical evidence on the validity of our results in Theorem \ref{thom:monotone-vi-expanding-ball}.}

\begin{figure}[!htbp]
\centering
\subfloat
{\includegraphics[scale=0.4]{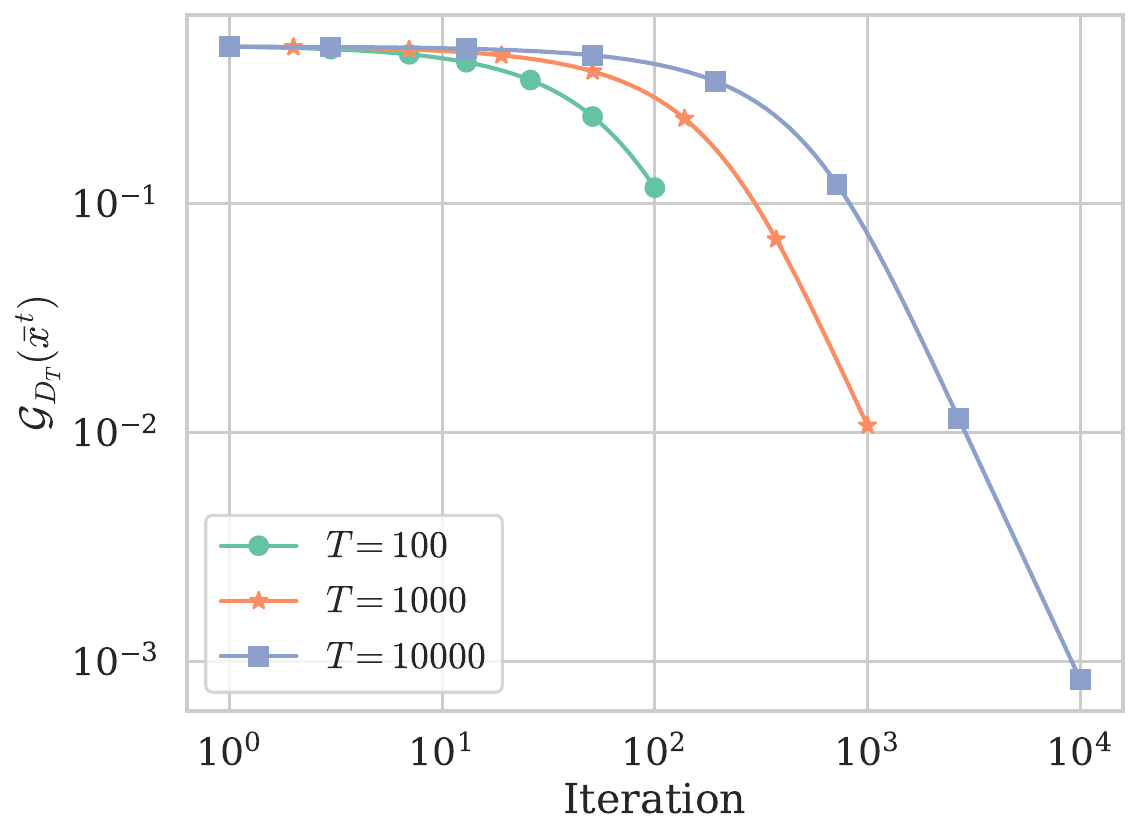}}
\subfloat
{\includegraphics[scale = 0.4]{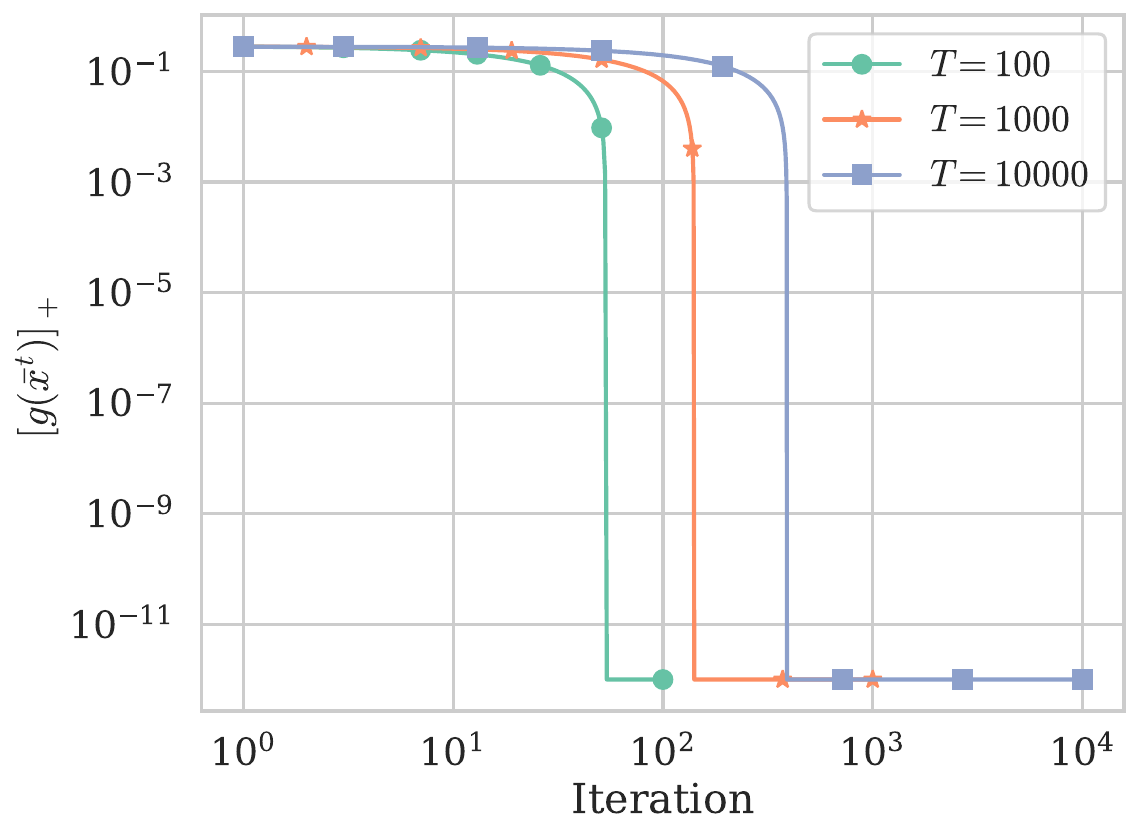}}
\caption{{Numerical results of CGM-MVI with varying iteration budgets $T\in\{100,1000,10000\}$. Left: restricted Minty gap $\mathcal{G}_{D_T}(\bar{x}^t)$ for $1\leq t \leq T$. Right: constraint violations $[g(\bar{x}^t)]_+$ for $1\leq t \leq T$.}}
\label{mono-op-validity}
\end{figure}

\section{Conclusions} \label{sec:conclusion}
In this paper, we investigate the convergence of CGM for problems with functional constraints: strongly convex optimization, strongly monotone VIs, {and general monotone VIs}. 
We showed that several assumptions adopted in existing convergence analyses of CGM are highly restrictive and that some are even mutually incompatible in the strongly convex optimization and strongly monotone VI settings. 
To address these issues, we established new convergence-rate guarantees for CGM with appropriately chosen step sizes under weaker and more reasonable assumptions. 
{For strongly monotone VIs, we further established pointwise convergence. 
For general monotone VIs, we introduced a horizon-dependent expanding-ball variant of CGM and established ergodic convergence guarantees over potentially unbounded feasible sets. Across all three settings, our results allow the initial point $x^0$ to be infeasible and do not require the feasible set $\mathcal{C}$ to be bounded. 
} 
Finally, the expanded numerical experiments corroborate the theoretical findings and illustrate the performance of CGM in these
settings.

\begin{appendix} 
\section{Methods related to CGM for solving constrained minimization problems}
\phantomsection
\label{appendix:a}
For the constrained minimization problem \eqref{pro:min_f(x)}, the algorithm closest to CGM is probably the sequential quadratic programming method (SQP) \cite{P78, P78b, Matlab25}. 
Recall that one key feature of CGM is its reduction of a complex problem to a sequence of simpler quadratic programming (QP) subproblems with linearized constraints. 
Likewise, SQP bases its iterations on lightweight QP subproblems.
With this in mind, we will first focus on highlighting the differences between CGM and SQP.

Under the additional assumption that $f$ and $g$ are twice continuously differentiable, to solve problem \eqref{pro:min_f(x)}, at iterate $t$, SQP requires solving the following QP subproblem to obtain the descent direction:
\begin{equation} \label{def:sqp}
    \begin{aligned} \min_{v\in\mathbb{R}^n} \quad & \nabla f(x^t)^{\top} v+\frac{1}{2} v^{\top} B^t v \\ \text {s.t.} \quad & g_i\!\left(x^t\right)+\nabla g_i\!\left(x^t\right)\!^{\top} v \leq 0, \quad i\in[m],\end{aligned}
\end{equation}
where $B^t\in \mathbb{R}^{n\times n}$ is typically required to be positive definite, and it serves as an approximation of the Hessian of the Lagrangian function $$\mathcal{L}(x,\lambda)=f(x)+\lambda^{\top} g(x), ~\text{with} ~ g(x)=(g_1(x),g_2(x),\ldots,g_m(x)),\; \lambda \in\mathbb{R}^{m}_+$$ evaluated at the current iterate $(x^t,\lambda^t)$. The iterate is updated by $x^{t+1} = x^t + \alpha_t v^t$, where $v^t$ is the optimal solution of \eqref{def:sqp}, and $\alpha_t$ is used determined by line search. The matrix $B^t$ in \eqref{def:sqp} is often updated by quasi-Newton methods. SQP is effective for solving
medium- and small-sized nonlinear programs, but its scalability is limited by the high memory cost of QP subproblems in large-scale settings. This limitation has, in turn, motivated the development of various SQP methods incorporating the active set technique. Both Schittkowski \cite{S92, S09} and Liu \cite{L05} constructed (working) active sets to reduce the size of the QP subproblems in SQP methods by utilizing information from local constraint violations and multipliers. Specifically, Schittkowski's method presupposes that at most $m_w$ constraints are active. Using a fixed user-defined tolerance $\varepsilon$, it defines indices 
\[
J_t^* = \{i \mid g_i(x^t) \geq -\varepsilon \text{ or } \lambda_i^t > 0\} \quad \text{and} \quad K_t^* = [m] \setminus J_t^*.
\]
A heuristic then selects a subset $\bar{K}_t^* \subset K_t^*$ such that the working active set $W_t = J_t^* \cup \bar{K}_t^*$ has exactly $m_w$ elements.
Liu's method, on the other hand, requires extra theoretical conditions, including a strengthened MFCQ and uniformly bounded $B^t$. Its active set, $I(x^t, \lambda^t, \varepsilon_t) = \{i \mid g_i(x^t) \geq -(\varepsilon_t + \lambda_i^t)\}$,
relies on a varying tolerance parameter $\varepsilon_t$ that must converge to zero to ensure the algorithm's convergence.
In contrast, the QP subproblem in CGM is much simpler, as it does not require estimating second-order information of the Lagrangian function and typically employs a simple active-set technique that relies solely on local constraint violations ${I_{x}^{\textbf{act}}}:= \{ i \in [m] \mid g_i(x^t) \geq 0 \}$. 
This feature makes the QP subproblem in CGM easier to solve. 
Moreover, the step size strategy in CGM can adopt either a constant or a variable step size that satisfies certain conditions, without requiring a complex line search.

Next, 
consider that CGM computes a descent direction by projecting the negative gradient onto the linearized, sparse active constraints in \eqref{update-v-f}, and this projection-based approach for generating descent directions shares conceptual similarities with the Riemannian gradient descent method (RGD) \cite{PAR08}. We then discuss the similarities and differences between RGD and CGM. In RGD, the descent direction is obtained by projecting the negative Euclidean gradient onto the tangent space of the constraint manifold—a local linear approximation of the feasible set. This yields the Riemannian gradient $\mathrm{grad} f(x)$, specifically, for an equality-constrained manifold $\mathcal{M} = \{x \in \mathbb{R}^n : g_i(x) = 0, \, \forall  i \in [m]\}$ with regular constraints, i.e., the gradient vectors $\{\nabla g_1(x), \dots, \nabla g_m(x)\}$ are linearly independent at $x\in \mathcal{M}$, the tangent space at $x$ is given by $T_x\mathcal{M} = \{v \in \mathbb{R}^n : \nabla g_i(x)^\top v = 0, \, \forall  i \in [m]\}$, and the Riemannian gradient admits the variational characterization:
\[
\mathrm{grad} f(x) = \arg \min_{v \in T_{x}\mathcal{M}} \frac{1}{2} \| v + \nabla f(x) \|^2,
\]
which exhibits strong similarity with the update direction in \eqref{update-v-f}, in fact, the two are identical when $x^t \in \mathcal{M}$.
The key difference lies in how the next iterate is produced: RGD requires an additional retraction step to map points from the tangent space back onto the manifold, thereby maintaining feasibility. 
In contrast, CGM does not require $x^{t+1}$ to be feasible.

\section{Recent progress for solving constrained VIs}
\phantomsection
\label{appendix:b}
For solving \eqref{pro-vi}, one can adopt the traditional projected gradient method \cite{FP03} and the Frank-Wolfe method \cite{TJNO20}, which necessitate explicit projection or linear minimization oracles to address the feasibility issue. However, for general constrained sets, 
performing such oracles is prohibitively expensive, so an alternative line of work turns to primal-dual algorithms that exploit (augmented) Lagrangian reformulations \cite{YJC23, CYPJ24, DQM24, BDL23, ZZZ25}. Nevertheless, both these algorithms and their theoretical analysis heavily depend on the existence or the magnitude of the optimal Lagrange multipliers. Specifically, \cite{YJC23} proposed an ADMM-based interior point method, which was further improved by \cite{CYPJ24} through a warm-starting technique. 
\cite{DQM24} extended the constraint extrapolation method \cite{BDL23} for constrained minimization problems to constrained VIs, relying on bounded optimal Lagrange multipliers for step size selection and convergence rates, and requiring one projection per iteration onto a ``simple'' convex compact set to ensure the boundedness of the iterates theoretically. \cite{ZZZ25} developed a primal-dual system and designed the augmented Lagrangian-based ALAVI method. This algorithm ensures global convergence under a condition weaker than monotonicity—termed primal-dual variational coherence, but it requires two projections each iteration onto the dual cone $\mathcal{C}^*$. More recently, \cite{AA25} proposed a mirror descent type method for solving \eqref{pro-vi}. This method dynamically switches between objective-descent and constraint-descent steps, depending on whether the functional constraints satisfy $\max_{i\in [m]} g_i(x)\leq \varepsilon$ at each iteration. With time-varying step sizes, and assuming $F$ is bounded and monotone, the method achieves a $\mathcal{O}(1/\sqrt{T})$ convergence rate.
\end{appendix}
\printbibliography
\end{document}